\documentclass[10pt,reqno]{amsart}

\usepackage[dvipsnames]{xcolor}
\usepackage{amsmath,amssymb,amsthm,amsfonts,enumerate,tikz,bm,tikz-cd}
\usepackage{mathrsfs}
\usetikzlibrary{matrix,arrows,positioning,calc}
\usepackage[all]{xy}
\usepackage[bookmarksnumbered,colorlinks]{hyperref}
\usepackage{url}
\numberwithin{equation}{section}
\usepackage{graphicx}
\usepackage[T1]{fontenc}
\usepackage{textcomp}
\usepackage{palatino,helvet}
\usepackage{scrextend}
\deffootnote{2em}{0em}{\thefootnotemark\quad}

\theoremstyle{definition}
\newtheorem{definition}{Definition}[section]
\newtheorem{remark}[definition]{Remark}
\newtheorem{example}[definition]{Example}

\theoremstyle{plain}
\newtheorem{theorem}[definition]{Theorem}
\newtheorem{proposition}[definition]{Proposition}
\newtheorem{lemma}[definition]{Lemma}
\newtheorem{corollary}[definition]{Corollary}

\theoremstyle{remark}

\usepackage{paralist}

\newcommand{\mcA}{\mathcal{A}}
\newcommand{\mcB}{\mathcal{B}}
\newcommand{\mcC}{\mathcal{C}}

\newcommand{\mcF}{\mathcal{F}}
\newcommand{\mcG}{\mathcal{G}}

\newcommand{\mcI}{\mathcal{I}}

\newcommand{\mcM}{\mathcal{M}}
\newcommand{\mcO}{\mathcal{O}}
\newcommand{\mcP}{\mathcal{P}}
\newcommand{\mcT}{\mathcal{T}}

\newcommand{\mcW}{\mathcal{W}}
\newcommand{\mcX}{\mathcal{X}}
\newcommand{\mcY}{\mathcal{Y}}
\newcommand{\mcZ}{\mathcal{Z}}
\newcommand{\mcGP}{\mathcal{GP}}
\newcommand{\mcGI}{\mathcal{GI}}
\newcommand{\mcWGP}{\mathcal{WGP}}
\newcommand{\mcWGI}{\mathcal{WGI}}

\newcommand{\mcDP}{\mathcal{DP}}
\newcommand{\mfC}{\mathfrak{C}}
\newcommand{\mfF}{\mathfrak{F}}

\newcommand{\msF}{\mathscr{F}}
\newcommand{\msI}{\mathscr{I}}
\newcommand{\msM}{\mathscr{M}}
\newcommand{\mfGF}{\mathfrak{GF}}

\newcommand{\mbZ}{\mathbb{Z}}

\newcommand{\Hom}{{\rm Hom}}
\newcommand{\Ext}{{\rm Ext}}
\newcommand{\Tor}{{\rm Tor}}
\newcommand{\Ch}{\mathsf{Ch}}

\newcommand{\Ho}{{\rm Ho}}
\newcommand{\pd}{{\rm pd}}
\newcommand{\id}{{\rm id}}
\newcommand{\Gpd}{{\rm Gpd}}
\newcommand{\Gid}{{\rm Gid}}
\newcommand{\FPD}{{\rm FPD}}
\newcommand{\FID}{{\rm FID}}
\newcommand{\FGPD}{{\rm FGPD}}
\newcommand{\FGID}{{\rm FGID}}
\newcommand{\glGPD}{{\rm gl.GPD}}
\newcommand{\glGID}{{\rm gl.GID}}

\newcommand{\resdim}{{\rm resdim}}
\newcommand{\coresdim}{{\rm coresdim}}
\newcommand{\Mod}{\mathsf{Mod}}
\newcommand{\modu}{\mathsf{mod}}

\newcommand{\add}{{\rm add}}
\newcommand{\Ker}{{\rm Ker}}
\newcommand{\Ima}{{\rm Im}}
\newcommand{\CoKer}{{\rm CoKer}}

\newcommand{\Qcoh}{\mathfrak{Qcoh}}

\newcommand{\lfpone}{$\mathsf{(lfp1)}$}
\newcommand{\lfptwo}{$\mathsf{(lfp2)}$}
\newcommand{\lfpthree}{$\mathsf{(lfp3)}$}

\newcommand{\GPaone}{$\mathsf{(gpa1)}$}
\newcommand{\GPatwo}{$\mathsf{(gpa2)}$}
\newcommand{\GPathree}{$\mathsf{(gpa3)}$}
\newcommand{\GPafour}{$\mathsf{(gpa4)}$}
\newcommand{\GPafive}{$\mathsf{(gpa5)}$}

\newcommand{\Gone}{$\mathsf{(g1)}$}
\newcommand{\Gtwo}{$\mathsf{(g2)}$}
\newcommand{\Gthree}{$\mathsf{(g3)}$}

\newcommand{\gaone}{$\mathsf{(ga1)}$}
\newcommand{\gatwo}{$\mathsf{(ga2)}$}
\newcommand{\gathree}{$\mathsf{(ga3)}$}
\newcommand{\gafour}{$\mathsf{(ga4)}$}

\newcommand{\tiltingone}{$\mathsf{(t1)}$}
\newcommand{\tiltingtwo}{$\mathsf{(t2)}$}
\newcommand{\tiltzero}{$\mathsf{(tilt0)}$}
\newcommand{\tiltone}{$\mathsf{(tilt1)}$}
\newcommand{\tilttwo}{$\mathsf{(tilt2)}$}
\newcommand{\tiltthree}{$\mathsf{(tilt3)}$}
\newcommand{\tiltfour}{$\mathsf{(tilt4)}$}
\newcommand{\tiltfive}{$\mathsf{(tilt5)}$}

\begin{document}

\title{Gorenstein categories relative to G-admissible triples}

\author{Sergio Estrada}
\address[S. Estrada]{Departamento de Matem\'aticas. Universidad de Murcia. 30100 Murcia. ESPA\~{N}A}
\email{sestrada@um.es}

\author{Octavio Mendoza}
\address[O. Mendoza]{Instituto de Matem\'aticas. Universidad Nacional Aut\'onoma de M\'exico. Mexico City 04510. MEXICO}
\email{omendoza@matem.unam.mx}

\author{Marco A. P\'erez}
\address[M. A. P\'erez]{Instituto de Matem\'atica y Estad\'istica ``Prof. Ing. Rafael Laguardia''. Facultad de Ingenier\'ia. Universidad de la Rep\'ublica. Montevideo 11300. URUGUAY}
\email{mperez@fing.edu.uy}

\maketitle

\begin{abstract}
We present the notion of Gorenstein categories relative to G-admissi-ble triples. This is a relativization of the concept of Gorenstein category (an abelian category with enough projective and injective objects, in which the suprema of the sets $\{ \pd(I) \ \text{:} \ I \text{ is injective} \}$ and $\{ \id(P) \ \text{:} \ P \text{ is projective} \}$ are finite). Such categories turn out to be a suitable setting on which it is possible to obtain hereditary abelian model structures where the (co)fibrant objects are Gorenstein injective (resp., Gorenstein projective) objects relative to GI-admissible (resp., GP-admissible) pairs. Applications and examples of these structures are given. Moreover, we link relative Gorenstein categories with tilting theory and obtain relations between different relative homological dimensions.
\end{abstract}


\pagestyle{myheadings}
\markboth{\rightline {\scriptsize S. Estrada, O. Mendoza and M. A. P\'{e}rez}}
         {\leftline{\scriptsize Gorenstein categories relative to G-admissible triples}}


\section*{Introduction}

In the present article, we study Gorenstein objects relative to GP-admissible and GI-admissible pairs, having in mind the goal of finding a suitable notion of relative Gorenstein categories within which it is possible to do homotopy theory via the construction of model structures.

As a generalization of Iwanaga-Gorenstein rings, Beligiannis and Reiten proposed in \cite{BR} the concept of Gorenstein categories. These are abelian categories with enough projective and injective objects, in which $\sup\{ \pd(I) \ \text{:} \ I \text{ is injective} \}$ and $\sup\{ \id(P) \ \text{:} \ P \text{ is projective} \}$ are both finite. A slightly different concept of Gorenstein category was proposed independently by Enochs, Garc\'ia-Rozas and the first named author in \cite{Tate}. These are Grothendieck categories with finite finitistic projective and injective dimensions, equipped with a generator of finite projective dimension, and in which an object has finite projective dimension if and only if it has finite injective dimension. The latter conditions allow to obtain an abelian model structure where the trivial objects are given by the class of objects with finite projective dimension, and where the fibrant objects are the Gorenstein injective objects. Moreover, in the presence of enough projective objects, one can also obtain the dual model structure, with the same trivial objects, and whose cofibrant objects are precisely the Gorenstein projective objects. So Gorenstein categories turn out to be a nice general setting in which it is possible to carry over the homotopy theory made up of Gorenstein projective and Gorenstein injective modules over Iwanaga-Gorenstein rings, described in Hovey's \cite[\S \ 8]{Hovey}.

We aim to show that the theory of Gorenstein projective and Gorenstein injective objects in an abelian category $\mcC$ relative to GP-admissible and GI-admissible pairs (in the sense of Becerril, Santiago and the second named author \cite{BMS}), also has appealing homotopical aspects similar to those of Gorenstein projective and Gorenstein injective modules. In order to achieve this, we propose the concept of Gorenstein categories relative to a G-admissible triple. The latter is a combination of a GP-admissible pair $(\mathcal{X,Y})$ and a GI-admissible pair $(\mathcal{Y,Z})$ such that the intersections $\omega := \mcX \cap \mcY$ and $\nu := \mcY \cap \mcZ$ are ``well behaved'' with respect to the $(\mcX,\mcY)$-Gorenstein projective and $(\mcY,\mcZ)$-Gorenstein injective objects. More specifically, in order to deal with the possible absence of enough projective and injective objects, we shall need on the one hand that every $(\mathcal{Y,Z})$-Gorenstein injective (resp., $(\mathcal{X,Y})$-Gorenstein projective) object admits a ``proper'' (co)resolution by objects in $\omega$ (resp., in $\nu$), and on the other hand that the objects in $\omega$ and the $(\mathcal{Y,Z})$-Gorenstein injective objects (resp., the objects in $\nu$ and the $(\mathcal{X,Y})$-Gorenstein projective objects) are $\Ext$-orthogonal to each other. Moreover, if in addition $\id(\omega) < \infty$, $\pd(\nu) < \infty$ and $\mcC$ has finite global $(\mathcal{X,Y})$-Gorenstein projective and $(\mathcal{Y,Z})$-Gorenstein injective dimensions (in this case we say that $\mcC$ is a \emph{strongly} $(\mathcal{X,Y,Z})$-Gorenstein category), we are in condition to prove our main result: \\

\noindent{{\bf Theorem.}} If $\mcC$ is a strongly $(\mathcal{X,Y,Z})$-Gorenstein category in which every object in $\nu$ has finite injective dimension relative to $\omega$, then there exists a unique hereditary abelian model structure on $\mcC$, where the trivial objects are formed by the class of objects with finite injective dimension, and the (total) left and right Ext-orthogonal complements of $\mcY$ and $\omega$, denoted ${}^{\perp}\mathcal{Y}$ and $\omega^\perp$, respectively, are the classes of cofibrant and fibrant objects. \\

We shall also give a precise description of the homotopy category of the previous model structure and its dual. Furthermore, besides of these homotopical aspects, Gorenstein categories relative to G-admissible triples are also a fruitful setting for interesting implications derived from relative tilting theory, presented in the works \cite{AM21P1,AM21} by Argud\'in and the second named author.


\subsection*{Organization}

After recalling some preliminary notions from homological and homotopical algebra, in Section \ref{sec:RG-objects} we revisit the notion of $(\mcX,\mcY)$-Gorenstein projective and $(\mcX,\mcY)$-Gorenstein injective objects: a relativization of the concepts of Gorenstein projective and Gorenstein injective objects with respect to two classes $\mcX$ and $\mcY$ of objects in an abelian category $\mcC$. These concepts were originally introduced by Q. Pan and F. Cai \cite{PanCai} in the category of modules over a ring. Later, Xu in \cite{Xu17} focuses on the case where it is given a hereditary complete cotorsion pair $(\mcX,\mcY)$ of modules, and studies the $(\mcX,\mcX \cap \mcY)$-Gorenstein projective modules. We shall work in the context of abelian categories, as several of the results by Pan, Cai and Xu carry over to such categories. Moreover, we present an approach to Gorenstein objects relative to Frobenius pairs and GP-admissible pairs. A consequence of these relativizations is the existence of covers, envelopes and cotorsion pairs associated to Gorenstein objects. We also recall Gorenstein dimensions relative to GP-admissible pairs. Most of the interesting properties of Gorenstein objects and dimensions relative to such pairs are already covered in \cite{BMS}. However, we show some other new in Propositions \ref{coro:properties_relative_Gorenstein_projective}, \ref{prop:GPdim_vs_omegadim} and Corollary \ref{coro:essGIC}. For example, the latter result asserts that  given a GP-admissible pair $(\mcX,\mcY)$ in an abelian category $\mcC$ with finite global $(\mathcal{X,Y})$-Gorenstein projective dimension and such that $\mcX \cap \mcY$ is closed under direct summands, then the smallest thick subcategory containing $\mcX \cap \mcY$ is the class of objects with finite injective dimension.

In Section \ref{sec:RG-models} we give a first approach to model category structures relative to Gorenstein objects. This section is motivated by a result of Xu \cite{Xu17}, namely, if $(\mathcal{X,Y})$ is a hereditary complete cotorsion pair in the category $\Mod(R)$ of left $R$-modules such that every $R$-module has finite $(\mathcal{X},\mcX\cap\mcY)$-Gorenstein projective dimension, then there exists an abelian model structure on $\Mod(R)$ whose cofibrant objects are the $(\mathcal{X},\mcX \cap \mcY)$-Gorenstein projective $R$-modules, $\mcY$ is the class of fibrant objects, and the trivial objects are given by the objects with finite resolution dimension relative to $\mcX$. In Theorem \ref{theo:Xu_model} we give a generalization of this result for abelian categories, which drops the assumption of every object having finite relative Gorenstein projective dimension. Our model structure has the same cofibrant and trivial objects, but it is obtained on the exact category of objects with finite $(\mcX,\mcX \cap \mcY)$-Gorenstein projective dimension. The fibrant objects are the objects in $\mcY$ with finite $(\mcX,\mcX \cap \mcY)$-Gorenstein projective dimension. We also provide in Proposition \ref{prop:homotopy_category} a description of the homotopy category of this model, which will be denoted by $\mcM^{\rm GP}_{(\mcX,\mcY)}$. Specifically, it is equivalent to the stable category of $(\mcX,\mcX \cap \mcY)$-Gorenstein projective objects which also belong to $\mcY$. One curious aspect about $\mcM^{\rm GP}_{(\mcX,\mcY)}$ is that it is pretty unlikely that it forms a Auslander-Buchweitz model structure (in the sense of Becerril, Santiago and the second and third authors \cite{BMPS19}). Indeed, the only case where this occurs is when $\mcX \cap \mcY$ is the class of projective objects, which is equivalent to saying that the $(\mcX,\mcX \cap \mcY)$-Gorenstein projective and the Gorenstein projective objects coincide (see Proposition \ref{prop:relative_G-projective_model}).

The results in Section \ref{sec:RG-models} suggest to look for other sources of pairs $(\mcX,\mcY)$ and other categorical settings to construct model structures with $(\mathcal{X,Y})$-Gorenstein projective objects as cofibrant objects. Such sources and settings are introduced in Section \ref{sec:RG-categories}. We begin this section recalling the concept of Gorenstein categories in the senses of \cite{BR,Tate}. We show in Proposition \ref{prop:equivalence} that these two definitions are the same for Grothendieck categories with enough projective objects and equipped with a generator of finite projective dimension. The next step will be to study the global and finitistic Gorenstein projective dimensions of an abelian category $\mcC$, relative to a GP-admissible pair $(\mcX,\mcY)$. We show in Proposition \ref{lpdGP} that the finiteness of the global $(\mathcal{X,Y})$-Gorenstein projective dimension of $\mathcal{C}$ is equivalent to the finiteness of the finitistic projective dimension of $\mathcal{C}$ relative to $\mcX \cap \mcY$, provided that every object has finite $(\mcX,\mcY)$-Gorenstein projective dimension. This, along with characterizations for these dimensions, will be key to show in Corollary \ref{glGPD=id} that having finite global $(\mcX,\mcY)$-Gorenstein projective dimension is equivalent to the existence of a hereditary complete cotorsion pair in $\mcC$ of the form $(\mcGP_{(\mcX,\mcY)},(\mcX \cap \mcY)^\wedge)$, where $\mcGP_{(\mcX,\mcY)}$ is the class of $(\mcX,\mcY)$-Gorenstein projective objects, and $(\mcX \cap \mcY)^\wedge$ is the class of objects having a finite resolution relative to $\mcX \cap \mcY$. If in addition $\mcC$ has enough injective objects and $\mcX \cap \mcY$ is closed under direct summands and special precovering, then there exists a unique abelian model structure on $\mcC$ such that $\mcGP_{(\mcX,\mcY)}$ and $(\mcX \cap \mcY)^\wedge$ are the classes of cofibrant and trivial objects, and the right Ext-orthogonal complement of $\mcX \cap \mcY$ is the class of fibrant objects (see Proposition \ref{prop:model_first-approach}). However, the condition that $\mcX \cap \mcY$ is special precovering may be difficult to fulfill (see Remark \ref{rmk:models_first_approach}). One advantage about strongly Gorenstein categories relative to G-admissible triples is that neither this condition nor its dual are required if we aim to obtain Gorenstein model structures, as we can see in the statement of the main theorem, which we prove in Theorem \ref{thm:Gorenstein_models}.

In Section \ref{sec:RG-categories} we also prove several homological properties of (strongly) Gorenstein categories relative to G-admissible triples. Concerning the latter, we propose a way to induce G-admissible triples in Proposition \ref{Gc=Gt}. Later in Theorem \ref{MThmGc} we give a complete characterization of strongly relative Gorenstein categories. One can note that Gorenstein categories in the sense of \cite{BR} are precisely the strongly Gorenstein categories relative to the G-admissible triple $(\mcP,\mcP^{<\infty},\mcI)$, where $\mcP$ and $\mcI$ denote the classes of projective and injective objects, and $\mcP^{<\infty}$ is the class of objects with finite projective dimension (see Proposition \ref{prop:absolute_Gorenstein_category}). New instances of strongly relative Gorenstein categories are obtained in Examples \ref{ex:condition_triples}, \ref{ex:FrobeniusGorenstein}, \ref{ex:Gorenstein_categoriesAC}, \ref{ex:global_dim_rel_Gor}, \ref{ex:induced_on_quivers} and \ref{ex:Gor_st_from_tilting}. Concerning the homotopical aspects of strongly relative Gorenstein categories, we show that the homotopy category of the model structure found in Theorem \ref{thm:Gorenstein_models} is equivalent to the stable category of the $\omega$-Gorenstein objects studied by Christensen, Thompson and the first named author in \cite{CETflat-cotorsion}.

Finally, Section \ref{sec:tilting} is devoted to explore the relation between relative Gorenstein categories and relative tilting theory presented in \cite{AM21P1,AM21,BMS}. This will be done by means of certain G-admissible triples which we call ``tilting''. After proving some outcomes about Gorenstein objects and dimensions relative to tilting G-admissible triples, we give in Corollary \ref{coro2:tilting-in-triples} an application  to Artin algebras. Specifically, we show that if $\Lambda$ is an Artin algebra for which it is possible to find a tilting G-admissible triple $(\mathcal{X,Y,Z})$ of finitely generated $\Lambda$-modules such that the global $(\mathcal{X,Y})$-Gorenstein projective and $(\mathcal{Y,Z})$-Gorenstein injective dimensions of $\Lambda$ are finite (say equal to $m$ and $n$, respectively), then there exist a basic $n$-tilting $\Lambda$-module $T$ and a basic $m$-cotilting $\Lambda$-module $U$ such that $\mcGI_{(\mcY,\mcZ)}$ (resp., $\mcGP_{(\mcX,\mcY)}$) is completely described as the class of $\Lambda$-modules which are injective relative to $T$ (resp., projective relative to $U$).


\subsection*{Notations and conventions}

Throughout, $\mcC$ will denote an abelian category (not necessarily with enough projective and injective objects), and $\Ch(\mcC)$ will stand for the (also abelian) category of chain complexes of objects in $\mcC$. Some of these categories considered in the present article are:
\begin{itemize}
\item The category $\Mod(R)$ of (unitary) left $R$-modules, where $R$ is an associative ring with identity. Unless otherwise specified, $R$-modules are always left.

\item The category $\modu(\Lambda)$ of finitely generated (left) $\Lambda$-modules, where $\Lambda$ is an Artin algebra.

\item The category $\Ch(\Mod(R))$ of chain complexes of $R$-modules (which we denote by $\Ch(R)$ for simplicity).

\item The category $\Qcoh(X)$ of quasi-coherent sheaves over a noetherian and semi-separated scheme $X$.
\end{itemize}

If $\mcX$ is a class of objects in $\mcC$, denoted $\mcX \subseteq \mcC$, we may sometimes regard $\mcX$ as a full subcategory of $\mcC$. We shall let $\mcP$ and $\mcI$ denote the classes of projective and injective objects of $\mcC$, respectively. In the case where $\mathcal{C} = \Mod(R)$, we shall use the notations $\mcP(R)$ and $\mcI(R)$ for the classes of projective and injective $R$-modules. Given two objects $C, D \in \mcC$, we shall denote the set of morphisms $f \colon C \to D$ from $C$ to $D$ by $\Hom(C,D)$. Monomorphisms and epimorphisms in $\mcC$ will be usually denoted by $\rightarrowtail$ and $\twoheadrightarrow$, respectively.


\section{Preliminaries}\label{sec:prelims}

We recall some homological and homotopical preliminary notions and notations in abelian categories (although they will also have similar formulations and meanings for exact categories). In what follows, $\mcX$ and $\mcY$ will denote classes of objects in $\mcC$.


\subsection*{Orthogonal classes}

For any pair of objects $X, Y \in \mcC$ and a positive integer $i > 0$, we can define the group of $i$-fold extensions $\Ext^i(X,Y)$ in the sense of Yoneda (see for instance Sieg's \cite{Sieg} for a detailed treatise on this matter). We set the notation $\Ext^i(\mcX,\mcY) = 0$ whenever $\Ext^i(X,Y) = 0$ for every $X \in \mcX$ and $Y \in \mcY$. In the case where $\mcY = \{ C \}$ for some object $C \in \mcC$, we write $\Ext^i(\mcX,C) = 0$. The notation $\Ext^i(C,\mcY) = 0$ has a similar meaning.

The \emph{$i$-th right} and \emph{total right orthogonal complements} of $\mcX$ are defined as:
\begin{align*}
\mcX^{\perp_i} & := \{ C \in \mcC \text{ : } \Ext^i(\mcX, C) = 0 \} & \text{and} & & \mcX^\perp & := \bigcap_{i > 0} \mcX^{\perp_i}.
\end{align*}
The \emph{$i$-th left} and \emph{total left orthogonal complements} of $\mcX$, ${}^{\perp_i}\mcX$ and ${}^\perp\mcX$, are defined dually.


\subsection*{Chain complexes}

Given a chain complex
\[
X_\bullet = \cdots \to X_{m+1} \xrightarrow{\partial^X_{m+1}} X_m \xrightarrow{\partial^X_m} X_{m-1} \to \cdots,
\]
in $\Ch(\mathcal{C})$, we denote its cycles by $Z_m(X_\bullet) := \Ker(\partial^X_m)$ for every $m \in \mathbb{Z}$. We say that $X_\bullet$ is \emph{$\Hom(-,\mcY)$-acyclic} if the induced complex of abelian groups
\[
\Hom(X_\bullet,Y) = \cdots \to \Hom(X_{m-1},Y) \xrightarrow{\Hom(\partial^X_m,Y)} \Hom(X_m,Y) \to \cdots
\]
is exact for every $Y \in \mcY$. The notion of \emph{$\Hom(\mcY,-)$-acyclic complexes} is dual.


\subsection*{Relative homological dimensions}

Given an object $C \in \mcC$, by an \emph{$\mcX$-resolution of $C$} we mean an exact complex
\[
\cdots \to X_m \to X_{m-1} \to \cdots \to X_1 \to X_0 \twoheadrightarrow C
\]
with $X_k \in \mcX$ for every $k \geq 0$.  If $X_k = 0$ for every $k > m$, we say that the previous resolution has \emph{length} equal to $m$. The \emph{resolution dimension relative to $\mcX$} (or the \emph{$\mcX$-resolution dimension}) of $C$ is defined as the value
\[
\resdim_{\mcX}(C) := \min \{ m \in \mathbb{Z}_{\geq 0} \ \mbox{ : } \ \text{there exists an $\mcX$-resolution of $C$ of length $m$} \}.
\]
If $\{ m \in \mathbb{Z}_{\geq 0} \ \mbox{ : } \ \text{there exists an $\mcX$-resolution of $C$ of length $m$} \}$ is the empty set, then $\resdim_{\mcX}(C) = \infty$.
Moreover,
\[
\resdim_{\mcX}(\mcY) := \sup \{ \resdim_{\mcX}(Y) \ \mbox{ : } \ \text{$Y \in \mcY$} \}
\]
defines the \emph{resolution dimension of $\mcY$ relative to $\mcX$}. The classes of objects with bounded (by some $n \geq 0$) and finite $\mcX$-resolution dimension will be denoted by
\begin{align*}
\mcX^\wedge_n & := \{ C \in \mcC \text{ : } \resdim_{\mcX}(C) \leq n \} & \text{and} & & \mcX^\wedge & := \bigcup_{n \geq 0} \mcX^\wedge_n.
\end{align*}

Dually, we can define \emph{$\mcX$-coresolutions} and the \emph{coresolution dimension of $C$ and $\mcY$ relative to $\mcX$} (denoted $\coresdim_{\mcX}(C)$ and $\coresdim_{\mcX}(\mcY)$). We also have the dual notations $\mcX^\vee_n$ and $\mcX^\vee$ for the classes of objects with bounded and finite $\mcX$-coresolution dimension.

There are also other homological dimensions defined in terms of extension functors. The \emph{projective dimensions of $C$ and $\mcX$ relative to $\mcY$} are defined by
\begin{align*}
\pd_{\mcY}(C) & := \min \{ m \in \mathbb{Z}_{\geq 0} \text{ : } \Ext^{i}(C,\mcY) = 0 \text{ for every $i > m$} \}, \\
\pd_{\mcY}(\mcX) & := \sup \{ \pd_{\mcY}(X) \text{ : } X \in \mcX \}.
\end{align*}
If $\{ m \in \mathbb{Z}_{\geq 0} \text{ : } \Ext^{i}(C,\mcY) = 0 \text{ for every $i > m$} \}$ is the empty set, then $\pd_{\mcY}(C) = \infty$. We shall denote by $\mcP^{<\infty}_{\mcY}$  the class of objects of $\mcC$ with finite projective dimension relative to $\mcY$. The \emph{injective dimensions of $C$ and $\mcX$ relative to $\mcY$}, denoted by $\id_{\mcY}(C)$ and $\id_{\mcY}(\mcX)$, and the class $\mcI^{< \infty}_{\mcY}$, are defined dually. In the case where $\mcY = \mcC$, we write $\pd_{\mcC}(\mcX) = \pd(\mcX)$ and $\id_{\mcC}(\mcX) = \id(\mcX)$ for the (absolute) projective and injective dimensions, and $\mcP^{<\infty}_{\mcC} = \mcP^{< \infty}$ and $\mcI^{<\infty}_{\mcC} = \mcI^{< \infty}$ for the classes of objects with finite projective and injective dimension. The equality $\pd_{\mcY}(\mcX) = \id_{\mcX}(\mcY)$ is a folklore result (see for instance Auslander and Buchweitz's  \cite{AB}).


\subsection*{Approximations}

A morphism $f \colon X \to C$ is said to be an \emph{$\mcX$-precover of $C$} if $X \in \mcX$ and if for every other morphism $f' \colon X' \to C$ with $X' \in \mcX$, there exists $h \colon X' \to X$ such that $f' = f \circ h$. If in addition, in the case where $X' = X$ and $f' = f$, the equality $f = f \circ h$ can only be completed by automorphisms $h$ of $X$, then $f$ is said to be an \emph{$\mcX$-cover of $C$}. Finally, an $\mcX$-precover $f \colon X \to C$ of $C$ is \emph{special} if it is epic and $\Ker(f) \in \mcX^{\perp_1}$. Dually, one has the concepts of (special) $\mcX$-(pre)envelopes.


\subsection*{(Co)resolving subcategories}

The class $\mcX$ is \emph{resolving} if $\mcP \subseteq \mcX$ and if $\mcX$ is closed under extensions and under taking kernels of epimorphisms between its objects (we shall often say ``closed under epikernels`` for short). \emph{Coresolving subcategories} are defined dually. If $\mcX$ is closed under extensions, direct summands and epikernels (resp., monocokernels), then $\mcX$ is said to be \emph{left} (resp., \emph{right}) \emph{thick}. By a \emph{thick} class we shall mean a class that is both left and right thick. The smallest thick subcategory containing $\mcX$ shall be denoted by ${\rm Thick}(\mcX)$.


\subsection*{Cotorsion pairs}

Two classes $\mcX, \mcY \subseteq \mcC$ form a \emph{cotorsion pair $(\mcX,\mcY)$} in $\mcC$ if $\mcX = {}^{\perp_1}\mcY$ and $\mcY = \mcX^{\perp_1}$. A cotorsion pair $(\mcX,\mcY)$ is \emph{complete} if every object of $\mcC$ has a special $\mcX$-precover and a special $\mcY$-preenvelope. Finally, $(\mcX,\mcY)$ is said to be \emph{hereditary} if $\Ext^i(\mcX,\mcY) = 0$ for every $i > 0$.

\begin{remark}
Let $(\mcX,\mcY)$ be a cotorsion pair in $\mcC$:
\begin{itemize}
\item $(\mcX,\mcY)$ is hereditary complete if, and only if, $\mcX = {}^{\perp}\mcY$, $\mcY = \mcX^{\perp}$, and every object of $\mcC$ has a special $\mcX$-precover and a special $\mcY$-preenvelope.
\item If $\mcC$ has enough projective (resp., injective) objects, then $(\mcX,\mcY)$ is hereditary if, and only if, $\mcX$ is resolving (resp., $\mcY$ is coresolving).
\end{itemize}
\end{remark}


\subsection*{Model structures on exact categories}

Following Gillespie's \cite{GiExact}, an \emph{exact model structure} on an exact category $\mcC$ is formed by three classes of morphisms in $\mcC$, called fibrations, cofibrations and weak equivalences, such that: a morphism $f$ is a cofibration (resp., fibration) if, and only if, it is an admissible monomorphism (resp., an admissible epimorphism) such that $\CoKer(f)$ is a cofibrant object (resp., $\Ker(f)$ is a fibrant object). Recall that an object $C \in \mcC$ is (co)fibrant if $C \to 0$ (resp., $0 \to C$) is a (co)fibration.

In Gillespie's \cite[Thm. 3.3]{GiExact}, there is an appealing one-to-one correspondence between exact model structures on weakly idempotent complete (w.i.c.) exact categories $\mcC$ (see \cite[Def. 2.2]{GiExact}) and pairs of cotorsion pairs in $\mcC$ satisfying certain \emph{compatibility condition}. More specifically, we shall be only interested in the implication that asserts that if $\mcA$, $\mcB$ and $\mcW$ are three classes of objects of $\mcC$ such that $\mcW$ is thick, and $(\mcA \cap \mcW,\mcB)$ and $(\mcA,\mcB \cap \mcW)$ are complete cotorsion pairs, then there exists a unique exact model structure on $\mcC$ such that $\mcA$, $\mcB$ and $\mcW$ are the classes of cofibrant, fibrant and trivial objects, respectively.\footnote{This was originally proved in the setting of abelian categories by Hovey in \cite[Thm. 2.2]{Hovey}, and independently by Beligiannis and Reiten in \cite[Thm. VIII.3.6]{BR}. In this case, model structures obtained from $\mcA$, $\mcB$ and $\mcW$ as before are called \emph{abelian}.} As these model structures are uniquely determined by the classes $\mcA$, $\mcB$ and $\mcW$, then they will be denoted as triples $(\mcA,\mcW,\mcB)$, known as \emph{Hovey triples}.


\section{Relative Gorenstein objects}\label{sec:RG-objects}

In this section we consider relative Gorenstein objects in the sense of \cite[Defs. 3.2 \& 3.7]{BMS}. Given a pair $(\mcX,\mcY)$ of classes of objects in an abelian category $\mcC$, an object $C \in \mcC$ is said to be \emph{$(\mcX,\mcY)$-Gorenstein projective} if $C \simeq Z_0(X_\bullet)$ for some exact and $\Hom(-,\mcY)$-acyclic complex $X_\bullet \in \Ch(\mcX)$. Here, $\Ch(\mcX)$ denotes the class of complexes $X_\bullet \in \Ch(\mcC)$ with $X_m \in \mcX$ for every $m \in \mbZ$. \emph{$(\mcX,\mcY)$-Gorenstein injective objects} are defined dually.  One also has the notion of relative weak Gorenstein objects. According to \cite[Def. 3.11]{BMS}, an object $C \in \mcC$ is \emph{weak $(\mcX,\mcY)$-Gorenstein projective} if there exists an $\mcX$-coresolution of $C$ with cycles in ${}^{\perp}\mcY$ (including $C \in {}^\perp\mcY$). \emph{Weak $(\mcX,\mcY)$-Gorenstein injective objects} are defined dually. The classes of (weak) $(\mcX,\mcY)$-Gorenstein projective and $(\mcX,\mcY)$-Gorenstein injective objects will be denoted by $\mcGP_{(\mcX,\mcY)}$ and
$\mcGI_{(\mcX,\mcY)}$ (resp., $\mathcal{WGP}_{(\mcX,\mcY)}$ and $\mathcal{WGI}_{(\mcX,\mcY)}$).

Some of the most useful properties of $\mcGP_{(\mcX,\mcY)}$ (see Proposition \ref{prop:GP_description} below) are obtained when the pair $(\mcX,\mcY)$ is admissible in the sense of \cite[Def. 3.1]{BMS}. Recall that $\omega \subseteq \mcX$ is a \emph{relative cogenerator} in $\mcX$ if for every $X \in \mcX$ there is a monomorphism $X \rightarrowtail W$ with $W \in \omega$ and cokernel in $\mcX$. \emph{Relative generators} are defined dually. Consider now the following assertions:
\begin{enumerate}
\item[\GPaone] $\pd_{\mcY}(\mcX) = 0$.

\item[\GPatwo] $\mcX$ is a relative generator in $\mcC$.

\item[\GPathree] $\mcX$ and $\mcY$ are closed under finite coproducts.

\item[\GPafour] $\mcX$ is closed under extensions.

\item[\GPafive] $\mcX \cap \mcY$ is a relative cogenerator in $\mcX$.
\end{enumerate}
The pair $(\mcX,\mcY)$ is said to be (\emph{weak}) \emph{GP-admissible} if it satisfies all of the previous five assertions (resp., \GPaone \ and \GPatwo). The dual notions are known as (\emph{weak}) \emph{GI-admissible pairs} (see \cite[Def. 3.6]{BMS}).

\begin{remark}\label{rmk:GPGI-admissible} ~ \
\begin{enumerate}
\item $\mcC$ has enough projective objects if, and only if, the pair $(\mcP,\mcC)$ is (weak) GP-admissible.

\item The previous definition also covers the concept of $(\mcX,\mcY)$-Gorenstein projective modules in the sense of Pan and Cai's \cite[Def. 2.1]{PanCai}), where it is assumed that $\mcP \subseteq \mcX$. However, in \cite{PanCai} it is not necessarily required that $(\mcX,\mcY)$ is a GP-admissible pair.

\item Note that every hereditary complete cotorsion pair $(\mcX,\mcY)$ in $\mcC$ is both GP-admissible and GI-admissible. Moreover, the resulting pairs $(\mcX,\mcX \cap \mcY)$ and $(\mcX \cap \mcY,\mcY)$ are GP-admissible and GI-admissible, respectively. With respect to the latter, we have the $(\mcX,\mcX \cap \mcY)$-Gorenstein projective and $(\mcX \cap \mcY,\mcY)$-Gorenstein injective objects studied in Xu's \cite{Xu17}.\footnote{Xu's work appears within the setting of the category $\Mod(R)$ of $R$-modules, although several important results there are also valid for abelian categories.}
\end{enumerate}
\end{remark}

\begin{example}\label{ex:GP_objects} \
\begin{enumerate}
\item Note that $\mcGP_{(\mcP,\mcP)}$ coincides with the class of Gorenstein projective objects of $\mcC$, which we shall denote by $\mcGP$ for simplicity (see for instance Enochs and Jenda's \cite[Def. 10.2.1]{EJ}).

\item Similarly, if $\mcF(R)$ denotes the class of flat $R$-modules, for the GP-admissible pair $(\mcP(R),\mcF(R))$, the class $\mcGP_{(\mcP(R),\mcF(R))}$ coincides with the class of Ding projective $R$-modules, which we denote by $\mcDP(R)$ (see Gillespie's \cite[Def. 3.7]{GillespieDCh}).

\item Consider the category $\Qcoh(X)$ of quasi-coherent sheaves over a noetherian and semi-separated scheme $X$. Let $\mfF(X)$ denote the class of (quasi-coherent) flat sheaves (over $X$), and $\mfC(X) = \mfF(X)^{\perp_1}$ the class of cotorsion sheaves. We recall that a sheaf $\msM \in \Qcoh(X)$ is \emph{Gorenstein flat} if $\msM = Z_0(\msF_\bullet)$ for some exact complex $\msF_\bullet$ of flat sheaves in $\Qcoh(X)$ such that $\msF_\bullet \otimes \msI$ is an exact complex of sheaves of $\mcO_X$-modules for every injective sheaf $\msI \in \Qcoh(X)$.\footnote{Here, $\otimes$ denotes the tensor sheaf (see Hartshorne's \cite[pp. 109]{Hartshorne}).} Such complexes $\msF_\bullet$ are called \emph{$F$-totally acyclic}. By Murfet and Salarian's \cite[Thm. 4.18]{MurfetSalarian}, an exact complex of flat sheaves $\msF_\bullet$ over such schemes $X$ is $F$-totally acyclic if, and only if, it is $\Hom(-,\mfF(X) \cap \mfC(X))$-acyclic. It follows that a sheaf is Gorenstein flat if, and only if, it is $(\mfF(X), \mfF(X) \cap \mfC(X))$-Gorenstein projective.

\item Let $Q$ be a quiver and $\mcC = {\rm Rep}(Q,\Mod(R))$ be the category of $\Mod(R)$-valued representations of $Q$. Let us recall some notations from \cite{EnOyoTor04,HolmJor19}. We denote by $Q_0$ and $Q_1$ the sets of vertexes and arrows of $Q$, respectively. For every $\alpha \in Q_1$, we denote by $s(\alpha)$ and $t(\alpha)$ the source and target of $\alpha$, respectively. By $Q^{i \to \ast}_1$, with $i \in Q_0$, we mean the subset of $Q_1$ formed by those $\alpha \in Q_1$ such that $s(\alpha) = i$. The notation $Q^{\ast \to i}_1$ has a similar meaning. For every $F \in {\rm Rep}(Q,\Mod(R))$ and $i \in Q_0$,
\[
\psi^F_i \colon F_i \to \prod_{\alpha \in Q^{i \to \ast}_1} F_{t(\alpha)}
\]
is the morphism induced by the morphisms $F_\alpha \colon F_i \to F_{t(\alpha)}$. Dually, we have the morphism
\[
\varphi^F_i \colon \bigoplus_{\alpha \in Q^{\ast \to i}_1} F_{s(\alpha)} \to F_i
\]
is the morphism induced by the morphisms $F_\alpha \colon F_{s(\alpha)} \to F_i$. Given a class of $R$-modules $\mcX$, we denote by $\Psi(\mcX)$ the class of representations $F \in {\rm Rep}(Q,\Mod(R))$ such that ${\rm Ker}(\psi^F_i) \in \mcX$ and $\psi^F_i$ is an epimorphism for every $i \in Q_0$. Dually, $\Phi(\mcX)$ denotes the class of representations $F \in {\rm Rep}(Q,\Mod(R))$ such that ${\rm CoKer}(\varphi^F_i) \in \mcX$ and $\varphi^F_i$ is a monomorphism for every $i \in Q_0$.

Now consider the $A_2$ quiver
\[
Q \colon 1 \xrightarrow{\alpha} 2
\]
with no relations, and let $R$ be an Iwanaga-Gorenstein ring. Then,
\[
(\Phi(\mcGP(R)),{\rm Rep}(Q,\mcP(R)^{< \infty}))
\]
is a hereditary complete cotorsion pair by \cite[Thm. A.]{DiLiLiangXu} and \cite[Prop. 4.18]{ArgudinMendoza-quiver}, and so in particular it is a GP-admissible pair. Dually,
\[
({\rm Rep}(Q,\mcI(R)^{< \infty}),\Psi(\mcGI(R)))
\]
is GI-admissible. In this case,
\begin{align*}
\Phi(\mcGP(R)) & = \mcGP_{(\Phi(\mcGP(R)),{\rm Rep}(Q,\mcP(R)^{< \infty}))}, \text{ and} \\
\Psi(\mcGI(R)) & =  \mcGI_{({\rm Rep}(Q,\mcI(R)^{< \infty}),\Psi(\mcGI(R)))}.
\end{align*}

\item Another important example of GP-admissible and GI-admissible pairs can be obtained from the notion of $n$-$\mcX$-tilting classes in an abelian category $\mcC$, introduced by Argud\'in and the second author in \cite[Def. 3.1]{AM21}. This concept offers a unified framework for different generalizations of tilting theory appearing in the literature. $n$-$\mcX$-Tilting classes are defined in terms of the following relative resolutions and approximations \cite[Defs. 3.13 \& 4.1 (b,e)]{AM21P1}: Let $\mcT$ and $\mcX$ be classes of objects in $\mcC$, and $n \geq 0$ be a nonnegative integer. A \emph{finite $\mcT_\mcX$-coresolution of length $n$} of an object $M \in \mcC$ is an exact sequence in $\mcC$ of the form
\[
M \stackrel{f_0}\rightarrowtail T_0 \xrightarrow{f_1} T_1 \to \cdots \to T_{n-1} \stackrel{f_n}\twoheadrightarrow T_n,
\]
with $T_n \in \mcX \cap \mcT$, $T_k \in \mcT$ for every $0 \leq k \leq n-1$, and $\Ima(f_i) \in \mcX$ for every $1 \leq i \leq n-1$. The class of all objects in $\mcC$ having a finite $\mcT_\mcX$-coresolution will be denoted by $\mcT^{\vee}_\mcX$.

The class $\mcT$ is \emph{$n$-$\mcX$-tilting} if the following conditions are satisfied:
\begin{itemize}
\item[\tiltzero] $\mcT$ is closed under direct summands.

\item[\tiltone] $\pd_\mcX(\mcT) \leq n$.

\item[\tilttwo] $\mcT \cap \mcX \subseteq \mcT^\perp$.

\item[\tiltthree] There is a class $\alpha \subseteq \mcT^{\vee}_\mcX$ which is a relative generator in $\mcX$.

\item[\tiltfour] There is a class $\beta \subseteq \mcT^\perp \cap \mcX^\perp$ which is a relative cogenerator in $\mcX$.

\item[\tiltfive] Every $M \in \mcT^\perp \cap \mcX$ admits a $\mcT$-precover $T\to M$ with $T \in \mcX$.
\end{itemize}
The dual concept is known as \emph{$n$-$\mcX$-cotilting} classes. In the case $\mcX = \mcC$, $n$-$\mcX$-(co)tilting classes are simply called \emph{$n$-(co)tilting}.

Suppose now that $\mcT \subseteq \mcC$ is an $n$-tilting class in an abelian category $\mcC$ with enough projective and injective objects. In \cite[Lem. 3.44]{AM21} it is proved that $({}^\perp(\mcT^\perp),\mcT^\perp)$ is a hereditary complete cotorsion pair in $\mcC$ with ${}^\perp(\mcT^\perp) \cap \mcT^\perp = \mcT$. So from Remark \ref{rmk:GPGI-admissible} (3) we know that $(\mcT,\mcT^\perp)$ is a GI-admissible pair in $\mcC$. Dually, if $\mcT \subseteq \mcC$ is an $n$-cotilting class, then $({}^\perp\mcT,\mcT)$ is a GP-admissible pair in $\mcC$.
\end{enumerate}
\end{example}


\subsection*{Properties of relative Gorenstein objects}

In what follows, if $(\mcX,\mcY)$ is a (weak) GP-admissible pair and $(\mathcal{W,Z})$ is a (weak) GI-admissible pair, we denote its \emph{kernels} by
\begin{align*}
\omega & := \mcX \cap \mcY & & \text{and} & \nu & := \mcW \cap \mcZ.
\end{align*}
The next result consists of recalling some characterizations and properties of $(\mcX,\mcY)$-Gorenstein projective objects. The reader can find the proofs in \cite[Prop. 3.16 \& Coroll. 3.33]{BMS} and \cite[Prop. 3.7]{Xu17}.

\begin{proposition}\label{prop:GP_description}
Let $(\mcX,\mcY)$ be a weak GP-admissible pair in $\mcC$. Then, $C \in \mcGP_{(\mcX,\mcY)}$ if, and only if, $C \in {}^\perp\mcY$ and has a $\Hom(-,\mcY)$-acyclic $\mcX$-coresolution. If in addition, $(\mcX,\mcY)$ is GP-admissible, then $\mcGP_{(\mcX,\mcY)}$ is left thick.
\end{proposition}

\begin{remark}\label{rmk:containment_projectives}
If $(\mcX,\mcY)$ is a GP-admissible pair with $\mcP \subseteq \mcX$ (which occurs, for instance, if $\mcX$ is closed under direct summands by using \GPatwo), then $\mcGP_{(\mcX,\mcY)}$ is resolving.
\end{remark}

The following result is a slight generalization of Xu's \cite[Prop. 3.6]{Xu17}, which characterizes when $(\mcX,\omega)$-Gorenstein projective $R$-modules are Gorenstein projective. Its validity can be traced back to Auslander and Buchweitz \cite{AB}, and can be proved following the arguments in \cite[Lems. 3.3, 3.5 \& Prop. 3.6]{Xu17} and using Proposition \ref{prop:GP_description}.

\begin{proposition}\label{prop:relative_GP_are_absolute_GP}
Let $(\mcX,\mcY)$ be a GP-admissible pair in an abelian category $\mcC$ with enough projective objects, such that $\mcX$ and $\omega$ are closed under direct summands. Then, $\mcGP_{(\mcX,\omega)} = \mcGP$ if, and only if, $\omega = \mcP$.
\end{proposition}


\subsection*{Relative Gorenstein dimensions}

Several properties of relative Gorenstein dimensions can be deduced from the theory of Frobenius pairs presented in \cite[Def. 2.5]{BMPS19}. Recall that a pair $(\mcX,\mcY)$ of classes of objects in $\mcC$ is said to be \emph{left Frobenius} if the following conditions hold:
\begin{enumerate}
\item[\lfpone] $\mcX$ is left thick.

\item[\lfptwo] $\mcY$ is closed under direct summands.

\item[\lfpthree] $\mcY$ is a relative cogenerator in $\mcX$ with $\id_{\mcX}(\mcY) = 0$.
\end{enumerate}
The notion of \emph{right Frobenius pair} in $\mcC$ is dual.

Several properties of left (and right) Frobenius pairs are proved in \cite{BMPS19}. Below we add a new property related to the description of the class $(\mcY^\wedge)^\vee$, which from now on will be denoted by
\[
\mcY^\smile := (\mcY^\wedge)^\vee.
\]

\begin{proposition}\label{essLFC}
Let $(\mcX,\mcY)$ be a left Frobenius pair in $\mcC.$ Then:
\begin{enumerate}
\item $\mcY^\smile = \mcI^{<\infty}$ if $\resdim_\mcX(\mcC)<\infty$.

\item $\id_\mcX(M) = \id_{\mcY}(M) = \coresdim_{\mcY^\wedge}(M)$ for every $M \in \mcY^\smile$.
\end{enumerate}
\end{proposition}

\begin{proof}
Part (1) follows as in \cite[Coroll. 5.3]{AB}. For part (2), notice that $\mcY^\wedge$ is right thick by \cite[Thm. 2.11 \& Prop. 2.13]{BMPS19}, and that $\mcY$ is a relative generator in $\mcY^\wedge$ with $\id_{\mcY}(\mcY^\wedge) = 0$. Thus, the dual of \cite[Props. 2.1 \& 4.7]{AB} yields the desired equality.
\end{proof}

\begin{example}\label{ex:relative_GP_Frobenius_pair}
Let $(\mcX,\mcY)$ be a GP-admissible pair in $\mcC$ such that $\omega$ is closed under direct summands. Then, $(\mcGP_{(\mcX,\mcY)},\omega)$ is a left Frobenius pair in $\mcC$ by \cite[Thm. 3.32 \& Coroll. 4.9]{BMS}.
\end{example}

One of the most important properties that one can deduce from the previous example is related to the existence of special $\mcGP_{(\mcX,\mcY)}$-precovers. Before stating it (Proposition \ref{coro:properties_relative_Gorenstein_projective} (4)), recall from \cite[Defs. 3.3 \& 4.17]{BMS} that the \emph{$(\mcX,\mcY)$-Gorenstein projective dimension} of an object $C \in \mcC$ is defined as
\[
\Gpd_{(\mcX,\mcY)}(C) := \resdim_{\mcGP_{(\mcX,\mcY)}}(C),
\]
and that the \emph{global $(\mcX,\mcY)$-Gorenstein projective dimension of $\mcC$} is given by
\begin{align*}
\glGPD_{(\mcX,\mcY)}(\mcC) & := \resdim_{\mcGP_{(\mcX,\mcY)}}(\mcC).
\end{align*}
The \emph{$(\mcX,\mcY)$-Gorenstein injective dimension of $C$} and the \emph{global $(\mcX,\mcY)$-Gorenstein injective dimension of $\mcC$}, denoted $\Gid_{(\mcX,\mcY)}(C)$ and $\glGID_{(\mcX,\mcY)}(\mcC)$, are defined dually.

\begin{remark}
The cases $\mcX = \mcY = \mcP$ and $\mcX = \mcY = \mcI$ yield the (\emph{global}) \emph{Gorenstein projective} and \emph{Gorenstein injective dimensions}, respectively, which will be denoted by $\Gpd(C)$, $\Gid(C)$, $\glGPD(\mcC)$ and $\glGID(\mcC)$, for simplicity. One appealing fact about these global dimensions is that
\[
\glGPD(\mcC) = \glGID(\mcC)
\]
in the case where $\mcC$ has enough projective and injective objects (see \cite[Prop. VII.1.3 (v)]{BR}). The relative version of this fact will be also valid in our proposed notion of relative Gorenstein category, to be introduced in Section \ref{sec:RG-categories} (see Theorem \ref{MThmGc}).
\end{remark}

The following result summarizes some previously known, and some other new, properties of relative Gorenstein objects and dimensions.

\begin{proposition}\label{coro:properties_relative_Gorenstein_projective}
Let $(\mcX,\mcY)$ be a GP-admissible pair in $\mcC$ with $\omega$ closed under direct summands. Then, the following assertions hold true:
\begin{enumerate}
\item $(\mcGP_{(\mcX,\mcY)},\omega)$ is a left Frobenius pair in $\mcC$, and $\mcGP_{(\mcX,\mcY)}^\wedge$ is thick.

\item $\omega = \mcGP_{(\mcX,\mcY)} \cap \omega^\wedge$ and $\mcGP_{(\mcX,\mcY)}^\wedge \cap {}^{\perp}\omega = \mcGP_{(\mcX,\mcY)} = \mcGP_{(\mcX,\mcY)}^\wedge \cap {}^{\perp}(\omega^\wedge)$.

\item $\pd_{\omega^\wedge}(\mcGP_{(\mcX,\mcY)}) = 0$.

\item For every $C \in \mcGP_{(\mcX,\mcY)}^\wedge$ there exist short exact sequences
\[
K \rightarrowtail D \twoheadrightarrow C \text{ \ and \  } C \rightarrowtail L \twoheadrightarrow D'
\]
where $D, D' \in \mcGP_{(\mcX,\mcY)}$,
\[
\hspace{1cm} \resdim_{\omega}(K) = \Gpd_{(\mcX,\mcY)}(C) - 1 \text{ \ and \ } \resdim_{\omega}(L) \leq \Gpd_{(\mcX,\mcY)}(C).
\]
Moreover,
\[
\Gpd_{(\mcX,\mcY)}(C) = \pd_{\omega}(C) = \pd_{\omega^\wedge}(C).
\]

\item ${\rm Thick}(\omega) = \omega^\smile = \mcGP_{(\mcX,\mcY)}^\wedge \cap \mcI^{<\infty}_{\mcGP_{(\mcX,\mcY)}}$. Moreover,
\[
\id_{\mcGP_{(\mcX,\mcY)}}(M) = \id_{\omega}(M) = \coresdim_{\omega^\wedge}(M)
\]
for every $M \in \omega^\smile$.

\item $\omega^\vee = \mcGP_{(\mcX,\mcY)} \cap \mcI^{<\infty}_{\mcGP_{(\mcX,\mcY)}}$. Moreover,
\[
\id_{\mcGP_{(\mcX,\mcY)}}(M) = \id_{\omega}(M) = \coresdim_{\omega^\wedge}(M) = \coresdim_{\omega}(M)
\]
for every $M \in \omega^\vee$.

\item If $\glGPD_{(\mcX,\mcY)}(\mcC) < \infty$, then $\omega^\smile = \mcI^{<\infty}$.

\item If $\id_{\mcGP_{(\mcX,\mcY)}}(\mcGP_{(\mcX,\mcY)}) < \infty$, then the following assertions hold:
\begin{itemize}
\item[(i)] $\mcGP_{(\mcX,\mcY)} = \omega^\vee$ and $\mcGP_{(\mcX,\mcY)}^\wedge = (\omega^\vee)^\wedge = \omega^\smile$.

\item[(ii)] $\id_{\mcGP_{(\mcX,\mcY)}}(\mcGP_{(\mcX,\mcY)}) = \id_\omega(\mcGP_{(\mcX,\mcY)}) = \coresdim_{\omega^\wedge}(\mcGP_{(\mcX,\mcY)}) < \infty$.

\item[(iii)] In case $\omega$ is closed under epikernels, we have that $\mcGP_{(\mcX,\mcY)} = \omega$, and moreover $\mcGP_{(\mcX,\mcY)} = \mcP$ if $\mcC = \mcGP_{(\mcX,\mcY)}^\wedge$.
\end{itemize}
\end{enumerate}
\end{proposition}

\begin{proof}
Parts (1) to (4) are consequences of \cite{BMPS19,BMS}.
\begin{enumerate}
\item[(5)] It follows from \cite[Thm. 2.15]{BMPS19} and Proposition \ref{essLFC} (2) applied to the left Frobenius pair $(\mcGP_{(\mcX,\mcY)},\omega)$.

\item[(6)] The equality $\omega^\vee = \mcGP_{(\mcX,\mcY)} \cap \mcI^{<\infty}_{\mcGP_{(\mcX,\mcY)}}$ follows by \cite[Prop. 2.7]{BMPS19}. The last equality between relative dimensions is a consequence of part (5), the dual of \cite[Prop. 2.1]{AB} and Proposition \ref{essLFC} (2) applied to the left Frobenius pair $(\mcGP_{(\mcX,\mcY)},\omega)$ (see part (1)).

\item[(7)] It follows by Proposition \ref{essLFC} (1).

\item[(8)] Notice that parts (ii) and (iii) will follow from (i), (2) and (6). So we only focus on proving (i). Since $\id_{\mcGP_{(\mcX,\mcY)}}(\mcGP_{(\mcX,\mcY)}) < \infty$, we get from part (6) that $\mcGP_{(\mcX,\mcY)} = \omega^\vee$. Thus, $\mcGP_{(\mcX,\mcY)}^\wedge = (\omega^\vee)^\wedge$, and it is thick by (1). The rest follows by noticing that $\omega \subseteq (\omega^\vee)^\wedge \subseteq {\rm Thick}(\omega) = \omega^\smile$ (see part (5)).
\end{enumerate}
\end{proof}

The next result describes a relation between the $(\mcX,\mcY)$-Gorenstein projective dimension, and the resolution dimensions relative to $\mcY$ and $\omega$.

\begin{proposition}\label{prop:GPdim_vs_omegadim}
Let $(\mcX,\mcY)$ be a GP-admissible pair in $\mcC$ such that $\omega$, $\mcY^\wedge$ and $\mcX \cap \mcY^\wedge$ are closed under direct summands, and $\mcY^\wedge$ is closed under extensions. Then, the following equality holds for every $n \geq 0$:
\[
(\mcGP_{(\mcX,\mcY)})^\wedge_n \cap \mcY^\wedge = \omega^\wedge_n.
\]
\end{proposition}

\begin{proof}
Let us use induction on $n$. The case $n = 0$ follows by \cite[Thm. 3.34 (d)]{BMS}. Now suppose that $(\mcGP_{(\mcX,\mcY)})^\wedge_{n-1} \cap \mcY^\wedge = \omega^\wedge_{n-1}$, and let $C \in (\mcGP_{(\mcX,\mcY)})^\wedge_n \cap \mcY^\wedge$. By Proposition \ref{coro:properties_relative_Gorenstein_projective} (4), there exists a short exact sequence $K \rightarrowtail G \twoheadrightarrow C$ with $G \in \mcGP_{(\mcX,\mcY)}$ and $K \in \omega^\wedge_{n-1}$. By the induction hypothesis, we have that $K \in \mcY^\wedge$. Thus, $G \in \mcGP_{(\mcX,\mcY)} \cap \mcY^\wedge = \omega$ since $\mcY^\wedge$ is closed under extensions, and hence $C \in \omega^\wedge_n$. The containment $\omega^\wedge_n \subseteq (\mcGP_{(\mcX,\mcY)})^\wedge_n \cap \mcY^\wedge$, on the other hand, is clear.
\end{proof}

\begin{remark}
Under some sufficient conditions on a GP-admissible pair $(\mcX,\mcY)$, we can guarantee that $\mcY^\wedge$ is closed under extensions. This occurs for instance if
\begin{enumerate}
\item $\mcY$ is closed under direct summands, and $\mcY \subseteq \mcGP_{(\mcX,\mcY)}$ (see \cite[Thm. 3.34 (b)]{BMS}); or

\item $\mcY$ is closed under extensions, and $\omega$ is a relative generator in $\mcY$ with $\pd_{\mcY}(\omega)$ $= 0$ (see \cite[Lem. 3.5 (2)]{HMP} by Huerta and the second and third authors).
\end{enumerate}
In particular, we have that $\mcY^\wedge$ is closed under extensions if $(\mcX,\mcY)$ is a hereditary complete cotorsion pair in $\mcC$.
\end{remark}

We have the following consequence of Proposition \ref{coro:properties_relative_Gorenstein_projective} (5), (6) and (7).

\begin{corollary}\label{coro:essGIC}
For a GP-admissible pair $(\mcX,\mcY)$ in $\mcC$ with $\omega$ closed under direct summands and $\glGPD_{(\mcX,\mcY)}(\mcC) < \infty$, the following statements hold true:
\begin{enumerate}
\item ${\rm Thick}(\omega) = \omega^\smile = \mcI^{<\infty}_{\mcGP_{(\mcX,\mcY)}} = \mcI^{<\infty}$. Moreover,
\[
\id_{\mcGP_{(\mcX,\mcY)}}(M) = \id_\omega(M) = \coresdim_{\omega^\wedge}(M)
\]
for every $M \in \mcI^{<\infty}$.

\item If $\id_{\mcGP_{(\mcX,\mcY)}}(\mcGP_{(\mcX,\mcY)}) < \infty$, then
\begin{enumerate}
\item[(i)] $\mcC = \mcI^{<\infty}$ and $\mcGP_{(\mcX,\mcY)} = \omega^\vee$.

\item[(ii)] $\pd(\mcGP_{(\mcX,\mcY)}) = \pd(\omega) = \coresdim_{\omega^\wedge}(\mcC)$.

\item[(iii)] $\mcGP_{(\mcX,\mcY)} = \mcP$ if $\omega$ is closed under epikernels.
\end{enumerate}
\end{enumerate}
\end{corollary}


\section{Relative Gorenstein model structures}\label{sec:RG-models}

One of the main contributions in Xu's work \cite[Thm. 4.2 (2)]{Xu17} is the discovery of new abelian model structures on the category $\Mod(R)$ of left $R$-modules involving $\mcGP_{(\mcX,\omega)}$ as the class of cofibrant objects, where $(\mcX,\mcY)$ is a hereditary complete cotorsion pair in $\Mod(R)$ such that $\Mod(R) = (\mcGP_{(\mcX,\omega)})^\wedge$. The arguments used by Xu cover part of the theory of AB contexts along with Hovey-Gillespie correspondence. Specifically, under the previous conditions, one obtains
two complete cotorsion pairs $(\mcGP_{(\mcX,\omega)} \cap \mcX^\wedge,\mcY)$ and $(\mcGP_{(\mcX,\omega)},\mcX^\wedge \cap \mcY)$ in $\Mod(R)$ (where $\mcX^\wedge$ is thick by \cite[Thm. 2.11]{BMPS19}), and thus a unique abelian model structure on $\Mod(R)$ such that $\mcGP_{(\mcX,\omega)}$, $\mcY$ and $\mcX^\wedge$ are the classes of cofibrant, fibrant and trivial objects, respectively.

However, the condition $\Mod(R) = \mcGP_{(\mcX,\omega)}^\wedge$ is more or less difficult to see in general, even in the well known case where $\mcY$ is the whole category $\Mod(R)$ and $\mcX$ is the class of projective left $R$-modules. For instance, when $R$ is an Iwanaga-Gorenstein ring (see \cite[Def. 9.1.1]{EJ}), then $\Mod(R) = \mcGP^\wedge$. So a natural step towards a more general context on which a similar model structure still exists is to drop the assumption that $\Mod(R) = \mcGP_{(\mcX,\omega)}^\wedge$. Indeed, the main goal of this section is to prove the following extension of Xu's \cite[Thm. 4.2 (2)]{Xu17}.

\begin{theorem}\label{theo:Xu_model}
Let $(\mcX,\mcY)$ be a hereditary complete cotorsion pair in $\mcC$. Then, there exists a unique hereditary exact model structure on $\mcGP_{(\mcX,\omega)}^\wedge$ such that $\mcGP_{(\mcX,\omega)}$, $\mcY \cap \mcGP^\wedge_{(\mcX,\omega)}$ and $\mcX^\wedge$ are the classes of cofibrant, fibrant and trivial objects, respectively. Dually, there exists a unique hereditary exact model structure on $\mcGI^\vee_{(\omega,\mcY)}$ such that $\mcX \cap \mcGI^\vee_{(\omega,\mcY)}$, $\mcGI_{(\omega,\mcY)}$ and $\mcY^\vee$ are the classes of cofibrant, fibrant and trivial objects, respectively. We shall refer to these models as the \textbf{Gorenstein projective} and \textbf{Gorenstein injective model structures relative to $\bm{(\mcX,\mcY)}$}, and shall be denoted by $\mcM^{\rm GP}_{(\mcX,\mcY)}$ and $\mcM^{\rm GI}_{(\mcX,\mcY)}$.
\end{theorem}

The proof of previous theorem will be possible by applying the Hovey-Gillespie correspondence in the setting of w.i.c. exact categories. For now, we want to stress the fact that relative Gorenstein projective model structures are different from many exact Gorenstein model structures already covered in the literature, such as the \emph{Auslander-Buchweitz} (or \emph{AB}) \emph{model structures}. These structures were introduced and studied in \cite{BMPS19}, and cover a wide range of projective exact model structures whose cofibrant objects are relative Gorenstein objects, such as: Gorenstein projective modules, Ding projective modules, Gorenstein AC-projective modules, among others. AB model structures are obtained from strong left Frobenius pairs in the following way: if $(\mcX,\omega)$ is a \emph{strong} left Frobenius pair in $\mcC$ (that is, $(\mcX,\omega)$ is a left Frobenius pair where $\omega$ is a relative generator in $\mcX$ with $\pd_{\mcX}(\omega) = 0$), then there exists a unique exact model structure on $\mcX^\wedge$ such that $\mcX$, $\mcX^\wedge$ and $\omega^\wedge$ are the classes of cofibrant, fibrant and trivial objects, respectively (see \cite[Thm. 4.1]{BMPS19}). The fact that $\mcM^{\rm GP}_{(\mcX,\mcY)}$ and $\mcM^{\rm GI}_{(\mcX,\mcY)}$ are not going to be AB models in general is because the left Frobenius pair $(\mcGP_{(\mcX,\omega)},\omega)$ mentioned in Example \ref{ex:relative_GP_Frobenius_pair} is not always strong.  Moreover, the only strong left Frobenius pair of the form $(\mcGP_{(\mcX,\omega)},\omega)$ is precisely $(\mcGP,\mcP)$ (that is, $\mcGP_{(\mcX,\omega)} = \mcGP$, or equivalently, $\omega = \mcP$). In other words, AB model structures practically have no room in the Gorenstein homological and homotopical algebra relative to cotorsion pairs. However, this constraint can be considered to extend Xu's \cite[Prop. 3.6]{Xu17} (and also Proposition \ref{prop:relative_GP_are_absolute_GP}).

\begin{proposition}\label{prop:relative_G-projective_model}
Let $(\mcX,\mcY)$ be a GP-admissible pair in an abelian category $\mcC$ with enough projective objects, such that $\mcX$ and $\omega$ are closed under direct summands. Then, the following conditions are equivalent:
\begin{itemize}
\item[(a)] $\mcGP_{(\mcX,\omega)} = \mcGP$.

\item[(b)] $\omega = \mcP$.

\item[(c)] The left Frobenius pair $(\mcGP_{(\mcX,\omega)},\omega)$ is strong.

\item[(d)] There exists a unique exact model structure on $\mcGP^\wedge_{(\mcX,\omega)}$ with $\mcGP_{(\mcX,\omega)}$, $\mcGP^\wedge_{(\mcX,\omega)}$ and $\omega^\wedge$ as the classes of cofibrant, fibrant and trivial objects, respectively.
\end{itemize}
\end{proposition}

\begin{proof}
The equivalence (a) $\Leftrightarrow$ (b) is Proposition \ref{prop:relative_GP_are_absolute_GP}. We split the rest of the proof into two circuits of implications:
\begin{itemize}
\item (b) $\Rightarrow$ (c) $\Rightarrow$ (b): Assuming (b), since (a) and (b) are equivalent, we have that $(\mcGP_{(\mcX,\omega)},\omega) = (\mcGP,\mcP)$ is a strong left Frobenius pair. Now suppose that (c) holds. For the containment $\mcP \subseteq \omega$, it suffices to notice that for every $P \in \mcP$ there exists a split epimorphism $W \twoheadrightarrow P$ with $W \in \omega$, since $\omega$ is a relative generator in $\mcGP_{(\mcX,\omega)}$ and $\mcP \subseteq \mcX \subseteq \mcGP_{(\mcX,\omega)}$ (see Remark \ref{rmk:containment_projectives}). For the remaining containment $\omega \subseteq \mcP$, let $W \in \omega$. Since $\mcC$ have enough projective objects, there exists a short exact sequence $K \rightarrowtail P \twoheadrightarrow W$ with $P \in \mcP$. Now using the fact that the class $\mcGP_{(\mcX,\omega)}$ is closed under epikernels by Proposition \ref{prop:GP_description} (note that $(\mcX,\omega)$ is also a GP-admissible pair), we have that $K \in \mcGP_{(\mcX,\omega)}$. On the other hand, since $\pd_{\mcGP_{(\mcX,\omega)}}(\omega) = 0$, the previous sequence splits. Hence, $W \in \mcP$.

\item (c) $\Rightarrow$ (d) $\Rightarrow$ (c): The first implication is a consequence of \cite[Thm. 4.1]{BMPS19}. Now suppose that condition (d) holds. By Example \ref{ex:relative_GP_Frobenius_pair}, we have that $(\mcGP_{(\mcX,\omega)},\omega)$ is a left Frobenius pair in $\mcC$. From (d),  $(\omega,\mcGP^\wedge_{(\mcX,\omega)})$ is a complete cotorsion pair in the exact category $\mcGP^\wedge_{(\mcX,\omega)}$, where $\mcGP_{(\mcX,\omega)} \cap \omega^\wedge = \omega$ by Proposition \ref{coro:properties_relative_Gorenstein_projective} (2). So for every $C \in \mcGP_{(\mcX,\omega)}$, we have a short exact sequence $C' \rightarrowtail W \twoheadrightarrow C$ where $W \in \omega$ and $C' \in \mcGP^\wedge_{(\mcX,\omega)}$. Actually, we can assert that $C' \in \mcGP_{(\mcX,\omega)}$ since $\mcGP_{(\mcX,\omega)}$ is closed under epikernels. Thus, $\omega$ is a relative generator in $\mcGP_{(\mcX,\omega)}$. It remains to show that $\Ext^i(W,G) = 0$ for every $i \geq 1$, $W \in \omega$ and $G \in \mcGP_{(\mcX,\omega)}$. By Proposition \ref{coro:properties_relative_Gorenstein_projective} (4), there exists a short exact sequence $G \rightarrowtail H \twoheadrightarrow G'$ where $H \in \omega^\wedge$ and $G' \in \mcGP_{(\mcX,\omega)}$. The previous induces an exact sequence of abelian groups
\[
\Ext^1(W,G') \to \Ext^2(W,G) \to \Ext^2(W,H)
\]
where $\Ext^1(W,G') = 0$ since $\omega = {}^{\perp_1}(\mcGP^\wedge_{(\mcX,\omega)}) \cap \mcGP^\wedge_{(\mcX,\omega)}$, and $\Ext^2(W,H) = 0$ since $\pd_{\omega^\wedge}(\omega) = 0$. Thus, $\Ext^2(W,G) = 0$, and inductively we can conclude that $\Ext^i(W,G) = 0$ for every $i \geq 1$.
\end{itemize}
\end{proof}

\begin{remark}
In the case where $(\mcX,\mcY)$ is a GP-admissible pair in $\mcC$ with enough projective objects, such that $\mcX$ is closed under direct summands and $\omega = \mcP$, then the previous proposition asserts that there exists a unique exact model structure on $\mcGP^\wedge$ such that $\mcGP$ is the class of cofibrant objects, $\mcGP^\wedge$ the class of fibrant objects, and $\mcP^\wedge$ the class of trivial objects. In the case where $\mcC = \Mod(R)$ with $R$ an Iwanaga-Gorenstein ring, this model structure coincides with Hovey's model structure \cite[Thm. 8.6]{Hovey}. For arbitrary rings, we obtain the exact model structure described in \cite[\S \ 6.1]{BMPS19}. Moreover, in the context of bicomplete abelian categories with enough projective and injective objects, Dalezios, Holm and the first named author have obtained the same result in \cite[Thms. 3.7 \& 3.9]{DEH18}.
\end{remark}


\subsection*{Gorenstein model structures relative to a cotorsion pair}

We are now in position to prove Theorem \ref{theo:Xu_model}.

\begin{proof}[\textit{Proof of Theorem \ref{theo:Xu_model}}]
The first thing to note is that $\mcGP^\wedge_{(\mcX,\omega)}$ is a w.i.c. exact category, since by Proposition \ref{coro:properties_relative_Gorenstein_projective} (1) we know that $\mcGP^\wedge_{(\mcX,\omega)}$ is a thick subcategory of $\mcC$. Now we construct two compatible and complete cotorsion pairs in $\mcGP^\wedge_{(\mcX,\omega)}$ that yield the desired model structure. We split our proof into several parts:
\begin{itemize}
\item Construction of the cotorsion pairs: First, note that $(\mcX,\omega)$ is a GP-admissi-ble pair in $\mcC$ with $\omega$ closed under direct summands. Then, by \cite[Coroll. 5.2 (a)]{BMS} we have that $(\mcGP_{(\mcX,\omega)},\omega^\wedge)$ is a complete cotorsion pair in $\mcGP^\wedge_{(\mcX,\omega)}$.

Now we show that $(\mcX, \mcY \cap \mcGP^\wedge_{(\mcX,\omega)})$ is a complete cotorsion pair in $\mcGP^\wedge_{(\mcX,\omega)}$. The equality $\mcY \cap \mcGP^\wedge_{(\mcX,\omega)} = \mcX^{\perp_1} \cap \mcGP^\wedge_{(\mcX,\omega)}$ is immediate since $(\mcX,\mcY)$ is a cotorsion pair in $\mcC$. So let us prove that
\[
\mcX = {}^{\perp_1}(\mcY \cap \mcGP^\wedge_{(\mcX,\omega)}) \cap \mcGP^\wedge_{(\mcX,\omega)}.
\]
The containment ($\subseteq$) is clear. Now let $C \in {}^{\perp_1}(\mcY \cap \mcGP^\wedge_{(\mcX,\omega)}) \cap \mcGP^\wedge_{(\mcX,\omega)}$. Since $(\mcX,\mcY)$ is a complete cotorsion pair in $\mcC$, there exists a short exact sequence $Y \rightarrowtail X \twoheadrightarrow C$ with $X \in \mcX$ and $Y \in \mcY$. Since $X \in \mcX \subseteq \mcGP_{(\mcX,\omega)}$, $C \in \mcGP^\wedge_{(\mcX,\omega)}$ and $\mcGP^\wedge_{(\mcX,\omega)}$ is thick, we have that $Y \in \mcY \cap \mcGP^\wedge_{(\mcX,\omega)}$. It follows that the previous sequence splits, and so $C \in \mcX$. Therefore, $(\mcX, \mcY \cap \mcGP^\wedge_{(\mcX,\omega)})$ is a cotorsion pair in $\mcGP^\wedge_{(\mcX,\omega)}$. Its completeness in the subcategory $\mcGP^\wedge_{(\mcX,\omega)}$ follows by the completeness of $(\mcX,\mcY)$ in $\mcC$ and the thickness of $\mcGP^\wedge_{(\mcX,\omega)}$.

\item Compatibility: We have to show the equalities
\[
(\mcY \cap \mcGP^\wedge_{(\mcX,\omega)}) \cap \mcX^\wedge = \omega^\wedge \ \text{and} \ \mcGP_{(\mcX,\omega)} \cap \mcX^\wedge = \mcX.
\]
The first one follows by the equality $\mcY \cap \mcX^\wedge = \omega^\wedge$. Indeed, since the pair $(\mcX,\mcY)$ is hereditary, the class $\mcY$ is coresolving, and so $\omega^\wedge \subseteq \mcY$. Thus, the containment $\mcY \cap \mcX^\wedge \supseteq \omega^\wedge$ follows. Now let $Y \in \mcY \cap \mcX^\wedge$. By \cite[Thm. 2.8]{BMPS19} there exists a short exact sequence $Y \rightarrowtail H \twoheadrightarrow X$ where $H \in \omega^\wedge$ and $X \in \mcX$, which splits since $\Ext^1(\mcX,\mcY) = 0$. Then, $Y \in \omega^\wedge$ since $\omega^\wedge$ is closed under direct summands by \cite[Thm. 3.6]{BMPS19}.

Regarding the second equality, note that the containment $(\supseteq)$ is clear. Now let $C \in \mcGP_{(\mcX,\omega)} \cap \mcX^\wedge$. By \cite[Thm. 2.8]{BMPS19} again, there exists a short exact sequence $K \rightarrowtail X \twoheadrightarrow C$ where $X \in \mcX$ and $K \in \omega^\wedge$. Now since $\pd_{\omega^\wedge}(\mcGP_{(\mcX,\omega)}) = 0$ by Proposition \ref{coro:properties_relative_Gorenstein_projective} (3), the previous sequence splits, and so $C \in \mcX$ since $\mcX$ is closed under direct summands.
\end{itemize}
The existence of the desired model structure follows by Hovey-Gillespie correspondence for w.i.c. exact categories \cite[Thm. 3.3]{GiExact}, applied to the compatible and complete cotorsion pairs $(\mcGP_{(\mcX,\omega)}, \omega^\wedge)$ and $(\mcX, \mcY \cap \mcGP^\wedge_{(\mcX,\omega)})$ in $\mcGP^\wedge_{(\mcX,\omega)}$.

Finally, the model structure $\mcM^{\rm GP}_{(\mcX,\mcY)}$ is \emph{hereditary} in the sense that the inducing cotorsion pairs $(\mcGP_{(\mcX,\omega)}, \omega^\wedge)$ and $(\mcX, \mcY \cap \mcGP^\wedge_{(\mcX,\omega)})$ in $\mcGP^\wedge_{(\mcX,\omega)}$ are hereditary. This is straightforward from the properties of these cotorsion pairs.
\end{proof}


\subsection*{Homotopy categories of relative Gorenstein model structures}

Given a model structure on an exact category, its homotopy category is obtained by formally inverting the morphisms in the class of weak equivalences (see Hovey's \cite[Def. 1.2.1]{HoveyBook}), or equivalently via homotopy relations (see Dywer and Spali\'nski's \cite{DwyerSpalinski}). Concerning the homotopy categories of $\mcM^{\rm GP}_{(\mcX,\mcY)}$ and $\mcM^{\rm GI}_{(\mcX,\mcY)}$, we can use the results from \cite[\S \ 4 \& 5]{GiExact} to simplify their computations. Indeed, since these models are exact, these homotopy relations have particular descriptions due to \cite[Prop. 4.4]{GiExact} and \cite[Thm. 1.2.10 (i)]{HoveyBook}. Specifically, we have the following.

\begin{proposition}\label{prop:homotopy_category}
Let $(\mcX,\mcY)$ be a hereditary complete cotorsion pair in $\mcC$. Then, the homotopy category of $\mcM^{\rm GP}_{(\mcX,\mcY)}$, denoted ${\rm Ho}(\mcM^{\rm GP}_{(\mcX,\mcY)})$, is equivalent to the stable category
\[
(\mcGP_{(\mcX,\omega)} \cap \mcY) / \sim,
\]
where for two morphisms $f, g \colon X \to Y$ with $X, Y \in \mcGP_{(\mcX,\omega)} \cap \mcY$,
\[
\text{$f \sim g$ if, and only if, $f - g$ factors through an object in $\omega$.}
\]
The homotopy category ${\rm Ho}(\mcM^{\rm GI}_{(\mcX,\mcY)})$ has a dual description.
\end{proposition}

\begin{remark}
We note the triangulated structure of $\Ho(\mcM^{\rm GP}_{(\mcX,\mcY)})$, since the category $(\mcGP_{(\mcX,\omega)} \cap \mcY) / \sim$ is triangulated by Beligiannis' \cite[Thm. 2.11]{Bel}. The latter follows by the fact that $\mcGP_{(\mcX,\omega)} \cap \mcY$ is a Frobenius category (with $\omega$ as the class of projective-injective objects). Let us only show the existence of enough injective objects. Suppose we are given an object $C \in \mcGP_{(\mcX,\omega)} \cap \mcY$. From Proposition \ref{prop:relative_GP_are_absolute_GP}, we can find a monomorphism $\iota \colon C \rightarrowtail W$ with $W \in \omega$ and $\CoKer(\iota) \in \mcGP_{(\mcX,\omega)}$. Also, since $\mcY$ is a coresolving subcategory of $\mcC$, we have that $\CoKer(\iota) \in \mcGP_{(\mcX,\omega)} \cap \mcY$. Then, $\iota$ is clearly an admissible monomorphism in $\mcGP_{(\mcX,\omega)} \cap \mcY$. On the other hand, it is clear that $\Ext^{i}(\mcGP_{(\mcX,\omega)} \cap \mcY,\omega) = 0$ for every $i \geq 1$.
\end{remark}

In some cases, it is possible to rewrite the Frobenius category $\mcGP_{(\mcX,\omega)} \cap \mcY$ as a certain subcategory of Gorenstein objects relative to $\omega$, in a sense due to Christensen, Estrada and Thompson \cite{CETflat-cotorsion}. This will yield an alternative description of the homotopy category of the model structure $\mcM^{\rm GP}_{(\mcX,\mcY)}$.

Let $\mu$ be a subcategory of $\mcC$. Recall from \cite[Def. 1.1]{CETflat-cotorsion} that an object $C \in \mcC$ is  \emph{right $\mu$-Gorenstein} if $C = Z_0(M_\bullet)$ for some exact and $\Hom(-,\mu \cap \mu^\perp)$-acyclic complex $M_\bullet \in \Ch(\mu)$ with cycles in $\mu^\perp$. Such complex $M_\bullet$ is called \emph{right $\mu$-totally acyclic}. \emph{Left $\mu$-Gorenstein objects} and \emph{left $\mu$-totally acyclic complexes} are defined dually. In the case where $\mu$ is a self-orthogonal subcategory (that is, $\Ext^i(\mu,\mu) = 0$ for every $i \geq 1$) one has by \cite[Props. 1.3 \& 1.5]{CETflat-cotorsion} that a complex is right $\mu$-totally acyclic if, and only if, it is left $\mu$-totally acyclic, and so an object is left $\mu$-Gorenstein if, and only if, it is right $\mu$-Gorenstein.\footnote{The definition of self-orthogonal class considered in \cite{CETflat-cotorsion} is slightly different. Namely, the authors ask that $\Ext^1(\mu,\mu) = 0$.} \footnote{Note that $\mu$ is self-orthogonal if, and only if, $\mu = \mu \cap \mu^\perp = \mu \cap {}^\perp\mu$.} The category of such objects will be denoted by ${\rm Gor}_\mu$.

\begin{proposition}\label{prop:Gor_equality}
Let $(\mcX,\mcY)$ be a hereditary complete cotorsion pair in $\mcC$. Then, the containment $\mcGP_{(\mcX,\omega)} \cap \mcY \subseteq {\rm Gor}_\omega$ holds. Moreover, $\mcGP_{(\mcX,\omega)} \cap \mcY \supseteq {\rm Gor}_\omega$ also holds provided that every exact complex in $\Ch(\mcY)$ has cycles in $\mcY$.
\end{proposition}

\begin{proof}
The containment ($\supseteq$) is clear with the assumption that every exact complex of objects in $\mcY$ has cycles in $\mcY$. Now let us show ($\subseteq$) and take $Y \in \mcGP_{(\mcX,\omega)} \cap \mcY$. We construct a left and right totally $\omega$-acyclic exact complex $W_\bullet$ such that $Y = Z_0(W_\bullet)$. By \cite[Lem. 3.5]{Xu17}\footnote{This result also holds in abelian categories.}, we have that $Y \in {}^\perp\omega$ and that there exists a $\Hom(-,\omega)$-acyclic $\omega$-coresolution of $Y$, say
\[
\eta^+ \colon Y \rightarrowtail W^0 \to W^1 \to \cdots.
\]
Note that $\eta^+$ is also $\Hom(\omega,-)$-acyclic since it has cycles in $\mcY$ (being $\mcY$ coresolving). It is only left to show that there is also a $\Hom(\omega,-)$-acyclic and $\Hom(-,\omega)$-acyclic $\omega$-resolution
\[
\eta_{-} \colon \cdots \to W_1 \to W_0 \twoheadrightarrow Y.
\]
On the one hand, since $Y \in \mcGP_{(\mcX,\omega)}$, there is a $\Hom(-,\omega)$-acyclic short exact sequence $Y_1 \rightarrowtail X_1 \twoheadrightarrow Y$ with $X_1 \in \mcX$ and $Y_1 \in \mcGP_{(\mcX,\omega)}$. On the other hand, since $(\mcX,\mcY)$ is a complete cotorsion pair, there is a short exact sequence $Y_1 \rightarrowtail Y'_1 \twoheadrightarrow X'_1$ with $Y'_1 \in \mcY$ and $X'_1 \in \mcX$. Taking the pushout of $Y'_1 \leftarrow Y_1 \to X_1$ yields the following exact commutative diagram:
\[
\begin{tikzpicture}[description/.style={fill=white,inner sep=2pt}]
\matrix (m) [matrix of math nodes, row sep=3em, column sep=3em, text height=1.25ex, text depth=0.25ex]
{
Y_1 & X_1 & Y \\
Y'_1 & W_0 & Y \\
X'_1 & X'_1 & {} \\
};
\path[->]
(m-1-1)-- node[pos=0.5] {\footnotesize$\mbox{\bf po}$} (m-2-2)
;
\path[>->]
(m-1-1) edge (m-1-2) edge (m-2-1)
(m-2-1) edge (m-2-2)
(m-1-2) edge (m-2-2)
;
\path[->>]
(m-1-2) edge (m-1-3)
(m-2-2) edge (m-2-3)
(m-2-1) edge (m-3-1)
(m-2-2) edge (m-3-2)
;
\path[-,font=\scriptsize]
(m-1-3) edge [double, thick, double distance=2pt] (m-2-3)
(m-3-1) edge [double, thick, double distance=2pt] (m-3-2)
;
\end{tikzpicture}
\]
Note from this diagram that $W_0 \in \omega$ and $Y'_1 \in \mcGP_{(\mcX,\omega)} \cap \mcY$, that the central row is $\Hom(\omega,-)$-acyclic since $Y'_1 \in \mcY$, and $\Hom(-,\omega)$-acyclic since $Y \in \mcGP_{(\mcX,\omega)}$. So we can apply the previous procedure to $Y'_1$ in order to obtain a $\Hom(-,\omega)$-acyclic and $\Hom(\omega,-)$-acyclic short exact sequence $Y'_2 \rightarrowtail W_1 \twoheadrightarrow Y'_1$ such that $W_1 \in \omega$ and $Y'_2 \in \mcGP_{(\mcX,\omega)} \cap \mcY$. Continuing this process yields the $\Hom(-,\omega)$-acyclic and $\Hom(\omega,-)$-acyclic $\omega$-resolution $\eta_{-}$.
\end{proof}

A natural question that arises is wether ${\rm Gor}_\omega = \mcGP_{(\mcX,\omega)} \cap \mcY$ implies that every exact complex in $\mathsf{Ch}(\mcY)$ has cycles in $\mcY$. Below we give a counter-example for this.

\begin{example}
Let $R$ be a commutative noetherian ring with dualizing complex that is not Gorenstein (see Iyengar and Krause's \cite[\S \ 6]{IK06} for an example of such rings). We know by Enochs and Iacob's \cite[Coroll. 1]{EI15} that the class $\mcGI(R)$ of Gorenstein injective $R$-modules is the right half of a hereditary complete cotorsion pair $({}^\perp\mcGI(R),\mcGI(R))$. In this case, it is easy to note that $\omega = \mcI(R)$, and then ${\rm Gor}_\omega = \mcGI(R)$. That is, ${\rm Gor}_\omega = \mcY$. On the other hand, note that $\mcGP_{({}^\perp\mcGI(R),\mcI(R))}$ is the class of modules which are cycles of exact complexes in $\Ch({}^\perp\mcGI(R))$. The latter is in turn the whole category $\Mod(R)$, since every $R$-module has a ${}^\perp\mcGI(R)$-resolution and an injective (and so ${}^\perp\mcGI(R)$-)coresolution. Thus the equality ${\rm Gor}_\omega = \mcGP_{(\mcX,\omega)} \cap \mcY$ holds in this case.\footnote{Over an arbitrary ring $R$, it is still true that $({}^\perp\mcGI(R),\mcGI(R))$ is a hereditary complete cotorsion pair in $\mathsf{Mod}(R)$, due to \v{S}aroch and \v{S}t'ov\'{\i}\v{c}ek \cite[Thm. 4.6]{SS}. Hence, the equality ${\rm Gor}_\omega = \mcGP_{(\mcX,\omega)} \cap \mcY$ always holds for the choice $\mcX = {}^\perp\mcGI(R)$ and $\mcY = \mcGI(R)$.}

On the other hand, we assert that there exists an exact complex of Gorenstein injective modules with at least one cycle that is not Gorenstein injective. By \cite[Prop. 4]{EFI17} (proved by Fu, Iacob and the first author), the latter is equivalent to finding an exact complex of injective modules that is not ${\rm Hom}_R(\mcI(R),-)$-acyclic. Due to the assumptions on the ring $R$, such complexes exist by \cite[Coroll. 5.5]{IK06}.

As a consequence of \cite[Thm. 4.6]{SS}, the stable category $\mcGI(R) / \sim$ of Gorenstein injective $R$-modules can be obtained as the homotopy category of the Gorenstein injective model structure $(\Mod(R),{}^\perp\mcGI(R),\mcGI(R))$, which can also be obtained from Theorem \ref{theo:Xu_model} and Proposition \ref{prop:homotopy_category}.
\end{example}

As an application of Proposition \ref{prop:Gor_equality}, and following Example \ref{ex:GP_objects} (3), we obtain a different proof of \cite[Coroll. 2.6]{CET} (by Christensen, Thompson and the first author).

\begin{example}
Given a semi-separated noetherian scheme $X$, consider the classes $\mfGF(X)$ of Gorenstein flat sheaves and ${\rm Gor}_{\mathfrak{F}(X) \cap \mathfrak{C}(X)}$ of Gorenstein flat-cotorsion sheaves (see \cite[Def. 4.1]{CET}) in $\mathfrak{Qcoh}(X)$. It is known from \cite[Coroll. 4.2]{EnochsEstrada} (Enochs and the first author) and Saor\'{\i}n and \v{S}\v{t}ov\'{\i}\v{c}ek's \cite[Lem. 4.25]{SaorinStovicek} that $(\mfF(X),\mfC(X))$ is a hereditary complete cotorsion pair in $\mathfrak{Qcoh}(X)$. On the other hand, by \cite[Thm. 3.3]{CET} we know that every exact complex of cotorsion sheaves has cotorsion cycles\footnote{The cited result holds in the more general setting where $X$ is semi-separated and quasi-compact.}. Then by Proposition \ref{prop:Gor_equality} we have that
\[
{\rm Gor}_{\mfF(X) \cap \mfC(X)} = \mcGP_{(\mfF(X),\mfF(X) \cap \mfC(X))} \cap \mfC(X).
\]
Recall also from Example \ref{ex:GP_objects} (3) that $\mfGF(X) = \mcGP_{(\mfF(X),\mfF(X) \cap \mfC(X))}$. Then,
\[
{\rm Gor}_{\mfF(X) \cap \mfC(X)} = \mfGF(X) \cap \mfC(X).
\]
This equality is also obtained in \cite[Thm. 4.3]{CET}.

It then follows by Theorem \ref{theo:Xu_model} and Proposition \ref{prop:homotopy_category} that there exists a unique hereditary exact model structure on $\mfGF(X)^\wedge$ such that $\mfGF(X)$, $\mfGF(X)^\wedge \cap \mfC(X)$ and $\mfF(X)^\wedge$ are the classes of cofibrant, fibrant and trivial objects, and whose homotopy category is equivalent to the stable category ${\rm Gor}_{\mfF(X) \cap \mfC(X)} / \sim$. On the other hand, it is know from \cite[Thm. 2.5]{CET} the existence of a unique abelian model structure on the whole category $\mathfrak{Qcoh}(X)$ where $\mfGF(X)$ is the class of cofibrant objects and $\mfC(X)$ is the class of fibrant objects, whose homotopy category is also equivalent to ${\rm Gor}_{\mfF(X) \cap \mfC(X)} / \sim$. One advantage of applying Theorem \ref{theo:Xu_model} is that we obtain a simpler description of the trivial objects.
\end{example}

\begin{remark} \
\begin{enumerate}
\item It is known from \cite[Thm. 2.2]{CET} that $\mfGF(X)$ is a resolving subcategory of $\mathfrak{Qcoh}(X)$, provided that $X$ is noetherian and semi-separated. The proof given in \cite{CET} is \emph{local} in the sense that, over such schemes, the notion of Gorenstein flat sheaves is local (see \cite[Thm. 1.6]{CET}). From the equality $\mfGF(X) = \mcGP_{(\mfF(X),\mfF(X) \cap \mfC(X))}$ we have a \emph{non local} argument for the property of being resolving.

\item The affine case $X = {\rm Spec}(R)$ for the facts mentioned in the previous example is already covered in \cite[Coroll. 3.11]{SS}.
\end{enumerate}
\end{remark}


\section{Relative Gorenstein categories}\label{sec:RG-categories}

The purpose of this section is to propose a concept of relative Gorenstein categories as suitable settings to obtain hereditary abelian model structures whose (co)fibrant objects are the Gorenstein injective (resp., projective) objects relative to GI-admissible (resp., GP-admissible) pairs. These categories will also represent a generalization of the (absolute) Gorenstein categories in the sense specified in Definition \ref{def:Gorenstein_category} below.


\subsection*{Recalling Gorenstein categories}

We begin this section considering approaches to the notion of Gorenstein categories: one by  Beligiannis and Reiten \cite[Def. VII.2.1]{BR} motivated in one part by the study of Cohen-Macaulay rings, and the other one by Enochs, Garc\'ia Rozas and the first author \cite[Def. 2.18]{Tate}. The approach in \cite{BR} is slightly more general than \cite{Tate}, while the latter is more convenient when working in ambient Grothendieck categories (although, under some conditions, both definitions are equivalent - see Proposition \ref{prop:equivalence} below). We shall use the term ``\emph{Gorenstein category}'' accordingly to the terminology in \cite{BR}, and refer to the concept presented in \cite{Tate} as ``\emph{locally projectively finite (l.p.f.) Gorenstein category}''.

For any abelian category $\mcC$, recall that the \emph{finitistic projective} and the \emph{finitistic injective dimensions of $\mcC$} are defined as the values
\begin{align*}
\FPD(\mcC) & := \pd(\mcP^{< \infty}) & & \text{and} & \FID(\mcC) & := \id(\mcI^{< \infty}).
\end{align*}

\begin{definition}[Beligiannis and Reiten]\label{def:Gorenstein_category1}
A \textbf{Gorenstein category} is an abelian category $\mcC$ with enough projective and injective objects such that $\pd(\mcI) < \infty$ and $\id(\mcP) < \infty$. 
\end{definition}

\begin{definition}[Enochs, Estrada and Garc\'ia-Rozas]\label{def:Gorenstein_category}
We say that a Grothendieck category $\mcC$ is \textbf{locally projectively finite} (\textbf{l.p.f.}) \textbf{Gorenstein} if the following conditions are satisfied: 
\begin{enumerate}
\item[\Gone] For any $C \in \mcC$, $\pd(C) < \infty$ if, and only if, $\id(C) < \infty$. That is, $\mcP^{< \infty} = \mcI^{< \infty}$.

\item[\Gtwo] $\FPD(\mcC) < \infty$ and $\FID(\mcC) < \infty$.

\item[\Gthree] $\mcC$ has a generator with finite projective dimension.
\end{enumerate}
\end{definition}

\begin{remark}\label{rmk:Gorenstein_category} 
Gorenstein categories in the sense of Beligiannis and Reiten \cite[Def. VII.2.1]{BR} are related to the class of \emph{Cohen-Macaulay objects}, defined as $\mathcal{WGP}_{(\omega,\omega)}$ where $\omega$ is a self-orthogonal class in $\mcC$.
\begin{enumerate}
\item If $\mcC$ is an abelian category with enough projective objects, then note that $\mathcal{WGP}_{(\mathcal{P},\mathcal{P})} = \mcGP$. Moreover, by \cite[Prop. VII.1.3 (i)]{BR} one has the inequality
\begin{align*}
\FPD(\mcC) & \leq \id(\mcP) \leq \resdim_{\mathcal{WGP}_{(\mathcal{P},\mathcal{P})}}(\mcC) = \glGPD(\mcC).
\end{align*}
\end{enumerate} 
Suppose in addition that $\mcC$ has enough injective objects, then:
\begin{enumerate}
\setcounter{enumi}{1}
\item From \cite[Thm. VII.2.2]{BR} we can note that if $\mathcal{C}$ is Gorenstein then \Gone\, and \Gtwo\, hold.

\item From \cite[proof of Prop. VII.2.4]{BR}, we can also note that if $\mcC$ is AB3 and AB3${}^\ast$ (that is, products and coproducts exist in $\mcC$), then $\mcC$ is Gorenstein if, and only if, $\mcP \subseteq \mcI^{<\infty}$ and $\mcI \subseteq \mcP^{<\infty}$.
\end{enumerate}
\end{remark}

Definition \ref{def:Gorenstein_category} follows the spirit of \cite[Def. 2.18]{Tate}. In what follows, we recall the necessary terminology and results to see that Gorenstein categories are equivalent to the namesake concept in \cite{BR}. The understanding of this will be key in our path towards a relativization of Gorenstein categories.

We are now ready to show the following characterization of l.p.f. Gorenstein categories.

\begin{proposition}\label{prop:equivalence}
Let $\mcC$ be a Grothendieck category with enough projective objects. Then, the following conditions are equivalent:
\begin{enumerate}
\item[(a)] $\mcC$ is a l.p.f. Gorenstein category.

\item[(b)] $\glGPD(\mcC) < \infty$ and $\glGID(\mcC)< \infty$.

\item[(c)] $\mcC$ satisfies \Gone\, and \Gtwo.

\item[(d)] $\mcC$ is a Gorenstein category.
\end{enumerate}
\end{proposition}

\begin{proof} \
\begin{itemize}
\item (a) $\Leftrightarrow$ (b): Proved in \cite[Thm. 2.28]{Tate}.

\item (b) $\Rightarrow$ (c): Follows by Remark \ref{rmk:Gorenstein_category} (1) and its dual.

\item (c) $\Rightarrow$ (a): We have condition \Gthree \ from the hypothesis, and \Gtwo \ from Remark \ref{rmk:Gorenstein_category} (1) and its dual. Finally, since $\mcC$ has enough projective objects, $\id(\mcP) < \infty$, and $\mcI^{< \infty}$ is thick, we can note that $\mcP^{< \infty} = \mcP^\wedge \subseteq (\mcI^{< \infty})^\wedge = \mcI^{< \infty}$. The other containment follows similarly.

\item (c) $\Leftrightarrow$ (d): Follows by Remark \ref{rmk:Gorenstein_category} (keep in mind that Grothendieck categories are complete). 
\end{itemize}
\end{proof}


\subsection*{Finitistic and global dimensions relative to admissible pairs}

Before proposing a definition of relative Gorenstein category from Gorenstein objects relative to admissible pairs, let us study the interplay between finitistic and global dimensions relative to such objects. Given a GP-admissible pair $(\mcX,\mcY)$ in $\mcC$, recall from \cite[Defs. 4.16 \& 4.17]{BMS} that the \emph{finitistic Gorenstein projective dimension of $\mcC$ relative to $(\mcX,\mcY)$} is defined by
\begin{align*}
\FGPD_{(\mcX,\mcY)}(\mcC) & := \resdim_{\mcGP_{(\mcX,\mcY)}}(\mcGP^\wedge_{(\mcX,\mcY)}).
\end{align*}
The \emph{finitistic Gorenstein injective dimension of $\mcC$ relative to a GI-admissible pair} $(\mathcal{W,Z})$ is defined dually, and will be denoted by $\FGID_{(\mathcal{W,Z})}(\mcC)$.

Given $\mcZ$ a class of objects of $\mcC$, the \emph{finitistic projective} and \emph{injective dimensions of $\mcC$ relative to $\mcZ$} are defined by
\begin{align*}
\FPD_{\mcZ}(\mcC) & := \pd_{\mcZ}(\mcP^{< \infty}_{\mcZ}) & & \text{and} & \FID_{\mcZ}(\mcC) & := \id_{\mcZ}(\mcI^{< \infty}_{\mcZ}).
\end{align*}

\begin{proposition} \label{lpdGP}
Let $(\mcX,\mcY)$ be a GP-admissible pair in $\mcC$ such that $\omega$ is closed under direct summands. Then, the following assertions are equivalent:
\begin{enumerate}
\item[(a)] $\glGPD_{(\mcX,\mcY)}(\mcC) < \infty$.

\item[(b)] $\FPD_{\omega}(\mcC) < \infty$ and $\mcC = \mcGP^\wedge_{(\mcX,\mcY)}.$

\item[(c)] $\pd_{\omega}(\omega^\wedge) < \infty$ and $\mcC = \mcGP^\wedge_{(\mcX,\mcY)}.$
\end{enumerate}
Moreover, if any of these conditions holds true, we have that $\mcC =\mcP^{< \infty}_{\omega}$ and
\[
\glGPD_{(\mcX,\mcY)}(\mcC) = \FPD_{\omega}(\mcC) \leq \pd_{\omega}(\omega^\wedge) + 1.
\]
\end{proposition}

\begin{proof} \
\begin{itemize}
\item (a) $\Rightarrow$ (b): Set $n := \glGPD_{(\mcX,\mcY)}(\mcC) < \infty$, and so $\Gpd_{(\mcX,\mcY)}(C) \leq n$ for every $C \in \mcC$. Since $\omega$ is closed under direct summands and $\mcP^{< \infty}_{\omega} \subseteq \mcC = \mcGP^\wedge_{(\mcX,\mcY)}$, we have from \cite[Coroll. 4.33 (b)]{BMS} that
\[
\FPD_{\omega}(\mcC) = \FGPD_{(\mcX,\mcY)}(\mcC) = \glGPD_{(\mcX,\mcY)}(\mcC) < \infty.
\]

\item (b) $\Rightarrow$ (c): Suppose now that $m := \FPD_{\omega}(\mcC) < \infty$ and $\mcC = \mcGP^\wedge_{(\mcX,\mcY)}.$ Using \cite[Coroll. 4.33 (b)]{BMS} again, we have that $\mcP^{< \infty}_{\omega} = \mcGP^\wedge_{(\mcX,\mcY)}$. Since $\omega \subseteq \mcGP_{(\mcX,\mcY)}$, we have that $\omega^\wedge \subseteq \mcGP^\wedge_{(\mcX,\mcY)} = \mcP^{<\infty}_{\omega}$. Then,
\[
\pd_{\omega}(\omega^\wedge) \leq \FPD_{\omega}(\mcC) < \infty.
\]

\item (c) $\Rightarrow$ (a): Finally, suppose that $k := \pd_{\omega}(\omega^\wedge) < \infty$ and $\mcC = \mcGP^\wedge_{(\mcX,\mcY)}$. Let $C \in \mcC$ with $m := \Gpd_{(\mcX,\mcY)}(C) < \infty$. By Proposition \ref{coro:properties_relative_Gorenstein_projective} (4), we know that there exists a short exact sequence $K \rightarrowtail G \twoheadrightarrow C$ where $G \in \mcGP_{(\mcX,\mcY)}$ and $K \in \omega^\wedge$. By \cite[Coroll. 4.11 (a)]{BMS} we have that $\pd_{\omega}(C) = m$. On the other hand, $\pd_{\omega}(G) = 0$, and $\pd_{\omega}(K) \leq k$ by condition (c). Using a dimension shifting argument, we can note that $\Ext^i(C,W) = 0$ for every $i > k+1$ and $W \in \omega$. Thus, $\pd_{\omega}(C) \leq k + 1$, for every $C \in \mcC$.
\end{itemize}
\end{proof}

The following result is a consequence of \cite[Lem. 2.11]{BMS}.

\begin{proposition}\label{prop:weak_resdim_vs_pdim}
For a weak GP-admissible pair $(\mcX,\mcY)$ in $\mcC,$ the equality
\[
\resdim_{{}^{\perp}\mcY}(C) = \pd_{\mcY}(C)
\]
holds for every $C \in \mcC$. Moreover, $({}^{\perp}\mcY)^\wedge = \mcP^{< \infty}_{\mcY}$, and then
\[
\FPD_{\mcY}(\mcC) = \resdim_{{}^{\perp}\mcY}(({}^{\perp}\mcY)^\wedge).
\]
\end{proposition}

\begin{remark}\label{Pinf=pdfin}
If $\mcC$ is an abelian category with enough projective objects, and so $\mcP^\wedge = \mcP^{< \infty}$, the previous proposition implies that $\pd(C) = \resdim_\mcP(C)$, for any $C\in\mcC.$
\end{remark}

Note from Proposition \ref{prop:weak_resdim_vs_pdim} that we can obtain \cite[Prop. 1.11 \& Coroll. 1.12]{Cortes} (proved by Cort{\'e}s Izurdiaga, Guil Asensio and the first author), specified in part (1) of the following result.

\begin{corollary}\label{coro:weak_resdim_vs_pdim}
Let $(\mcX,\mcY)$ be a hereditary complete cotorsion pair in $\mcC$.
\begin{enumerate}
\item The equality
\[
\pd_{\mcY}(C) = \resdim_{\mcX}(C)
\]
holds for every $C \in \mcC$. Moreover, $\mcX^\wedge = \mcP^{< \infty}_{\mcY}$, and then
\[
\FPD_{\mcY}(\mcC) = \resdim_{\mcX}(\mcX^\wedge).
\]

\item If in addition $\mcX \cap \mcY = \mcP$, then
\[
\FPD(\mcC) = \FPD_{\mcY}(\mcC).
\]
Moreover, $\pd_{\mcX}(\mcX) < \infty$ if, and only if, $\mcX = \mcP$.
\end{enumerate}
\end{corollary}

\begin{proof}
We only prove part (2), as part (1) follows by Proposition \ref{prop:weak_resdim_vs_pdim}.

By part (1), we know that $\FPD_{\mcY}(\mcC) = \resdim_{\mcX}(\mcX^\wedge)$. On the other hand, by \cite[Coroll. 3.17 (b)]{BMS} we have that $\mcGP_{(\mcX,\mcY)} = \mcX$, and so
\[
\FPD_{\mcY}(\mcC) = \resdim_{\mcGP_{(\mcX,\mcY)}}(\mcGP^\wedge_{(\mcX,\mcY)}) = \FGPD_{(\mcX,\mcY)}(\mcC).
\]
Finally, $\FGPD_{(\mcX,\mcY)}(\mcC) = \FPD(\mcC)$ follows by \cite[Coroll. 4.27]{BMS}.

Regarding the equivalence, note that the ``if'' part is clear. For the ``only if'' part, suppose that $\pd_{\mcX}(\mcX) < \infty$. As in \cite[Thm. 2.2 (a)]{AM} (Angeleri-H{\"u}gel and the second author), we can note that
\[
\pd(C) = \text{max}\{ \pd_{\mcX}(C), \pd_{\mcY}(C) \}
\]
for every $C \in \mcC$. Since $\pd_{\mcY}(\mcX) = 0$, it follows $\pd(\mcX) = \pd_{\mcX}(\mcX) < \infty$. On the other hand, we have by the dual of part (1) that
\[
\coresdim_{\mcY}(C) = \id_{\mcX}(C) \leq \id_{\mcX}(\mcC) = \pd(\mcX) < \infty
\]
for every $C \in \mcC$. Then, $\mcC = \mcY^\vee$. Finally, let $X \in \mcX \subseteq \mcY^\vee$. Since $(\mcP,\mcY)$ is a right Frobenius pair in $\mcC$, we have by the dual of \cite[Thm. 2.8]{BMPS19} that there exists a short exact sequence $Y \rightarrowtail H \twoheadrightarrow X$ where $Y \in \mcY$ and $H \in \mcP^\vee = \mcP$. This sequence splits, and hence $X \in \mcP$.
\end{proof}

A stronger version of Corollary \ref{coro:weak_resdim_vs_pdim} (and also a refinement of \cite[Thm. 2.10]{BMPS19}) holds for certain GP-admissible pairs. Before giving a precise statement, let us first show the following preliminary result.

\begin{lemma}\label{lemYperpX}
Let $(\mcX,\mcY)$ be a GP-admissible pair in $\mcC$ such that $\mcX$ is closed under direct summands. Then,
\[
{}^{\perp}\mcY \cap \mcX^\wedge_n = \mcX,
\]
for every $n \geq 0$.
\end{lemma}

\begin{proof}
We use induction on $n$. The initial case $n = 0$ is immediate since $\mcX \subseteq {}^{\perp}\mcY$. Now suppose that ${}^{\perp}\mcY \cap \mcX^\wedge_{n-1} \subseteq \mcX$ for $n \geq 1$, and let $C \in {}^{\perp}\mcY \cap \mcX^\wedge_n$. By \cite[Thm. 2.8]{BMPS19} there exists a short exact sequence $K \rightarrowtail X \twoheadrightarrow C$ with $X \in \mcX$ and $K \in \omega^\wedge_{n-1}$. Since ${}^{\perp}\mcY$ is resolving and $X, C \in {}^{\perp}\mcY$, we have that $K \in {}^{\perp}\mcY \cap \mcX^\wedge_{n-1}$, and so $K \in \mcX$ by the induction hypothesis. It follows that $\resdim_{\mcX}(C) \leq 1$, and again by \cite[Thm. 2.8]{BMPS19} there exists a short exact sequence $W \rightarrowtail X' \to C$ with $X' \in \mcX$ and $W \in \omega$. Since $C \in {}^{\perp}\mcY \subseteq {}^{\perp}\omega$, the previous sequence splits, and so $C \in \mcX$ since $\mcX$ is closed under direct summands.
\end{proof}

\begin{proposition}\label{prop:resdim_vs_pdim}
Let $(\mcX,\mcY)$ be a GP-admissible pair in $\mcC$ such that $\mcX$ is closed under direct summands. Then, $\mcX^\wedge\subseteq \mcP^{< \infty}_{\mcY}$ and
\begin{align}
\resdim_{\mcX}(C) & = \pd_{\mcY}(C) \label{eqn:resdim_vs_pdim}
\end{align}
for every $C \in \mcX^\wedge$. Moreover, $\mcC=\mcP^{< \infty}_{\mcY}$ and
\[
\FPD_{\mcY}(\mcC) = \resdim_{\mcX}(\mcC)
\]
if  $\mcC = \mcX^\wedge$.
\end{proposition}

\begin{proof}
The inequality $\pd_{\mcY}(C) \leq \resdim_{\mcX}(C)$ holds for every $C \in \mcC$, by the dual of \cite[Prop. 2.10 (a)]{BMS}. In particular, we get that $\mcX^\wedge\subseteq \mcP^{< \infty}_{\mcY}$. Now let $C \in \mcX^\wedge$ with $\resdim_{\mcX}(C) = m$. By the previous inequality, we have that $\pd_{\mcY}(C) = n < \infty$. Let us show that $m \leq n$. If $n = 0$, we have that $C \in {}^\perp\mcY \cap \mcX^\wedge = \mcX$ by Lemma \ref{lemYperpX}, and then $m = 0$. Now for the case $n \geq 1$, there is an exact sequence
\[
K \rightarrowtail X_{n-1} \to \cdots \to X_1 \to X_0 \twoheadrightarrow C
\]
with $X_0, X_1, \dots, X_{n-1} \in \mcX$ and $K \in \mcX^\wedge$. By dimension shifting, we can note that $K \in {}^\perp\mcY \cap \mcX^\wedge = \mcX$, and hence $m \leq n$.
\end{proof}

\begin{corollary}\label{glGPD=id}
Let $(\mcX,\mcY)$ be a GP-admissible pair in $\mcC$.
\begin{enumerate}
\item For every $C\in\mcGP^\wedge_{(\mcX,\mcY)}$, the equality
\begin{align}
\pd_\omega(C) & = \pd_\mcY(C) = \Gpd_{(\mcX,\mcY)}(C) \label{eqn1:glGPD=id}
\end{align}
holds\footnote{The first equality $\pd_\omega(C) = \pd_\mcY(C)$ was originally proved in \cite[Coroll. 4.15 (b1)]{BMS} under the additional assumption that $\omega$ is closed under direct summands, and the second one, $\pd_\mcY(C) = \Gpd_{(\mcX,\mcY)}(C)$, also in \cite[Coroll. 4.15 (a1)]{BMS} provided that $\mcY$ has the same closure property.}. Moreover,
\begin{align}
\glGPD_{(\mcX,\mcY)}(\mcC) & = \id(\omega) = \id(\mcY) \label{eqn2:glGPD=id}
\end{align}
if $\mcC = \mcGP^\wedge_{(\mcX,\mcY)}$.

\item Consider the following conditions:
\begin{enumerate}
\item[(a)] $\glGPD_{(\mcX,\mcY)}(\mcC) < \infty$.

\item[(b)] $\mcC = \mcGP^\wedge_{(\mcX,\mcY)}$ and $\id(\omega) < \infty$.

\item[(c)] $\mcC = \mcGP^\wedge_{(\mcX,\mcY)}$ and $\pd_\omega(\omega^\wedge) < \infty$.

\item[(d)] $(\mcGP_{(\mcX,\mcY)}, \omega^\wedge)$ is a hereditary complete cotorsion pair in $\mcC$ and $\id(\omega) < \infty$.
\end{enumerate}
Then, the implications (a) $\Leftrightarrow$ (b) $\Rightarrow$ (c) and (d) $\Rightarrow$ (b) are true. If in addition, $\omega$ is closed under direct summands, then all of the conditions are equivalent.
\end{enumerate}
\end{corollary}

\begin{proof} \
\begin{enumerate}
\item The equality \eqref{eqn2:glGPD=id} is a direct consequence of \eqref{eqn1:glGPD=id}, which in turn follows by applying Proposition \ref{prop:resdim_vs_pdim} to the GP-admissible pairs $(\mcGP_{(\mcX,\mcY)},\mcY)$ and $(\mcGP_{(\mcX,\mcY)},\omega)$ (see \cite[Thm. 3.34 (a)]{BMS}).

\item The equivalence $(a)\Leftrightarrow(b)$ follows from part (1), while the implications $(b)\Rightarrow(c)$ and $(d)\Rightarrow (b)$ are straightforward. Now assume that $\omega$ is closed under direct summands. Then, the implication $(c)\Rightarrow(a)$ follows from Proposition \ref{lpdGP}, while $(a)\Rightarrow (d)$ can be obtained from \cite[Coroll. 5.2 (c)]{BMS}.
\end{enumerate}
\end{proof}

We close this section giving a criterion for the finiteness of $\glGPD_{(\mcX,\mcY)}(\mcC)$ and $\glGID_{(\mcZ,\mcW)}(\mcC)$ in the case $\mcC$ has enough projective and injective objects.

\begin{proposition}\label{prop:big_diagram}
Let $\mcC$ be an abelian category with enough projective objects. If $(\mcX,\mcY)$ is a GP-admissible pair for which $\omega$ is closed under direct summands and $\mcP \subseteq \omega$, and $\nu$ is a relative cogenerator in $\mcC$ which is closed under direct summands and $\nu \subseteq \mcP^{< \infty}$, then the equality
\begin{align}
\glGPD_{(\mcX,\mcY)}(\mcC) & = \max \{ \pd(\nu), \id(\mcY) \} \label{eqn:Gmax}
\end{align}
holds in the following cases:
\begin{enumerate}
\item $\pd(\nu) = \infty$.
\item $\pd(\nu) < \infty$ and $\id(\mcY) < \infty$.
\end{enumerate}
Moreover, in case (2), we have that
\[
\pd_{\omega}(\nu) = \pd(\nu) = \Gpd_{(\mcX,\mcY)}(\nu) \leq \glGPD_{(\mcX,\mcY)}(\mcC) = \id(\omega) = \id(\mcY).
\]
\end{proposition}

\begin{proof}
We start by noticing from \cite[Coroll. 4.15 (b4)]{BMS} that
\[
\pd_{\omega}(\nu) \leq \Gpd_{(\mcX,\mcY)}(\nu),
\]
where the equality holds in the case where $\Gpd_{(\mcX,\mcY)}(\nu)$ is finite (see \cite[Coroll. 4.15 (b1)]{BMS}). Now since $\nu \subseteq \mcP^{<\infty}$, $\mcP \subseteq \omega$ and $\mcC$ has enough projective objects, we have from \cite[dual of Coroll. 2.2]{MS} (proved by Saenz and the second author) \footnote{The proof, given in the category of modules over a ring, also holds for abelian categories with enough projective objects.} that $\pd(\nu) = \pd_{\omega}(\nu)$. Moreover, note that $\pd(\nu) \leq \glGPD_{(\mcX,\mcY)}(\mcC)$, and hence the equality
\eqref{eqn:Gmax} holds in case (1).

Now let us focus on case (2). Let $n = \pd(\nu)$ and $m = \id(\mcY)$. In particular, $\nu \subseteq \mcGP_{(\mcX,\mcY)}^\wedge$ and so by \cite[Coroll. 4.15 (b1)]{BMS} we have
\[
\pd(\nu) = \pd_{\omega}(\nu) = \Gpd_{(\mcX,\mcY)}(\nu) \leq \glGPD_{(\mcX,\mcY)}(\mcC).
\]
Note also that
\[
\id(\mcY) \leq \glGPD_{(\mcX,\mcY)}(\mcC)
\]
from the definition of the $(\mcX,\mcY)$-Gorenstein projective dimension. Thus,
\[
\max\{n,m\} \leq \glGPD_{(\mcX,\mcY)}(\mcC).
\]
It remains to show that $\glGPD_{(\mcX,\mcY)}(\mcC) \leq \max\{n,m\}$. So let $C \in \mcC$ and consider a $\nu$-coresolution of $C$, say
\[
C \rightarrowtail V^0 \to V^1 \to V^2 \to \cdots,
\]
which can be constructed since $\nu$ is a relative cogenerator in $\mcC$. In order to show that $\Gpd_{(\mcX,\mcY)}(C) \leq \max\{n,m\}$, the idea is to consider projective resolutions of $C$ and of each $C^k := \Ker(V^k \to V^{k+1})$ with $k \geq 1$, and then to construct projective resolutions of each $V^k$ using the horseshoe lemma so that it will be possible to deduce the finiteness of $\Gpd_{(\mcX,\mcY)}(C)$. Indeed, from the assumptions we can construct the following commutative diagram with exact rows and columns:
\[
\begin{tikzpicture}[description/.style={fill=white,inner sep=2pt}]
\matrix (m) [matrix of math nodes, row sep=1.25em, column sep=1.5em, text height=1.25ex, text depth=0.25ex]
{
{} & Q^{-1}_n & {} & W^0_n & {} & {} & W^1_n & {} & \cdots \\
K^{-1}_n & {} & K^0_n & {} & {} & K^1_n & {} & \cdots & {} \\
{} & {} & {} & {} & \widetilde{K}^0_n \\
{} & Q^{-1}_{n-1} & {} & W^0_{n-1} & {} & {} & W^1_{n-1} & {} & \cdots \\
{} & \vdots & {} & \vdots & {} & {} & \vdots \\
{} & Q^{-1}_{1} & {} & W^0_{1} & {} & {} & W^1_{1} & {} & \cdots \\
{} & {} & {} & {} & P^0_1 \\
K^{-1}_0 & {} & K^0_0 & {} & {} & K^1_0 & \cdots & {} & {} \\
{} & {} & {} & {} & \widetilde{K}^0_0 \\
{} & Q^{-1}_{0} & {} & W^0_{0} & {} & {} & W^1_{0} & {} & W^2_0 & \cdots \\
{} & {} & {} & {} & P^0_0 & {} & {} & P^1_0 & {} & \\
{} & C & {} & V^0 & {} & {} & V^1 & {} & V^2 & \cdots \\
{} & {} & {} & {} & C^1 & {} & {} & C^2 \\
};
\path[>->]
(m-1-2) edge (m-1-4)
(m-1-4) edge (m-4-4)
(m-1-2) edge (m-4-2)
(m-2-1) edge (m-4-2)
(m-2-3) edge (m-4-4)
(m-3-5) edge (m-2-6)
(m-2-6) edge (m-4-7)
(m-6-2) edge (m-6-4)
(m-7-5) edge (m-6-7)
(m-8-6) edge (m-10-7)
(m-10-2) edge (m-10-4)
(m-8-1) edge (m-10-2)
(m-8-3) edge (m-10-4)
(m-13-5) edge (m-12-7)
(m-13-8) edge (m-12-9)
(m-11-5) edge (m-10-7)
(m-11-8) edge (m-10-9)
(m-9-5) edge (m-8-6)
;
\path[->]
(m-1-7) edge (m-4-7)
(m-4-2) edge (m-5-2)
(m-4-4) edge (m-5-4)
(m-4-7) edge (m-5-7)
(m-4-2) edge (m-4-4)
(m-4-4) edge (m-4-7)
(m-5-2) edge (m-6-2)
(m-5-4) edge (m-6-4)
(m-5-7) edge (m-6-7)
(m-6-4) edge (m-6-7)
(m-6-2) edge (m-10-2)
(m-6-4) edge (m-10-4)
(m-10-4) edge (m-10-7)
(m-10-7) edge (m-10-9)
(m-12-4) edge (m-12-7)
(m-12-7) edge (m-12-9)
(m-10-9) edge (m-10-10)
(m-12-9) edge (m-12-10)
(m-8-6) edge (m-8-7)
;
\path[->>]
(m-1-2) edge (m-2-1)
(m-1-4) edge (m-2-3)
(m-1-7) edge (m-2-6)
(m-6-2) edge (m-8-1)
(m-6-4) edge (m-8-3)
(m-6-4) edge (m-7-5)
(m-10-2) edge (m-12-2)
(m-10-4) edge (m-12-4)
(m-10-9) edge (m-12-9)
(m-12-4) edge (m-13-5)
(m-12-7) edge (m-13-8)
(m-10-4) edge (m-11-5)
(m-10-7) edge (m-11-8)
;
\path[-latex]
(m-2-3) edge [-,line width=6pt,draw=white] (m-2-6)
(m-2-1) edge [-,line width=6pt,draw=white] (m-2-3)
(m-2-3) edge [-,line width=6pt,draw=white] (m-3-5)
(m-8-1) edge [-,line width=6pt,draw=white] (m-8-3)
(m-8-3) edge [-,line width=6pt,draw=white] (m-8-6)
(m-9-5) edge [-,line width=6pt,draw=white] (m-11-5)
(m-8-3) edge [-,line width=6pt,draw=white] (m-9-5)
(m-11-5) edge [-,line width=6pt,draw=white] (m-13-5)
(m-11-8) edge [-,line width=6pt,draw=white] (m-13-8)
(m-2-6) edge [-,line width=6pt,draw=white] (m-2-8)
;
\path[->]
(m-1-4) edge (m-1-7)
(m-1-7) edge (m-1-9)
(m-2-3) edge (m-2-6)
(m-2-6) edge (m-2-8)
(m-4-7) edge (m-4-9)
(m-6-7) edge (m-6-9)
(m-8-3) edge (m-8-6)
;
\path[>->]
(m-2-1) edge (m-2-3)
(m-8-1) edge (m-8-3)
(m-9-5) edge (m-11-5)
(m-12-2) edge (m-12-4)
;
\path[->>]
(m-2-3) edge (m-3-5)
(m-8-3) edge (m-9-5)
(m-11-5) edge (m-13-5)
(m-10-7) edge (m-12-7)
(m-11-8) edge (m-13-8)
;
\path[-latex]
(m-7-5) edge [-,line width=6pt,draw=white] (m-9-5)
;
\path[->>]
(m-7-5) edge (m-9-5)
;
\end{tikzpicture}
\]
where all of the $P^j_i$, $W^j_i$ and $Q^{-1}_i$ are projective, and
\begin{align*}
W^0_0 & = Q^{-1}_0 \oplus P^0_0, & W^1_0 & = P^0_0 \oplus P^1_0, & & \dots & W^0_1 & = Q^{-1}_1 \oplus P^0_1, & \dots
\end{align*}
Since $\pd(\nu) = n$, we have that $K^j_n \in \mathcal{P}$ for every $j \geq 0$, and so
\[
\begin{tikzpicture}[description/.style={fill=white,inner sep=2pt}]
\matrix (m) [matrix of math nodes, row sep=1.5em, column sep=1.5em, text height=1.25ex, text depth=0.25ex]
{
K^{-1}_n & & K^0_n & & K^1_n & & K^2_n & & \cdots \\
 & & & \widetilde{K}^0_n & & \widetilde{K}^1_n \\
};
\path[>->]
(m-1-1) edge (m-1-3)
(m-2-4) edge (m-1-5)
(m-2-6) edge (m-1-7)
;
\path[->]
(m-1-3) edge (m-1-5)
(m-1-5) edge (m-1-7)
(m-1-7) edge (m-1-9)
;
\path[->>]
(m-1-3) edge (m-2-4)
(m-1-5) edge (m-2-6)
;
\end{tikzpicture}
\]
is a projective coresolution of $K^{-1}_n$. Moreover, for each $j \geq 0$ we have by construction the exact sequence
\begin{align}\label{eqn:partial_projective_C}
\widetilde{K}^j_n \rightarrowtail P^j_{n-1} \to \cdots \to P^j_0 \twoheadrightarrow C^{j+1}
\end{align}
where $P^j_i \in \mathcal{P}$ for every $0 \leq i \leq n-1$. In order to complete the proof, let us split it into two cases, namely, $m \leq n$ and $m > n$.
\begin{itemize}
\item Suppose $m \leq n$: In this case, we assert that $K^{-1}_n \in \mathcal{GP}_{(\mathcal{P,Y})} \subseteq \mathcal{GP}_{(\mathcal{X,Y})}$. By \cite[Prop. 3.16]{BMS}, it suffices to show that $K^{-1}_n \in \mathcal{WGP}_{(\mathcal{P,Y})}$. For, we need to check that $K^{-1}_n$ and each $\widetilde{K}^j_n$ belong to ${}^\perp\mcY$ in the constructed projective coresolution of $K^{-1}_n$. So let $Y \in \mcY$ and consider the partial projective resolution \eqref{eqn:partial_projective_C} of $C^{j+1}$. By dimension shifting, we have that $\Ext^{n+i}(C^{j+1},Y) \cong \Ext^{i}(\widetilde{K}^j_n,Y)$ for every $i \geq 1$. On the other hand, $\Ext^{n+i}(C^{j+1},Y) = 0$ since $\id(\mcY) = m \leq n$, and so $\widetilde{K}^j_n \in {}^\perp\mcY$. In a similar way, we also have that $K^{-1}_n \in {}^\perp\mcY$. Therefore, in the exact sequence
\begin{align*}
K^{-1}_n \rightarrowtail Q^{-1}_{n-1} \to \cdots \to Q^{-1}_0 \twoheadrightarrow C
\end{align*}
we have $K^{-1}_n \in \mathcal{WGP}_{(\mathcal{P,Y})}$, and so $\Gpd_{(\mathcal{X,Y})}(C) \leq n = \max\{m,n\}$.

\item Suppose $n < m$: In this case, it suffices to extend the construction of the previous big diagram to the $m$-th step. For every $n \leq i \leq m-1$ and every $j \geq 0$, we have a short exact sequence $K^j_{i+1} \rightarrowtail W^j_i \twoheadrightarrow K^j_i$ with $W^j_i \in \mathcal{P}$. Then, it will be possible to obtain an exact sequence
\[
\begin{tikzpicture}[description/.style={fill=white,inner sep=2pt}]
\matrix (m) [matrix of math nodes, row sep=1.5em, column sep=1.5em, text height=1.25ex, text depth=0.25ex]
{
K^{-1}_m & & K^0_m & & K^1_m & & K^2_m & & \cdots \\
 & & & \widetilde{K}^0_m & & \widetilde{K}^1_m \\
};
\path[>->]
(m-1-1) edge (m-1-3)
(m-2-4) edge (m-1-5)
(m-2-6) edge (m-1-7)
;
\path[->]
(m-1-3) edge (m-1-5)
(m-1-5) edge (m-1-7)
(m-1-7) edge (m-1-9)
;
\path[->>]
(m-1-3) edge (m-2-4)
(m-1-5) edge (m-2-6)
;
\end{tikzpicture}.
\]
As in the previous case, we have that $K^{-1}_m \in {}^\perp\mcY$ and $\widetilde{K}^j_m \in {}^\perp\mcY$. Hence, $\Gpd_{(\mathcal{X,Y})}(C) \leq m = \max\{m,n\}$.
\end{itemize}
\end{proof}

\begin{theorem}\label{thm:impliesBR}
Let $\mcC$ be an abelian category with enough projective and injective objects, equipped with a GP-admissible pair $(\mcX,\mcY)$ and a GI-admissible pair $(\mcW,\mcZ)$, such that:
\begin{itemize}
\item $\omega$ and $\nu$ are closed under direct summands,
\item $\mathcal{P} \subseteq \omega$ and $\mcI \subseteq \nu$.
\end{itemize}
Then,
\begin{align*}
\pd(\mcW) & < \infty & & \text{and} & \id(\mcY) & < \infty,
\end{align*}
if, and only if,
\begin{align*}
\glGPD_{(\mcX,\mcY)}(\mcC) & < \infty & & \text{and} & \glGID_{(\mcW,\mcZ)}(\mcC) & < \infty.
\end{align*}
Moreover, if any of these conditions holds, we have that
\begin{align*}
\Gpd_{(\mcX,\mcY)}(\nu) & = \Gid_{(\mcW,\mcZ)}(\omega) = \pd_{\omega}(\nu) = \pd(\nu) = \id(\omega) \\
& = \id(\mcY) = \pd(\mcW) = \glGPD_{(\mcX,\mcY)}(\mcC) = \glGID_{(\mcW,\mcZ)}(\mcC).
\end{align*}
\end{theorem}

\begin{proof}
Let us suppose first that $\pd(\mcW) < \infty$ and $\id(\mcY) < \infty$. By Proposition \ref{prop:big_diagram} and its dual, we have that
\begin{align*}
\pd_{\omega}(\nu) & = \pd(\nu) = \Gpd_{(\mcX,\mcY)}(\nu) \leq \glGPD_{(\mcX,\mcY)}(\mcC) = \id(\omega) = \id(\mcY), \\
\id_{\nu}(\omega) & = \id(\omega) = \Gid_{(\mcW,\mcZ)}(\omega) \leq \glGID_{(\mcW,\mcZ)}(\mcC) = \pd(\nu) = \pd(\mcW),
\end{align*}
where $\pd_{\omega}(\nu) = \id_{\nu}(\omega)$. Therefore, the stated equality follows.

The remaining implication follows from Corollary \ref{glGPD=id} and its dual.
\end{proof}

\begin{example}\label{ex:examples_of_global_Gorenstein_dimensions}
Let us comment on the finiteness of the Gorenstein global dimensions $\glGPD_{(\mcX,\mcY)}(\mcC)$ and $\glGID_{(\mcW,\mcZ)}(\mcC)$ for some well known cases in the category $\Mod(R)$ of $R$-modules.
\begin{enumerate}
\item For the admissible pairs $(\mathcal{P}(R),\mathcal{P}(R))$ and $(\mcI(R),\mcI(R))$, and their corresponding classes of relative Gorenstein projective and injective $R$-modules, that is, $\mcGP(R)$ and $\mcGI(R)$, it is known from Bennis and Mahdou's \cite[Thm. 1.1]{BennisMahdou} that $\glGPD_{(\mathcal{P}(R),\mathcal{P}(R))}(\Mod(R))$ and $\glGID_{(\mcI(R),\mcI(R))}(\Mod(R))$ coincide, and their common value, which we denote by ${\rm gl.Gdim}(R)$, is known as the \textbf{global Gorenstein dimension}. It then follows from the previous theorem that ${\rm gl.Gdim}(R) < \infty$ if, and only if, $\pd(\mcI(R)) < \infty$ and $\id(\mcP(R)) < \infty$ (that is, if $R$ is a left Gorenstein ring \cite[Def. VII.2.5]{BR}), obtaining thus Proposition \ref{prop:equivalence} in the context of $R$-modules. In this case, ${\rm gl.Gdim}(R) = \pd(\mcI(R)) = \id(\mcP(R))$.

\item Consider the GP-admissible pair $(\mathcal{P}(R),\mathcal{F}(R))$ and the GI-admissible pair $({\rm FP}\text{-}\mcI(R),\mcI(R))$, where ${\rm FP}\text{-}\mcI(R)$ denotes the class of \textbf{FP-injective} (a.k.a. \textbf{absolutely pure}) $R$-modules. Recall that such modules are defined as the right orthogonal complement $(\mathcal{FP}(R))^{\perp_1}$, where $\mathcal{FP}(R)$ denotes the class of finitely presented $R$-modules. From Example \ref{ex:GP_objects} (2), we know that $\mathcal{GP}_{(\mathcal{P}(R),\mathcal{F}(R))}$ is the class $\mathcal{DP}(R)$ of Ding projective $R$-modules. Dually, $\mcGI_{({\rm FP}\text{-}\mcI(R),\mcI(R))}$ is the class of Ding injective $R$-modules, which we denote by $\mathcal{DI}(R)$. Let
\begin{align*}
{\rm gl.DPD}(R) & := {\rm gl.GPD}_{(\mcP(R),\mcF(R))}(\Mod(R)), \text{ and} \\
{\rm gl.DID}(R) & := \glGID_{({\rm FP}\text{-}\mcI(R),\mcI(R))}(\Mod(R))
\end{align*}
denote the global Ding projective and Ding injective dimensions of $R$.

In 2019, Wang and Di \cite{WangDi19} proved that ${\rm gl.DPD}(R)  = {\rm gl.DID}(R)$ holds for any ring $R$. We give a proof for the reader's convenience. First, note that since $\mathcal{DP}(R) \subseteq \mathcal{GP}(R)$ and $\mathcal{DI}(R) \subseteq \mathcal{GI}(R)$, we clearly have that $
{\rm gl.Gdim}(R) \leq {\rm min}\{ {\rm gl.DPD}(R), {\rm gl.DID}(R) \}$. So if ${\rm gl.Gdim}(R) = \infty$, then ${\rm gl.DPD}(R) = {\rm gl.DID}(R)$ is clear. Let us then assume that ${\rm gl.Gdim}(R) = n < \infty$. It follows from \cite[Coroll. 2.7]{BennisMahdou} that $\pd(\mathcal{F}(R)) \leq n$ and ${\rm fd}(\mathcal{I}(R)) \leq n$ (where ${\rm fd}(-)$ denotes the flat dimension). The latter, along with \cite[Lem. 9.1.4]{EJ}, imply that ${\rm fd}({\rm FP}\text{-}\mcI(R)) \leq n$. Again, since ${\rm gl.Gdim}(R) = n$, we obtain that ${\rm id}({\rm FP}\text{-}\mcI(R)) \leq n$. From ${\rm max}\{ \pd(\mathcal{F}(R)), {\rm id}({\rm FP}\text{-}\mcI(R)) \} \leq n$, one can easily deduce that $\mathcal{DP}(R) = \mathcal{GP}(R)$ and $\mathcal{DI}(R) = \mathcal{GI}(R)$. Therefore, ${\rm gl.DPD}(R) = {\rm gl.DID}(R)$ in this case. We refer to the common value between ${\rm gl.DPD}(R)$ and ${\rm gl.DID}(R)$ as the (left) {\bf global Ding dimension}, which we shall denote by ${\rm gl.Ddim}(R)$. Note that the previous argument shows that ${\rm gl.Ddim}(R) = {\rm gl.Gdim}(R)$.

Combining Theorem \ref{thm:impliesBR}, Corollary \ref{glGPD=id} (2) and its dual, we have that the following are equivalent:
\begin{enumerate}
\item[(a)] ${\rm gl.Ddim}(R) < \infty$.
\item[(b)] $\id(\mathcal{F}(R)) < \infty$ and $\pd({\rm FP}\text{-}\mcI(R)) < \infty$.
\item[(c)] $\Mod(R) = \mathcal{DP}(R)^\wedge = \mathcal{DI}(R)^\vee$, $\FPD(R) < \infty$ and $\FID(R) < \infty$.
\item[(d)] $(\mathcal{DP}(R),\mcP(R)^\wedge)$ and $(\mcI(R)^\vee,\mathcal{DI}(R))$ are hereditary complete cotorsion pairs in $\Mod(R)$, and $R$ is a left Gorenstein ring.
\end{enumerate}
Notice that $\FPD(R) = \pd_{\mcP(R)}(\mcP(R)^\wedge)$ and $\FID(R) = \id_{\mcI(R)}(\mcI(R)^\vee)$ by \cite[Coroll. 2.2]{MendozaSaenz} and its dual. Under any of these finiteness conditions, we have that
\[
{\rm gl.Gdim}(R) = {\rm gl.Ddim}(R) = \id(\mathcal{F}(R)) = \pd({\rm FP}\text{-}\mcI(R)).
\]
An example of a ring $R$ for which ${\rm gl.Ddim}(R) = \infty$ can be found in \cite[Ex. 3.3]{Wang17}.

\item Let $\mathcal{FP}_\infty(R)$ denote the class of $R$-modules admitting a projective resolution by finitely generated $R$-modules. Recall from Bravo, Gillespie and Hovey's \cite[Def. 2.6]{BravoGillespieHovey} that $M \in \Mod(R)$ is \textbf{absolutely clean} (resp., \textbf{level}) if $\Ext^1(\mathcal{FP}_\infty(R),M) = 0$ (resp., $\Tor_1(\mathcal{FP}_\infty(R^{\rm op}),M) = 0$). Let us denote the classes of absolutely clean and level $R$-modules by $\mathcal{AC}(R)$ and $\mathcal{LV}(R)$, respectively. One can note that $(\mcP(R),\mathcal{LV}(R))$ is GP-admissible and $(\mathcal{AC}(R),\mcI(R))$ is GI-admissible. The corresponding relative Gorenstein $R$-modules, known as Gorenstein AC-projective and Gorenstein AC-injective, form two classes which we denote by $\mathcal{GP}_{\rm AC}(R)$ and $\mathcal{GI}_{\rm AC}(R)$, respectively (see \cite[\S \ 5 \& 6]{BravoGillespieHovey}). On the other hand, set
\begin{align*}
{\rm gl.GPD}_{\rm AC}(R) & := {\rm gl.GPD}_{(\mcP(R),\mathcal{LV}(R))}(\Mod(R)), \text{ and} \\
{\rm gl.GID}_{\rm AC}(R) & := \glGID_{(\mathcal{AC}(R),\mcI(R))}(\Mod(R)).
\end{align*}
We are not aware if the equality ${\rm gl.GPD}_{\rm AC}(R) = {\rm gl.GID}_{\rm AC}(R)$ holds for any ring $R$. As in the previous example, we can note that the following conditions are equivalent:
\begin{enumerate}
\item[(a)] ${\rm gl.GPD}_{\rm AC}(R) < \infty$ and ${\rm gl.GID}_{\rm AC}(R) < \infty$.
\item[(b)] $\id(\mathcal{LV}(R)) < \infty$ and $\pd(\mathcal{AC}(R)) < \infty$.
\item[(c)] $\Mod(R) = \mathcal{GP}_{\rm AC}(R)^\wedge = \mathcal{GI}_{\rm AC}(R)^\vee$, $\FPD(R) < \infty$ and $\FID(R) < \infty$.
\item[(d)] $(\mathcal{GP}_{\rm AC}(R),\mcP(R)^\wedge)$ and $(\mcI(R)^\vee,\mathcal{GI}_{\rm AC}(R))$ are hereditary complete cotorsion pairs in $\Mod(R)$, and $R$ is a left Gorenstein ring.
\end{enumerate}
Under any of these finiteness conditions, we have that
\[
{\rm gl.GPD}_{\rm AC}(R) = {\rm gl.GID}_{\rm AC}(R) = \id(\mathcal{LV}(R)) = \pd(\mathcal{AC}(R)).
\]
The equality ${\rm gl.GPD}_{\rm AC}(R) = {\rm gl.GID}_{\rm AC}(R)$ may also hold in the case where both dimensions are infinite. For instance, take $R$ a ring with infinite global Gorenstein dimension, and proceed as in the previous example.
\end{enumerate}
\end{example}


\subsection*{Relative Gorenstein categories}

We now have the necessary background to propose a consistent notion of relative Gorenstein categories, suitable for the construction of abelian model structures which encode the information of stable categories of relative Gorenstein projective and injective objects.

\begin{definition}\label{def:G-admissible_triple}
We say that a triple $(\mcX,\mcY,\mcZ)$ of classes of objects in $\mcC$ is \textbf{G-admissible} if the following conditions are satisfied:
\begin{enumerate}
\item[\gaone] $(\mcX,\mcY)$ is a GP-admissible pair, and $(\mcY,\mcZ)$ is a GI-admissible pair in $\mcC$.

\item[\gatwo] $\omega$, $\nu$ and $\mcY$ are closed under direct summands.

\item[\gathree] $\Ext^1(\mcGP_{(\mcX,\mcY)},\nu) = 0$ and  $\Ext^1(\omega, \mcGI_{(\mcY,\mcZ)}) = 0$.

\item[\gafour] Every object in $\mcGP_{(\mcX,\mcY)}$ admits a $\Hom(-,\nu)$-acyclic $\nu$-coresolution, and every object in $\mcGI_{(\mcY,\mcZ)}$ admits a $\Hom(\omega,-)$-acyclic $\omega$-resolution.
\end{enumerate}
\end{definition}

In the previous context, recall that since $(\mcX,\mcY)$ is GP-admissible and $(\mcY,\mcZ)$ is GI-admissible, the classes $\omega$ and $\nu$ denote the intersections $\mcX \cap \mcY$ and $\mcY \cap \mcZ$, respectively.

\begin{remark}\label{rmk:GtripleIDPD}
Let $(\mcX,\mcY,\mcZ)$ be a G-admissible triple. Then, condition \gathree \ can be replaced by
\begin{align*}
\pd_{\nu}(\mcGP_{(\mcX,\mcY)}) & = 0 & & \text{and} & \id_{\omega}(\mcGI_{(\mcY,\mcZ)}) & = 0.
\end{align*}
Indeed, assume that conditions \gaone, \gatwo, \gathree \ and \gafour \ in Definition \ref{def:G-admissible_triple} hold. Let $G \in \mcGP_{(\mcX,\mcY)}$ and $V \in \nu$, and consider an exact sequence $G' \rightarrowtail X \twoheadrightarrow G$ with $X \in \mcX$ and $G' \in \mcGP_{(\mcX,\mcY)}$. In the long cohomology sequence
\[
\cdots \to \Ext^{1}(G',V) \to \Ext^{2}(G,V) \to \Ext^{2}(X,V) \to \cdots
\]
note that $\Ext^{1}(G',V) = 0$ (since $\Ext^1(\mcGP_{(\mcX,\mcY)},\nu)=0$) and $\Ext^{2}(X,V) = 0$ (since $\pd_{\mcY}(\mcX) = 0$ and $\nu \subseteq \mcY$). Then, $\Ext^{2}(\mcGP_{(\mcX,\mcY)},\nu) = 0$. Repeating the previous argument inductively yields $\pd_{\nu}(\mcGP_{(\mcX,\mcY)}) = 0$, while the equality $\id_{\omega}(\mcGI_{(\mcY,\mcZ)}) = 0$ is dual.
\end{remark}

\begin{example}\label{ex:cotorsion_triples_are_G-admissible} \
\begin{enumerate}
\item Every hereditary complete cotorsion triple $(\mcX,\mcY,\mcZ)$ in an abelian category is G-admissible. Indeed, by \cite[Thm. 4.4]{EPZ18} (proved by Zhu and the first and third authors) one has that $\mcC$ has enough projective and injective objects, where $\mcP = \mcX \cap \mcY$ and $\mcI = \mcY \cap \mcZ$. Thus, all of the conditions in Definition \ref{def:G-admissible_triple} are straightforward.

\item In the category ${\rm Rep}(Q,\Mod(R))$, where $Q \colon 1 \xrightarrow{\alpha} 2$ is the $A_2$ quiver and $R$ is an Iwanaga-Gorenstein ring, we know from Example \ref{ex:GP_objects} (4) that
\[
(\Phi(\mcGP(R)),{\rm Rep}(Q,\mcP(R)^{< \infty} = \mcI(R)^{< \infty}),\Psi(\mcGI(R)))
\]
is a hereditary complete cotorsion triple, and so a G-admissible triple by the previous example.

\item From Example \ref{ex:GP_objects} (5), we know that if $\mcT$ is a class of objects in an abelian category $\mcC$ with enough projective and injective objects, such that $\mcT$ is $n$-tilting and $m$-cotilting for a pair of nonnegative integer $n, m \geq 0$, then $({}^\perp\mcT,\mcT)$ is a GP-admissible pair and $(\mcT,\mcT^\perp)$ is a GI-admissible pair, that is, condition \gaone \ in Definition \ref{def:G-admissible_triple} is satisfied. Since $\mcT \subseteq {}^\perp\mcT \cap \mcT^\perp$, condition \gatwo \ follows from \tiltzero. Let us now consider the class $\mcGI_{(\mcT,\mcT^\perp)}$. From \cite[Thm. 3.12 (a)]{AM21} we know that $\mcT^\vee = {}^\perp(\mcT^\perp)$, and so
\[
\mcGI_{(\mcT^\vee,\mcT^\perp)} = \mcGI_{({}^\perp(\mcT^\perp),\mcT^\perp)}.
\]
In particular, $\mcT^\vee$ is closed under direct summands, and so by the dual of \cite[Thm. 3.34 (d)]{BMS} we get that
\[
\mcGI_{(\mcT,\mcT^\perp)} = \mcGI_{(\mcT^\vee,\mcT^\perp)} = \mcGI_{({}^\perp(\mcT^\perp),\mcT^\perp)}.
\]
Moreover, since $({}^\perp(\mcT^\perp),\mcT^\perp)$ is a hereditary complete cotorsion pair, by \cite[Coroll. 3.17]{BMS} we obtain $\mcGI_{({}^\perp(\mcT^\perp),\mcT^\perp)} = \mcT^\perp$. Hence, $\mcGI_{(\mcT,\mcT^\perp)} = \mcT^\perp$, and dually, $\mcGP_{({}^\perp\mcT,\mcT)} = {}^\perp\mcT$. These equalities imply \gathree. Finally, in order to show \gafour, let $M \in \mcGI_{({}^\perp(\mcT^\perp),\mcT^\perp)} = \mcT^\perp$. Since $({}^\perp(\mcT^\perp),\mcT^\perp)$ is a complete cotorsion pair, we have a short exact sequence $M' \rightarrowtail T \twoheadrightarrow M$ where $M' \in \mcT^\perp$ and $T \in {}^\perp(\mcT^\perp)$. On the other hand, $\mcT^\perp$ is closed under extensions, and so $T \in {}^\perp(\mcT^\perp) \cap \mcT^\perp = \mcT$ by \cite[Lem. 3.9 (b) \& Prop. 3.11]{AM21}. The previous sequence is clearly $\Hom(\mcT,-)$-acyclic, and so inductively we can get a $\Hom(\mcT,-)$-acyclic $\mcT$-resolution of $M$. $\Hom(-,\mcT)$-acyclic $\mcT$-coresolutions for objects in $\mcGP_{({}^\perp\mcT,\mcT)}$ can be obtained dually. Therefore, $({}^\perp\mcT,\mcT,\mcT^\perp)$ is a G-admissible triple.
\end{enumerate}
\end{example}

More instances will be obtained in Example \ref{ex:G-admissible_systems} below after finding a way to induce G-admissible triples. Indeed, a natural question that arises is the following one: if we have a GP-admissible pair $(\mcX,\mcA)$ and a GI-admissible pair $(\mcB,\mcZ)$ of classes of objects in $\mcC$, is it possible to construct a $G$-admissible triple from them? We answer to this in the positive in the following proposition, which is a consequence of \cite[Thm. 3.34 (a,d)]{BMS} and its dual, under some specific requirements.

\begin{proposition}\label{Gc=Gt}
Let $(\mcX,\mcA)$ be a GP-admissible pair and $(\mcB,\mcZ)$ be a $GI$-admissible pair in $\mcC$ satisfying the following conditions:
\begin{enumerate}
\item $\mcA^\wedge = \mcB^\vee$.

\item The classes $\omega = \mcX \cap \mcA$, $\nu = \mcB \cap \mcZ$, $\mcA^\wedge$, $\mcX \cap \mcA^\wedge$ and $\mcA^\wedge \cap \mcZ$ are closed under direct summands.

\item $\Ext^1(\mcGP_{(\mcX,\mcA)},\nu) = 0$ and  $\Ext^1(\omega, \mcGI_{(\mcB,\mcZ)}) = 0$.

\item Every object in $\mcGP_{(\mcX,\mcA)}$ admits a $\Hom(-,\nu)$-acyclic $\nu$-coresolution, and every object in $\mcGI_{(\mcB,\mcZ)}$ admits a $\Hom(\omega,-)$-acyclic $\omega$-resolution.
\end{enumerate}
Then, $(\mcX,\mcA^\wedge,\mcZ)$  is a $G$-admissible triple in $\mcC$, and the equalities
\begin{align*}
\mcGP_{(\mcX,\mcA)} & = \mcGP_{(\mcX,\mcA^\wedge)}, & \mcGI_{(\mcB,\mcZ)} & = \mcGI_{(\mcA^\wedge,\mcZ)}, & \omega & = \mcX \cap \mcA^\wedge & & \text{and} & \nu & = \mcA^\wedge \cap \mcZ
\end{align*}
hold true.
\end{proposition}

\begin{example}\label{ex:G-admissible_systems} \
\begin{enumerate}
\item If $\mcC$ is a Gorenstein category, then by Proposition \ref{Gc=Gt} the triple $(\mcP,\mcP^\wedge,\mcI)$ is G-admissible. Indeed, by Remark \ref{Pinf=pdfin} and its dual, we have that $\mcP^\wedge = \mcP^{<\infty}$ and $\mcI^\vee = \mcI^{<\infty}$. Using now that $\pd(\mcI) < \infty$ and $\id(\mcP) < \infty$, it can be shown that the rest of the hypotheses in the previous proposition hold.

\item Let $R$ be a commutative coherent ring where $\id(\mcP(R)) < \infty$. Then,
\[
(\mcP(R),(\mcF(R))^\wedge,\mcI(R))
\]
is a G-admissible triple. Set $\mcX = \mcP(R)$, $\mcA = \mcF(R)$, $\mcB = {\rm FP}\text{-}\mcI(R)$ and $\mcZ = \mcI(R)$. Conditions (2), (3) and (4) in Proposition \ref{Gc=Gt} are straightforward. We now show condition (1), namely $\mcF(R)^\wedge = ({\rm FP}\text{-}\mcI(R))^\vee$. For, we shall need the equality
\[
\text{FP-}\id(M) = \coresdim_{{\rm FP}\text{-}\mcI(R)}(M),
\]
for every $R$-module $M$, where the \emph{FP-injective dimension} of $M$ is defined as
\[
\text{FP-}\id(M) := \inf \{ m \in \mathbb{Z}_{\geq 0} \text{ \ : \ } \Ext^{m+1}(\mathcal{FP}(R),M) = 0 \}.
\]
The previous equality follows from a standard dimension shifting argument.

We first show that $\mcF(R) \subseteq ({\rm FP}\text{-}\mcI(R))^\vee$. Let $M \in \mcF(R)$ and write $M = \varinjlim P_i$, where $P_i$ is a finitely generated projective $R$-module for every $i$ in a directed set $I$ (Lazard-Govorov Theorem, see Lam's \cite[Thm. 4.34]{Lam} for instance). If we set $\id(\mcP(R)) = k < \infty$, we have that $\text{FP-}\id(P_i) \leq \id(P_i) \leq k$ for every $i \in I$. Since $R$ is coherent, by Brown's \cite[Coroll. of Thm. 1]{Brown} it follows that
\[
\Ext^{k+1}(F,M) = \Ext^{k+1}(F,\varinjlim P_i) \cong \varinjlim \Ext^{k+1}(F,P_i) = 0,
\]
for every $F \in \mathcal{FP}(R)$, and then $\text{FP-}\id(M) \leq k$. By the previous, we get $M \in ({\rm FP}\text{-}\mcI(R))^\vee$. Finally, the containment $\mcF(R)^\wedge \subseteq ({\rm FP}\text{-}\mcI(R))^\vee$ follows by the fact that $({\rm FP}\text{-}\mcI(R))^\vee$ is closed under monocokernels.

Similarly, in order to show the containment $({\rm FP}\text{-}\mcI(R))^\vee \subseteq \mcF(R)^\wedge$, it suffices to show that ${\rm FP}\text{-}\mcI(R) \subseteq \mcF(R)^\wedge$, since $\mcF(R)^\wedge$ is closed under epikernels. So let $N \in {\rm FP}\text{-}\mcI(R)$. Since $R$ is coherent, by Fieldhouse's \cite[Thm. 2.2]{Fieldhouse}, we have that $N^+ := \Hom_{\mathbb{Z}}(N,\mathbb{Q/Z})$ is a flat (right) $R$-module. Then, $N^+ \in ({\rm FP}\text{-}\mcI(R))^\vee$ by the previous part. The rest follows from \cite[Thm. 2.1]{Fieldhouse}.

\item The triple $(\mcP(R),(\mathcal{LV}(R))^\wedge,\mcI(R))$ is G-admissible in the case where $R$ is an \textbf{AC-Gorenstein ring}. These rings were introduced by Gillespie in \cite[Def. 4.1]{GillespieAC}, and they are defined as those rings $R$ for which
\[
\resdim_{\mathcal{LV}(R)}(\mathcal{AC}(R)) < \infty \text{ \ and \ } \resdim_{\mathcal{LV}(R^{\rm op})}(\mathcal{AC}(R^{\rm op})) < \infty,
\]
where $\resdim_{\mathcal{LV}(R)}(-)$, and its dual counterpart $\coresdim_{\mathcal{AC}(R)}(-)$, are referred to as the \textbf{level} and \textbf{absolutely clean dimensions}. Using the functorial description of the level dimension given in \cite[Lem. 3.3]{GillespieAC}, we have that $\mathcal{LV}(R)^\wedge$ is closed under direct summands. Moreover, the equality $\mathcal{LV}(R)^\wedge = \mathcal{AC}(R)^\vee$ follows from \cite[Thm. 4.2]{GillespieAC}. The rest of the conditions in Proposition \ref{Gc=Gt} are straightforward.
\end{enumerate}
\end{example}

The conditions listed in the following definition were chosen in a way that they will make possible to reproduce, in the context of relative Gorenstein homological algebra, most of the important results in \cite[Ch. VII]{BR} on (absolute) Gorenstein categories.

\begin{definition}
Let $(\mcX,\mcY,\mcZ)$ be a G-admissible triple in $\mcC$.
\begin{enumerate}
\item $\mcC$ is \textbf{$\bm{(\mcX,\mcY,\mcZ)}$-Gorenstein} if $\id(\omega) < \infty$ and $\pd(\nu) < \infty$. If in addition $\mcC = \mcGP^\wedge_{(\mcX,\mcY)}$ (respectively, $\mcC = \mcGI^\vee_{(\mcY,\mcZ)}$), we say that $\mcC$ is \textbf{left} (respectively, \textbf{right}) \textbf{strongly $\bm{(\mcX,\mcY,\mcZ)}$-Gorenstein}. Left and right strongly $(\mcX,\mcY,\mcZ)$-Gorenstein categories will be called \textbf{strongly $\bm{(\mcX,\mcY,\mcZ)}$-Gorenstein}.

\item $\mcC$ is \textbf{kernel restricted $\bm{(\mcX,\mcY,\mcZ)}$-Gorenstein} if $\pd_\omega(\omega^\wedge) < \infty$ and $\id_\nu(\nu^\vee) < \infty$. If in addition $\mcC = \mcGP^\wedge_{(\mcX,\mcY)}$ and $\omega\subseteq \mcGI_{(\mcY,\mcZ)}^\vee$  (respectively, $\mcC = \mcGI^\vee_{(\mcY,\mcZ)}$ and $\nu\subseteq \mcGP_{(\mcX,\mcY)}^\wedge$), we say that $\mcC$ is \textbf{left} (respectively, \textbf{right}) \textbf{strongly kernel restricted $\bm{(\mcX,\mcY,\mcZ)}$-Gorenstein}.
\end{enumerate}
\end{definition}

The following result summarizes the existence of cotorsion theories for strongly $(\mcX,\mcY,\mcZ)$-Gorenstein categories. It is a consequence of Corollary \ref{glGPD=id} (2) and \cite[Coroll. 4.15 (a2)]{BMS}.

\begin{corollary}\label{MpropSG}
If $\mcC$ is a left strongly $(\mcX,\mcY,\mcZ)$-Gorenstein category, then $(\mcGP_{(\mcX,\mcY)},\omega^\wedge)$ is a hereditary complete cotorsion pair in $\mcC$, where $\mcGP_{(\mcX,\mcY)} = {}^{\perp}\mcY$. Moreover, the containment $\omega \subseteq \mcY \subseteq \omega^\wedge$ holds true.
\end{corollary}

\begin{remark}\label{GrelG}
Let $(\mcX,\mcY,\mcZ)$ be a G-admissible triple in $\mcC$. Note the following inequalities:
\begin{align*}
\pd_\omega(\omega^\wedge) & = \id_{\omega^\wedge}(\omega)\leq \id(\omega) & & \text{and} & \id_\nu(\nu^\vee) & = \pd_{\nu^\vee}(\nu)\leq \pd(\nu).
\end{align*}
Then:
\begin{enumerate}
\item If $\mcC$ is $(\mcX,\mcY,\mcZ)$-Gorenstein category, then it is kernel restricted $(\mcX,\mcY,\mcZ)$-Gorenstein.

\item From Corollary \ref{glGPD=id} and its dual, we can note that $\mcC$ is strongly $(\mcX,\mcY,\mcZ)$-Gorenstein if, and only if, it is left and right strongly kernel restricted $(\mcX,\mcY,\mcZ)$-Gorenstein.
\end{enumerate}
\end{remark}

Now we aim to provide characterizations of strongly $(\mcX,\mcY,\mcZ)$-Gorenstein categories. At some steps towards this goal, we shall need to consider the following set of conditions for a G-admissible triple $(\mathcal{X,Y,Z})$:
\begin{enumerate}
\item[($\mathsf{ga5}$)] Every object in $\mathcal{GP}_{(\mathcal{X,Y})}$ admits a $\omega$-resolution, and every object in $\mathcal{GI}_{(\mathcal{Y,Z})}$ admits a $\nu$-coresolution.

\item[($\mathsf{ga6}$)] If $\mathcal{C}$ is AB4 and AB4${}^\ast$, then $\omega$ is closed under arbitrary coproducts and $\nu$ is closed under arbitrary products.
\end{enumerate}

\begin{lemma}\label{lemcGP=cGI}
Let $\mcC$ be an AB4 and AB4${}^\ast$ abelian category and $(\mcX,\mcY,\mcZ)$ be a $G$-admissible triple in $\mcC$ satisfying ($\mathsf{ga5}$) and ($\mathsf{ga6}$). If $\id_\nu(\omega)<\infty$, $\omega\subseteq \mcGI_{(\mcY,\mcZ)}^\vee$ and $\mathcal{GP}_{(\mathcal{X,Y})} \subseteq \mathcal{I}^{<\infty}_{\mathcal{Z}}$, then
\[
\mcGP^\wedge_{(\mcX,\mcY)}\subseteq \mcGI^\vee_{(\mcY,\mcZ)}.
\]
\end{lemma}

\begin{proof}
Using that $\omega \subseteq \mcGI^\vee_{(\mcY,\mcZ)}$ and $\id_{\nu}(\omega) < \infty$, by the dual of \cite[Coroll. 4.15 (b1)]{BMS} it follows that $\Gid_{(\mcY,\mcZ)}(\omega)=\id_\nu(\omega)<\infty.$ On the other hand, from \cite[Prop. 2.5]{HMP22} (by Huerta and the second and third authors) one has that $\Gid_{(\mcY,\mcZ)}(\mcGP^\wedge_{(\mcX,\mcY)}) = \Gid_{(\mcY,\mcZ)}(\omega)$.
\end{proof}

\begin{theorem}\label{MThmGc}
For a $G$-admissible triple $(\mcX,\mcY,\mcZ)$ in $\mcC$, consider the following assertions:
\begin{enumerate}
\item[(a)] $\glGPD_{(\mcX,\mcY)}(\mcC) < \infty$ and $\glGID_{(\mcY,\mcZ)}(\mcC) < \infty$.

\item[(b)] $\mcC$ is strongly $(\mcX,\mcY,\mcZ)$-Gorenstein.

\item[(c)] $\mcC$ is kernel restricted $(\mcX,\mcY,\mcZ)$-Gorenstein and $\mcGP_{(\mcX,\mcY)}^\wedge = \mcC = \mcGI_{(\mcY,\mcZ)}^\vee$.

\item[(d)] $\mcC$ is left strongly $(\mcX,\mcY,\mcZ)$-Gorenstein.

\item[(e)] $\mcC$ is left strongly kernel restricted $(\mcX,\mcY,\mcZ)$-Gorenstein.

\item[(f)] $\mcC$ is right strongly $(\mcX,\mcY,\mcZ)$-Gorenstein.

\item[(g)] $\mcC$ is right strongly kernel restricted $(\mcX,\mcY,\mcZ)$-Gorenstein.
\end{enumerate}
Then, the implications (b) $\Leftrightarrow$ (a) $\Leftrightarrow$ (c), (d) $\Leftarrow$ (b) $\Rightarrow$ (f) and (e) $\Leftarrow$ (c) $\Rightarrow$ (g) hold true. Moreover,
\begin{enumerate}
\item If $\omega \subseteq \mcGI_{(\mcY,\mcZ)}^\vee$, then (d) $\Rightarrow$ (e).

\item If $\nu\subseteq \mcGP_{(\mcX,\mcY)}^\wedge$, then (f) $\Rightarrow$ (g).
\end{enumerate}

In addition, suppose $\mathcal{C}$ is AB4 and AB4${}^\ast$, and that $(\mathcal{X,Y,Z})$ satisfies ($\mathsf{ga5}$) and ($\mathsf{ga6}$). Then, the following assertions hold:
\begin{enumerate}
\setcounter{enumi}{2}
\item If $\mathcal{GP}_{(\mathcal{X,Y})} \subseteq \mathcal{I}^{<\infty}_{\mathcal{Z}}$, then (e) $\Rightarrow$ (a).

\item If $\mathcal{GI}_{(\mathcal{Y,Z})} \subseteq \mathcal{P}^{<\infty}_{\mathcal{X}}$, then (g) $\Rightarrow$ (a).
\end{enumerate}
In particular, under the assumptions in (1), (2), (3) and (4), all of the assertions from (a) to (g) are equivalent.
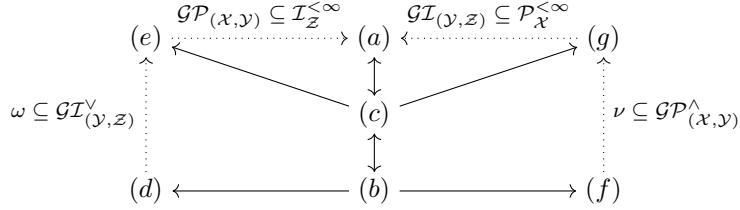
\[
\begin{tikzpicture}[description/.style={fill=white,inner sep=2pt}]
\matrix (m) [matrix of math nodes, row sep=1.5em, column sep=3em, text height=1.25ex, text depth=0.25ex]
{
(e) & {} & (a) & {} & (g) \\
{} & {} & (c) & {} & {} \\
(d) & {} & (b) & {} & (f) \\
};
\path[->]
(m-3-3) edge (m-3-5) edge (m-3-1)
(m-2-3) edge (m-1-5) edge (m-1-1)
;
\path[<->]
(m-1-3) edge (m-2-3)
(m-2-3) edge (m-3-3)
;
\path[dotted,->]
(m-3-1) edge node[left] {\footnotesize $\omega \subseteq \mcGI^\vee_{(\mcY,\mcZ)}$} (m-1-1)
(m-3-5) edge node[right] {\footnotesize$\nu \subseteq \mcGP^\wedge_{(\mcX,\mcY)}$} (m-1-5)
(m-1-1) edge node[above] {\footnotesize$\mathcal{GP}_{(\mathcal{X,Y})} \subseteq \mathcal{I}^{<\infty}_{\mathcal{Z}}$} (m-1-3)
(m-1-5) edge node[above] {\footnotesize$\mathcal{GI}_{(\mathcal{Y,Z})} \subseteq \mathcal{P}^{<\infty}_{\mathcal{X}}$} (m-1-3)
;
\end{tikzpicture}
\]
Moreover, for every AB4 and AB4${}^\ast$ strongly $(\mcX,\mcY,\mcZ)$-Gorenstein category $\mcC$ satisfying ($\mathsf{ga5}$) and ($\mathsf{ga6}$), the following equalities hold:
\begin{align}
\glGID_{(\mathcal{Y,Z})}(\mcC) & = \Gid_{(\mathcal{Y,Z})}(\mathcal{GP}_{(\mathcal{X,Y})}) = \Gid_{(\mathcal{Y,Z})}(\omega) = \id_{\nu}(\omega) = \id_{\nu}(\mathcal{GP}_{(\mathcal{X,Y})}) \nonumber \\
& = \pd(\nu) = \pd(\mcY) = \id_{\mathcal{Y}}(\mathcal{GP}_{(\mathcal{X,Y})}) = \pd_{\mathcal{Y}}(\mathcal{GI}_{(\mathcal{Y,Z})}) = \id(\mcY) \nonumber \\
& = \id(\omega) = \pd_{\omega}(\mathcal{GI}_{(\mathcal{Y,Z})}) = \Gpd_{(\mathcal{X,Y})}(\nu) = \Gpd_{(\mathcal{X,Y})}(\mathcal{GI}_{(\mathcal{Y,Z})}) \nonumber \\
& = \glGPD_{(\mathcal{X,Y})}(\mcC). \label{eqn:chain_equalities}
\end{align}
\end{theorem}

\begin{proof}
The first circuits of implications follow from Corollary \ref{glGPD=id} (2), its dual, and Remark \ref{GrelG}. This remark also implies parts (1) and (2). So we just focus on showing part (3), as (4) is dual. So assume $\mcC$ is left strongly kernel restricted $(\mcX,\mcY,\mcZ)$-Gorenstein, that is, $\pd_\omega(\omega^\wedge) < \infty$, $\id_\nu(\nu^\vee) < \infty$, $\mcC = \mcGP^\wedge_{(\mcX,\mcY)}$ and $\omega\subseteq \mcGI_{(\mcY,\mcZ)}^\vee$. Using Remark \ref{GrelG}, it remains to show that $\mcC$ is a right strongly kernel restricted $(\mcX,\mcY,\mcZ)$-Gorenstein category, that is, $\mcC = \mcGI^\vee_{(\mcY,\mcZ)}$. For, we use Lemma \ref{lemcGP=cGI}. From the assumptions, it remains to show $\id_\nu(\omega) < \infty$. By Proposition \ref{lpdGP} we have that $\glGPD_{(\mcX,\mcY)}(\mcC) < \infty$. Then, \cite[Coroll. 4.15 (b1)]{BMS} yields
\[
\id_\nu(\omega) = \pd_\omega(\nu) = \Gpd_{(\mcX,\mcY)}(\nu) \leq \glGPD_{(\mcX,\mcY)}(\mcC) < \infty.
\]
Finally, equality \eqref{eqn:chain_equalities} follows from \cite[Thm. 8.7]{HMP22}.
\end{proof}

As a particular consequence of the previous characterization, we show that the main properties of Gorenstein categories obtained in \cite[Ch. VII]{BR} are covered in our general framework.

\begin{lemma}\label{condwv}
Let $(\mcX,\mcY,\mcZ)$ be a triple of classes of objects in $\mcC$ with $\omega := \mcX \cap \mcY$ and $\nu := \mcY \cap \mcZ$.
\begin{enumerate}
\item If $\mcC$ has enough injective objects, $\mcI \subseteq \nu$ and $\id(\omega) < \infty$, then $\omega^\wedge \subseteq \mathcal{I}^{< \infty} \subseteq \nu^\vee$.

\item If in addition $(\mcY,\mcZ)$ is $GI$-admissible with $\nu$ closed under direct summands, then $\pd_{\mcGI_{(\mcY,\mcZ)}}(\omega) = 0$ and $\omega \subseteq \mcGI_{(\mcY,\mcZ)}^\vee$.
\end{enumerate}
\end{lemma}

\begin{proof}
Part (1) is straightforward from the dual of Remark \ref{Pinf=pdfin}. For part (2), from the dual of \cite[Coroll. 5.2 (b)]{BMS} we get that $\nu^\vee = {}^\perp\mcGI_{(\mcY,\mcZ)}\cap\mcGI_{(\mcY,\mcZ)}^\vee$. Then, part (2) follows from (1).
\end{proof}

\begin{proposition}\label{prop:absolute_Gorenstein_category}
Let $\mcC$ be an abelian category with enough projective and injective objects. The following conditions are equivalent:
\begin{enumerate}
\item[(a)] $\mcC$ is $(\mcP,\mcP^\wedge,\mcI)$-Gorenstein.

\item[(b)] $\mcC$ is strongly $(\mcP,\mcP^\wedge,\mcI)$-Gorenstein.

\item[(c)] $\mcC$ satisfies \Gone\, and \Gtwo.
\end{enumerate}
If in addition $\mcC$ is AB3 and AB3${}^\ast$, then the previous are also equivalent to:
\begin{enumerate}
\item[(d)] $\mcC$ is Gorenstein.
\end{enumerate}
\end{proposition}

\begin{proof} \
\begin{itemize}
\item (a) $\Rightarrow$ (b): Suppose first that $\mcC$ is a $(\mcP,\mcP^\wedge,\mcI)$-Gorenstein category. Then, $(\mcP,\mcP^\wedge,\mcI)$ is a G-admissible triple, $\id(\mcP) < \infty$ and $\pd(\mcP^\wedge \cap \mcI) < \infty$. One can note that actually $\mcP^\wedge \cap \mcI = \mcI$. Indeed, since $(\mcP^\wedge,\mcI)$ is a GI-admissible pair, we have that $\mcP^\wedge \cap \mcI$ is a relative generator in $\mcI$. So if we let $E \in \mcI$, there exists a split exact sequence $E' \rightarrowtail H \twoheadrightarrow E$ where $E' \in \mcI$ and $H \in \mcP^\wedge \cap \mcI$, and so $E \in \mcP^\wedge \cap \mcI$. In particular, $\mcI \subseteq \mcP^\wedge$, which in turn implies $\mcI^\vee \subseteq \mcP^\wedge$. The other containment follows from the assumption that $\id(\mcP) < \infty$. Also, $\pd(\mcI) = \pd(\mcP^\wedge \cap \mcI) < \infty$, while $\id(\mcP^\wedge) = \id(\mcP) < \infty$ and $\pd(\mcP^\wedge) = \pd(\mcI^\vee) = \pd(\mcI) < \infty$ by \cite[Lem. 2.13 (a)]{MS} and its dual. Then, by Theorem \ref{thm:impliesBR}, we have that $\glGPD_{(\mcP,\mcP^\wedge)}(\mcC) < \infty$ and $\glGID_{(\mcP^\wedge,\mcI)}(\mcC) < \infty$. Hence, assertion (b) is a consequence of Theorem \ref{MThmGc}.

\item (b) $\Rightarrow$ (c): If $\mcC$ is a strongly $(\mcP,\mcP^\wedge,\mcI)$-Gorenstein category, we can show as in the previous part that $\mcP^\wedge = \mcI^\vee$ (that is, we obtain condition \Gone \ in Definition \ref{def:Gorenstein_category}). We can then notice from $\mcP^\wedge = \mcI^\vee$ that $\mcGP = \mcGP_{(\mcP,\mcP)} = \mcGP_{(\mcP,\mcP^\wedge)}$ and $\mcGI = \mcGI_{(\mcI,\mcI)} = \mcGI_{(\mcI^\vee,\mcI)} = \mcGI_{(\mcP^\wedge,\mcI)}$. Using this, along with Remark \ref{rmk:Gorenstein_category} (1) and its dual, we obtain
\begin{align*}
\FPD(\mcC) & \leq \glGPD(\mcC) = \glGPD_{(\mcP,\mcP^\wedge)}(\mcC) < \infty \ \text{and} \\
\FID(\mcC) & \leq \glGID(\mcC) = \glGID_{(\mcI^\vee,\mcI)}(\mcC) < \infty,
\end{align*}
that is, condition  \Gtwo \ in Definition \ref{def:Gorenstein_category} holds.

\item (c) $\Rightarrow$ (a): Suppose that $\mcP^\wedge = \mcI^\vee$, $\FPD(\mcC) < \infty$ and $\FID(\mcC)  < \infty$. Then, $\id(\mcP) = \id(\mcP^\wedge) = \id(\mcI^\vee) = \FID(\mcC) < \infty$. In particular, the triple $(\mcP,\mcP^\wedge,\mcI)$ is G-admissible by Example \ref{ex:G-admissible_systems} (1). Thus, $\id(\mcP \cap \mcP^\wedge) = \id(\mcP) < \infty$. In a similar way, we have $\pd(\mcP^\wedge \cap \mcI) < \infty$. Hence, (a) follows.

\item (c) $\Leftrightarrow$ (d): Follows by Remark \ref{rmk:Gorenstein_category}.
\end{itemize}
\end{proof}

If $\mcC$ is a Gorenstein category or a strongly $(\mcP,\mcP^\wedge,\mcI)$-Gorenstein category, then $\mcP^\wedge = \mcI^\vee$. A natural question is whether the equality $\omega^\wedge = \nu^\vee$ holds in any strongly $(\mcX,\mcY,\mcZ)$-Gorenstein category. The following result shows that this only occurs for G-admissible triples $(\mcX,\mcY,\mcZ)$ where $\omega = \mcP$ and $\nu = \mcI$.

\begin{proposition}\label{wg=nc}
Let $\mcC$ be a strongly $(\mcX,\mcY,\mcZ)$-Gorenstein category. The following conditions are equivalent:
\begin{enumerate}
\item[(a)] $\mcC$ has enough projective and injective objects, $\mcP \subseteq \omega$ and $\mcI \subseteq \nu$.

\item[(b)] $\omega^\wedge = \mcI^{< \infty} = \mcP^{< \infty} = \nu^\vee$.

\item[(c)] $\omega^\wedge = \nu^\vee$.

\item[(d)] $\mcC$ has enough projective and injective objects, $\omega = \mcP$ and $\nu = \mcI$.
\end{enumerate}
In particular, if $\mcY$ is closed under epikernels and monocokernels, then
\begin{align}
\mcP^{<\infty} & = \omega^\wedge = \mcY = \nu^\vee = \mcI^{<\infty}. \label{eqn:kernels_monocokernels}
\end{align}
Furthermore, if $\mcC$ is a $(\mcX,\mcY,\mcZ)$-Gorenstein category with enough projective and injective objects such that $\mcI \subseteq \nu$ and $\mcP \subseteq \omega$, then it is strongly $(\mcX,\mcY,\mcZ)$-Gorenstein if, and only if, $\pd(\mcY) < \infty$ and $\id(\mcY) < \infty$.
\end{proposition}

\begin{proof}
Implications (b) $\Rightarrow$ (c) and (d) $\Rightarrow$ (a) are immediate, while (a) $\Rightarrow$ (b) follows from Lemma \ref{condwv} (1) and its dual. So we only focus on showing (c) $\Rightarrow$ (d). Suppose then that $\omega^\wedge = \nu^\vee$. By Corollary \ref{MpropSG} and its dual, we get the hereditary complete cotorsion triple $(\mcGP_{(\mcX,\mcY)},\omega^\wedge=\nu^\vee,\mcGI_{(\mcY,\mcZ)})$ in $\mcC.$ Thus, from \cite[Thm. 4.4]{EPZ18} the category $\mcC$ has enough projective and injective objects. Moreover, from \cite[Props. 3.3 \& 4.2]{EPZ18}, it follows that $\mcGP_{(\mcX,\mcY)} \cap \omega^\wedge = \mcP$ and $\mcI = \nu^\vee\cap\mcGI_{(\mcY,\mcZ)}$. Hence, from Proposition \ref{coro:properties_relative_Gorenstein_projective} (2) and its dual, we conclude that $\omega = \mcP$ and $\nu = \mcI$.

Now suppose that $\mcY$ is closed under epikernels and monocokernels. By Corollary \ref{MpropSG} again, we have $\omega \subseteq \mcY \subseteq \omega^\wedge$. On the other hand, $\omega^\wedge \subseteq \mcY$ since $\omega \subseteq \mcY$ and $\mcY$ is closed under monocokernels. Therefore,  $\mcY= \omega^\wedge$. The equality $\mcY = \nu^\vee$ follows similarly. Now following the proof of (c) $\Rightarrow$ (d), we can note that $\omega = \mcP$ and $\nu = \mcI$, and hence we obtain equality \eqref{eqn:kernels_monocokernels}. The last part is a consequence of Theorem \ref{thm:impliesBR}.
\end{proof}


\subsection*{Examples of relative Gorenstein categories}

We cover Frobenius categories and Beligiannis and Reinten's Gorenstein categories from \cite{BR}. The impossibility of obtaining relative Gorenstein categories from Ding projective and Gorenstein AC-projective modules, and their dual counterparts, different from the (absolute) Go-renstein category is exhibited. Moreover, we provide a characterization of the finiteness of the left global dimension of a ring $R$ in terms of the existence of a certain relative Gorenstein structure on the category $\Ch(R)$ of chain complexes of $R$-modules. We also show a way to induce from a well known cotorsion triple a relative Gorenstein category structure on the category of representations of the quiver of type ${\rm A}_2$ in the category of $R$-modules. Our last example displays a strongly Gorenstein category obtained from a tilting-cotilting class.

\begin{example}[conditions on cotorsion triples]\label{ex:condition_triples}
Let $(\mathcal{X,Y,Z})$ be a hereditary complete cotorsion triple in an abelian category $\mcC$. We know by Example \ref{ex:cotorsion_triples_are_G-admissible} (1) that $(\mathcal{X,Y,Z})$ be a G-admissible triple. On the other hand, by \cite[proof of Thm. 4.4]{EPZ18} we have that $\mcC$ has enough projective and injective objects, $\omega = \mcP$ and $\nu = \mcI$. It follows that $\mcC$ is $(\mcX,\mcY,\mcZ)$-Gorenstein if, and only if, $\id(\mcP) < \infty$ and $\pd(\mcI) < \infty$, that is, if $\mcC$ is a Gorenstein category. On the other hand, since $(\mcX,\mcY,\mcZ)$ is a hereditary complete cotorsion triple, we have by \cite[Coroll. 3.17]{BMS} and its dual that $\mcG_{(\mcX,\mcY)} = \mcX$ and $\mcGI_{(\mcY,\mcZ)} = \mcZ$. It then follows that $\mcC$ is a strongly $(\mcX,\mcY,\mcZ)$-Gorenstein category if, and only if, $\mcC$ is Gorenstein and $\mcC = \mcX^\wedge = \mcZ^\vee$. Moreover, condition ($\mathsf{ga5}$) is clearly satisfied, and if we assume $\mcC$ is AB4 and AB4${}^\ast$, then ($\mathsf{ga6}$) also holds since $\omega = \mcP$ and $\nu = \mcI$. Thus, in this case by Theorem \ref{MThmGc} we obtain:
\begin{align*}
\coresdim_{\mcZ}(\mcC) & = \coresdim_{\mcZ}(\mcX) = \coresdim_{\mcZ}(\mcP) = \id_{\mcI}(\mcP) = \id_{\mcI}(\mcX) = \pd(\mcI) \\
& = \pd(\mcY) = \id_{\mathcal{Y}}(\mcX) = \pd_{\mathcal{Y}}(\mcZ) = \id(\mcY) = \id(\mcP) = \pd_{\mcP}(\mcZ) \\
& = \resdim_{\mcX}(\mcI) = \resdim_{\mcX}(\mcZ) = \resdim_{\mcX}(\mcC).
\end{align*}

Among the examples of hereditary complete cotorsion triples considered in this article, we have:
\begin{itemize}
\item $(\mcP(R),\Mod(R),\mcI(R))$.

\item $(\Mod(R),\mcP(R),\Mod(R))$ provided that $R$ is a quasi-Frobenius ring. In this case, $\Mod(R)$ is a strongly $(\Mod(R),\mcP(R),\Mod(R))$-Gorenstein category (see Example \ref{ex:FrobeniusGorenstein} below).

\item $({\rm dg}\widetilde{\mcP(R)},\mathcal{E},{\rm dg}\widetilde{\mathcal{I}(R)})$, the \emph{Dold cotorsion triple}, where ${\rm dg}\widetilde{\mcP(R)}$ and ${\rm dg}\widetilde{\mathcal{I}(R)}$ denote the classes of differential graded projective and injective complexes, respectively. Following Gillespie's \cite[Def. 3.3]{GillespieFlat}, recall that a chain complex $P_\bullet$ of $R$-modules is dg-projective if $P_\bullet \in \Ch(\mcP(R))$ and every chain map from $P_\bullet$ to any exact complex is chain homotopic to $0$. Dg-injective complexes are defined dually. For instance, the reader can check Enochs and Jenda's \cite[\S 4.3]{EJ2} or Garc\'{\i}a Rozas' \cite[\S 2.3]{JRlibro} for details. The case where $\Ch(R)$ is a strongly $({\rm dg}\widetilde{\mcP(R)},\mathcal{E},{\rm dg}\widetilde{\mathcal{I}(R)})$-Gorenstein category is covered in Example \ref{ex:global_dim_rel_Gor} below.
\end{itemize}
\end{example}

\begin{remark}\label{rmk:RelGorChar}
From the previous example, we can observe that if $(\mcX,\mcY,\mcZ)$ is a G-admissible triple in an abelian category $\mcC$ with enough projective and injective objects, with $\mcX \cap \mcY = \mcP$ and $\mcY \cap \mcZ = \mcI$, then the following are equivalent:
\begin{enumerate}
\item[(a)] $\mcC$ is a $(\mcX,\mcY,\mcZ)$-Gorenstein category.

\item[(b)] $\mcC$ is $(\mcP,\mcP^\wedge,\mcI)$-Gorenstein.

\item[(c)] $\mcC$ is strongly $(\mcP,\mcP^\wedge,\mcI)$-Gorenstein.

\item[(d)] $\mcC$ is Gorenstein.
\end{enumerate}
\end{remark}

\begin{example}[Frobenius categories]\label{ex:FrobeniusGorenstein}
Given an abelian category $\mcC$ and $\mcF \subseteq \mcC$ be an abelian Frobenius subcategory, it is possible to obtain a relative Gorenstein category structure on $\mcF$. Let $\mathcal{PI}$ denote the class of projective-injective objects of $\mcF$. Then, it is easy to check that the triple $(\mathcal{F,PI,F})$ satisfies conditions \gaone, \gatwo \ and \gafour. Condition \gathree, on the other hand, follows after noticing that $\mcGP_{(\mathcal{F,PI})} = \mathcal{F}$ and $\mcGI_{(\mathcal{PI,F})} = \mathcal{F}$, since $(\mathcal{F,PI})$ and $(\mathcal{PI,F})$ are hereditary complete cotorsion pairs in $\mathcal{F}$. It then follows that $\mcF$ is strongly $(\mathcal{F,PI,F})$-Gorenstein.

For a particular example of such $\mcC$ and $\mcF$, let $Q$ and $Q'$ be the quivers given by
\[
\begin{tikzcd}
\bullet_{1} \arrow[out=90,in=0,loop,"\alpha"] & & \text{and} & & \bullet_{2} \arrow[r,"\beta"] & \bullet_{3}
\end{tikzcd}
\]
respectively. Now let $\mathbb{K}$ be a field and $\Lambda$ be the non connected algebra
\[
\Lambda = (\mathbb{K}Q / \langle \alpha^2 \rangle) \times \mathbb{K}Q'.
\]
Set $\mcC = \Mod(\Lambda) \cong \Mod(\mathbb{K}Q / \langle \alpha^2 \rangle) \times \Mod(\mathbb{K}Q')$ and $\mcF = \Mod(\mathbb{K}Q / \langle \alpha^2 \rangle)$. Then, $\Mod(\mathbb{K}Q / \langle \alpha^2 \rangle)$ is an abelian Frobenius subcategory of $\Mod(\Lambda)$.
\end{example}

\begin{example}[Ding modules]
Fix a commutative coherent ring $R$ with $\id(\mcP(R)) < \infty$. We know from Example \ref{ex:G-admissible_systems} (2) that $(\mcP(R),\mcF(R)^\wedge,\mcI(R))$ is a G-admissible triple with $\mcF(R)^\wedge = {\rm FP}\mbox{-}\mcI(R)^\vee$, $\omega = \mcP(R)$ and $\nu = \mcI(R)$. By Remark \ref{rmk:RelGorChar}, we have that $\Mod(R)$ is a $(\mcP(R),\mcF(R)^\wedge,\mcI(R))$-Gorenstein category if, and only if, $\pd(\mcI(R)) < \infty$. One also has by Proposition \ref{Gc=Gt} and Example \ref{ex:GP_objects} (2) that $\mcGP_{(\mcP(R),\mcF(R)^\wedge)} = \mcDP(R)$ and $\mcGI_{(\mcF(R)^\wedge,\mcI(R))} = \mathcal{DI}(R)$. If in addition $\pd(\mcI(R)) < \infty$, then the strongly $(\mcP(R),\mcF(R)^\wedge,\mcI(R))$-Gorenstein structure coincides with that of Proposition \ref{prop:absolute_Gorenstein_category} (see Example \ref{ex:examples_of_global_Gorenstein_dimensions}).
\end{example}

\begin{example}[AC-Gorenstein modules]\label{ex:Gorenstein_categoriesAC}
Given an AC-Gorenstein ring $R$, we know from Example \ref{ex:G-admissible_systems} (3) that there is a G-admissible triple $(\mcP(R),\mathcal{LV}(R)^\wedge,\mcI(R))$ with $\omega = \mcP(R)$ and $\nu = \mcI(R)$ as well. Note that $\Mod(R)$ is a $(\mcP(R),\mathcal{LV}(R)^\wedge,\mcI(R))$-Gorenstein category if, and only if, $R$ is a left Gorenstein ring.

On the other hand, from Example \ref{ex:examples_of_global_Gorenstein_dimensions} (3), we have that $\Mod(R)$ is a strongly $(\mcP(R),\mathcal{LV}(R)^\wedge,\mcI(R))$-Gorenstein category if, and only if, $R$ is an AC-Gorenstein ring, $\id(\mathcal{LV}(R)) < \infty$ and $\pd(\mathcal{AC}(R)) < \infty$. Since $\mcI(R) \subseteq \mathcal{AC}(R)$ and $\mcP(R) \subseteq \mathcal{LV}(R)$, in the case where $R$ is commutative, we have that $\Mod(R)$ is a strongly $(\mcP(R),\mathcal{LV}(R)^\wedge,\mcI(R))$-Gorenstein category if, and only if, $\id(\mathcal{LV}(R)) < \infty$ and $\pd(\mathcal{AC}(R)) < \infty$.
\end{example}

\begin{example}[rings of finite global dimension in terms of strongly Gorenstein categories]\label{ex:global_dim_rel_Gor}
For the Dold cotorsion triple $({\rm dg}\widetilde{\mathcal{P}(R)},\mathcal{E},{\rm dg}\widetilde{\mathcal{I}(R)})$, we show that $\Ch(R)$ is a strongly $({\rm dg}\widetilde{\mathcal{P}(R)},\mathcal{E},{\rm dg}\widetilde{\mathcal{I}(R)})$-Gorenstein category if, and only if, the (left) global dimension of $R$, denoted ${\rm gl.dim}(R)$, is finite.

For this G-admissible triple, we have that $\omega$ and $\nu$ are the classes of projective and injective $R$-complexes, respectively (see \cite[Thm. 3.12]{GillespieFlat}). Recall that an $R$-complex is projective if, and only if, it is exact with projective cycles. Injective $R$-complexes have the dual description. Following the notation in \cite[Def. 3.3]{GillespieFlat}, the classes of projective and injective $R$-complexes are denoted by $\widetilde{\mcP(R)}$ and $\widetilde{\mcI(R)}$. We thus have by \cite[Coroll. 3.17]{BMS} that
\begin{align*}
\mcGP_{({\rm dg}\widetilde{\mathcal{P}(R)},\mathcal{E})} & = {\rm dg}\widetilde{\mathcal{P}(R)}, & \mcGI_{(\mathcal{E},{\rm dg}\widetilde{\mathcal{I}(R)})} & = {\rm dg}\widetilde{\mathcal{I}(R)}, & \omega & = \widetilde{\mcP(R)} & \text{and} & & \nu & = \widetilde{\mcI(R)},
\end{align*}
and so
\begin{align*}
\glGPD_{({\rm dg}\widetilde{\mathcal{P}(R)},\mathcal{E})}(\Ch(R)) & = \resdim_{{\rm dg}\widetilde{\mathcal{P}(R)}}(\Ch(R)), \\
\glGID_{(\mathcal{E},{\rm dg}\widetilde{\mathcal{I}(R)})}(\Ch(R)) & = \coresdim_{{\rm dg}\widetilde{\mathcal{I}(R)}}(\Ch(R)).
\end{align*}

If $\Ch(R)$ is a strongly $({\rm dg}\widetilde{\mathcal{P}(R)},\mathcal{E},{\rm dg}\widetilde{\mathcal{I}(R)})$-Gorenstein category, it is straightforward to show that ${\rm gl.dim}(R) \leq \glGPD_{({\rm dg}\widetilde{\mathcal{P}(R)},\mathcal{E})}(\Ch(R)) < \infty$. Now suppose that ${\rm gl.dim}(R) = n < \infty$. We show that $\resdim_{{\rm dg}\widetilde{\mathcal{P}(R)}}(\Ch(R)) \leq n+1$, and the finiteness of $\coresdim_{{\rm dg}\widetilde{\mathcal{I}(R)}}(\Ch(R))$ will follow from a dual argument. For every $X_\bullet \in \Ch(R)$, it is possible to obtain a short exact sequence $E_\bullet \rightarrowtail P_\bullet \twoheadrightarrow X_\bullet$ where $P_\bullet \in {\rm dg}\widetilde{\mcP(R)}$ and $E_\bullet \in \mathcal{E}$ (see for instance \cite[Thm. 7.3.2]{EJ2}). On the other hand, from a standard homological algebra argument, one can note that $\pd(E_\bullet) \leq n$. Therefore, we have that $\resdim_{{\rm dg}\widetilde{\mcP(R)}}(X_\bullet) \leq n+1$.

We conclude pointing out that if $R$ is a ring with finite global dimension, then by Theorem \ref{MThmGc} we have that:
\begin{align*}
\coresdim_{{\rm dg}\widetilde{\mcI(R)}}(\Ch(R)) & = \resdim_{{\rm dg}\widetilde{\mcP(R)}}(\Ch(R)) \leq {\rm gl.dim}(R) + 1.
\end{align*}
\end{example}

\begin{example}[induced strongly Gorenstein category structure on the category of representations of a quiver]\label{ex:induced_on_quivers}
The category $\Ch(R)$ of chain complexes of $R$-modules considered in the previous example can be regarded as the category of $\Mod(R)$-valued representations of the $A^\infty_\infty$ quiver
\[
\cdots \to 2 \to 1 \to 0 \to -1 \to -2 \to \cdots
\]
with the relations that the composition of two consecutive arrows is zero. In this example, we consider another category of representations but with infinite global dimension, namely, the category ${\rm Rep}(Q,\Mod(R))$, where $Q \colon 1 \xrightarrow{\alpha} 2$ is the $A_2$ quiver and $R$ is an Iwanaga-Gorenstein ring with ${\rm gl.Gdim}(R) = 1$ and ${\rm gl.dim}(R) = \infty$. Under these conditions, we prove that the category ${\rm Rep}(Q,\Mod(R))$ is strongly $(\Phi(\mcGP(R)),{\rm Rep}(Q,\mcP(R)^{< \infty}),\Psi(\mcGI(R)))$-Gorenstein. We know from (2) in Example \ref{ex:cotorsion_triples_are_G-admissible} that  $(\Phi(\mcGP(R)),{\rm Rep}(Q,\mcP(R)^{< \infty}),\Psi(\mcGI(R)))$ is a G-admissible triple. The idea is to show that
\begin{align*}
\resdim_{\Phi(\mcGP(R))}({\rm Rep}(Q,\Mod(R))) & \leq 2, \text{ and} \\
\coresdim_{\Psi(\mcGI(R))}({\rm Rep}(Q,\Mod(R))) & \leq 2.
\end{align*}
We only focus on the first inequality. With respect to $\mcP^{< \infty}$ in ${\rm Rep}(Q,\Mod(R))$, we have by \cite[Coroll. 3.19]{ArgudinMendoza-quiver} that
\[
\pd(F) \leq \max\{ \pd(F_1), \pd(F_2) \} + 1
\]
for every $F \in {\rm Rep}(Q,\Mod(R))$. It follows that $\mcP^{< \infty} = {\rm Rep}(Q,\mcP(R)^{< \infty})$ (note in particular that the global projective dimension of ${\rm Rep}(Q,\Mod(R))$ is infinite). Now since $(\Phi(\mcGP(R)),\mcP^{< \infty})$ is a complete cotorsion pair, we have that
\[
{\rm Rep}(Q,\Mod(R)) = \Phi(\mcGP(R))^\wedge.
\]
So for every $F \in {\rm Rep}(Q,\Mod(R))$ we can consider an exact sequence
\[
K \rightarrowtail G^1 \to G^0 \twoheadrightarrow F
\]
where $G^0, G^1 \in \Phi(\mcGP(R))$ and $K, K^0 = {\rm Ker}(G^0 \to F) \in \mcP^{< \infty}$. We show that $K \in \Phi(\mcGP(R))$. For, it will be important to know that
\[
\Phi(\mcGP(R)) = \{ G \in {\rm Rep}(Q,\Mod(R)) \ {\rm : } \ G_1, {\rm CoKer}(G_\alpha) \in \mcGP(R) \text{ and } G_\alpha \text{ is monic} \}.
\]
Note we can form the following exact commutative diagram:
\[
\begin{tikzpicture}[description/.style={fill=white,inner sep=2pt}]
\matrix (m) [matrix of math nodes, row sep=1.25em, column sep=1.5em, text height=1.25ex, text depth=0.25ex]
{
{} & {} & {} & {} & {\rm Ker}(F_\alpha) \\
{} & G^1_1 & {} & G^0_1 & F_1 \\
K_1 & {} & K^0_1 & {} & {} \\
{} & G^1_2 & {} & G^0_2 & F_2 \\
K_2 & {} & K^0_2 & {} & {} \\
{} & G^1_2 / G^1_1 & {} & G^0_2 / G^0_1 & {\rm CoKer}(F_\alpha)  \\
K_2 / K_1 & {} & K^0_2 / K^0_1 & {} & {} \\
};
\path[>->]
(m-3-1) edge node[right] {\footnotesize$K_\alpha$} (m-5-1) (m-2-2) edge node[right] {\footnotesize$G^1_\alpha$} (m-4-2) (m-2-4) edge node[right] {\footnotesize$G^0_\alpha$} (m-4-4) (m-3-1) edge (m-2-2) (m-5-1) edge (m-4-2) (m-7-1) edge (m-6-2) (m-3-3) edge (m-2-4) (m-5-3) edge (m-4-4) (m-1-5) edge (m-2-5)
;
\path[->]
(m-2-2) edge (m-2-4) (m-4-2) edge (m-4-4) (m-2-5) edge node[right] {\footnotesize$F_\alpha$} (m-4-5) (m-7-3) edge (m-6-4)
;
\path[->>]
(m-2-4) edge (m-2-5) (m-4-4) edge (m-4-5) (m-6-4) edge (m-6-5) (m-5-1) edge (m-7-1) (m-4-2) edge (m-6-2) (m-6-2) edge (m-7-3) (m-5-3) edge (m-7-3) (m-4-4) edge (m-6-4) (m-4-5) edge (m-6-5) (m-2-2) edge (m-3-3) (m-4-2) edge (m-5-3)
;
\path[-latex]
(m-3-3) edge [-,line width=6pt,draw=white] (m-5-3)
;
\path[>->]
(m-3-3) edge node[right, pos=0.25] {\footnotesize$K^0_\alpha$} (m-5-3)
;
\end{tikzpicture}
\]
Here, $G^0_\alpha$, $G^1_\alpha$ are monic by the description of $\Phi(\mcGP(R))$. Using the Snake Lemma, we get that $K^0_\alpha$ is monic, and so is $K_\alpha$ by the same result. On the other hand, in the short exact sequences $K_1 \rightarrowtail G^1_1 \twoheadrightarrow K^0_1$ and $K_2 / K_1 \rightarrowtail G^1_2 / G^1_1 \twoheadrightarrow K^0_2 / K^0_1$ we know that $G^1_1, G^1_2 / G^1_1 \in \mcGP(R)$ from the description of $\Phi(\mcGP(R))$. Since ${\rm gl.Gdim}(R) = 1$, we get $K_1, K_2 / K_1 \in \mcGP(R)$. Hence, $K \in \Phi(\mcGP(R))$.
\end{example}

The previous examples of strongly Gorenstein categories concern G-admissible triples $(\mathcal{X,Y,Z})$ for which $\omega = \mcP$ and $\nu = \mcI$. The next example exhibits a strongly $(\mathcal{X,Y,Z})$-category structure for which the previous equalities do not necessarily occur.

\begin{example}[strongly Gorenstein categories from tilting-cotilting classes]\label{ex:Gor_st_from_tilting}
Let $\mcC$ be an abelian category with enough projective and injective objects, and $\mcT \subseteq \mcC$ be a tilting-cotilting class in $\mcC$. Consider the G-admissible triple $({}^\perp\mcT,\mcT,\mcT^\perp)$ from Example \ref{ex:cotorsion_triples_are_G-admissible} (3), for which $\mcGP_{({}^\perp\mcT,\mcT)} = {}^\perp\mcT$ and $\mcGI_{(\mcT,\mcT^\perp)} = \mcT^\perp$. One can show that $\mcC$ is strongly $({}^\perp\mcT,\mcT,\mcT^\perp)$-Gorenstein. On the one hand, $\id(\omega) = \id(\mcT) < \infty$ and $\pd(\nu) = \pd(\mcT) < \infty$ by \tiltone \ and its dual. Finally, by \cite[Prop. 3.16 (b)]{AM21} and its dual, we obtain that $\mcGP_{({}^\perp\mcT,\mcT)}^\wedge = ({}^\perp\mcT)^\wedge = \mcC = (\mcT^\perp)^\vee = \mcGI_{(\mcT,\mcT^\perp)}^\vee$.

As a particular instance, let $\Lambda$ be an Artin algebra and $\mathsf{mod}(\Lambda)$ be the category of finitely generated $\Lambda$-modules. Suppose there exists a tilting-cotilting $\Lambda$-module $T$. Then, ${\rm add}(T)$ (the class of direct summands of finite direct sums of copies of $T$) is a tilting-cotilting class in $\mathsf{mod}(\Lambda)$, and $({}^\perp T, {\rm add}(T), T^\perp)$ is its associated G-admissible triple. Applying \cite[Lem. 4.8]{BrLnMeVi} (Bravo, Lanzilotta, Vivero and the second author) one can find another way to show that $\mathsf{mod}(\Lambda)$ is a strongly $({}^\perp T, {\rm add}(T), T^\perp)$-Gorenstein category. Tilting-cotilting finitely generated modules over Artin algebras where characterized by Wei in \cite[Thms. A \& B]{WeiJ}, and the reader can find examples of such modules in \cite[Ex. 4.1]{WeiJ} and Ringel's \cite[Thm. 5]{RingelTilting}.
\end{example}


\subsection*{Model structures on relative Gorenstein categories}

A first approach to the existence of relative Gorenstein model structures can be formulated from part (2) in Corollary \ref{glGPD=id}. More specifically, if $(\mcX,\mcY)$ is a GP-admissible pair in an abelian category $\mcC$ such that $\mcC$ has finite $(\mcX,\mcY)$-Gorenstein projective global dimension and that $\omega$ is closed under direct summands, then we can obtain an abelian model structure on $\mcC$ where $\mcGP_{(\mcX,\mcY)}$ is the class of cofibrant objects, provided that $\omega$ satisfies certain conditions.

\begin{proposition}\label{prop:model_first-approach}
Let $(\mcX,\mcY)$ be a GP-admissible pair in an abelian category $\mcC$ with enough injective objects and $\glGPD_{(\mcX,\mcY)}(\mcC) < \infty$. Suppose $\omega := \mcX \cap \mcY$ is closed under direct summands. Then, the following are equivalent:
\begin{enumerate}
\item[(a)] There exists a unique abelian model structure on $\mcC$ given by the Hovey triple $\mcM = (\mcGP_{(\mcX,\mcY)}, \omega^\wedge, \omega^{\perp_1})$.

\item[(b)] $\omega$ is a special precovering class.
\end{enumerate}
If any of the previous conditions holds, then the homotopy category of $\mcM$ is equivalent to the stable category $(\mcGP_{(\mcX,\mcY)} \cap \omega^{\perp_1}) / \sim$ of bifibrant objects, where two maps $f, g \colon X \to Y$ are homotopic if, and only if, $f-g$ factors through an object in $\omega \cap \omega^{\perp_1}$.
\end{proposition}

\begin{proof} \
\begin{itemize}
\item (a) $\Rightarrow$ (b): By Hovey-Gillespie's correspondence, $(\mcGP_{(\mcX,\mcY)} \cap \omega^\wedge,\omega^{\perp_1}) = (\omega,\omega^{\perp_1})$ is a complete cotorsion pair in $\mcC$, where $\mcGP_{(\mcX,\mcY)} \cap \omega^\wedge = \omega$ follows by Proposition \ref{coro:properties_relative_Gorenstein_projective} (2). In particular, $\omega$ is special precovering.

\item (a) $\Leftarrow$ (b): On the one hand, by Corollary \ref{glGPD=id} we have that $(\mcGP_{(\mcX,\mcY)}, \omega^\wedge)$ is a complete cotorsion pair in $\mcC$. On the other hand, let us show that $(\omega,\omega^{\perp_1})$ is also a complete cotorsion pair. For every object $C \in {}^{\perp_1}(\omega^{\perp_1})$, there is a short exact sequence $K \rightarrowtail W \twoheadrightarrow C$ with $W \in \omega$ and $K \in \omega^{\perp_1}$, since $\omega$ is special precovering. This sequence splits, and hence $C \in \omega$ since $\omega$ is closed under direct summands. The previous shows that $(\omega,\omega^{\perp_1})$ is a cotorsion pair where $\omega$ is special precovering. The existence of special $\omega^{\perp_1}$-preenvelopes follows from the existence of enough injective objects, using the well known \emph{Salce's argument}.

It is only left to show that the pairs $(\mcGP_{(\mcX,\mcY)}, \omega^\wedge)$ and $(\omega,\omega^{\perp_1})$ are compatible. We already know that $\mcGP_{(\mcX,\mcY)} \cap \omega^\wedge = \omega$. The equality $\omega^\wedge = \omega^{\perp_1} \cap \omega^\wedge$, on the other hand, follows inductively and from the fact that $\omega$ is self-orthogonal. The existence of the model structure described in (a) then follows by Hovey's correspondence.
\end{itemize}
The last part is a consequence of Gillespie's \cite[Thm. 2.6 (i)]{GillespieHereditary}.
\end{proof}

\begin{example}\label{ex:models_first_approach} \
Let $\mcC$ be an abelian category with enough projective and injective objects, which in addition is $\mcP$-Gorenstein (that is, ${\rm gl.GPD}(\mcC) < \infty$) in the sense of \cite[Ch. VII. \S 1]{BR}.  Proposition \ref{prop:model_first-approach} applied to the GP-admissible pair $(\mcP,\mcP)$ yields an abelian projective model structure on $\mcC$ such that $\mcGP$ is the class of cofibrant objects and $\mcP^{< \infty}$ is the class of trivial objects. Its homotopy category is equivalent to the stable category of Gorenstein projective objects. This is  slight generalization of \cite[Thm. 8.6]{Hovey}.
\end{example}

\begin{remark}\label{rmk:models_first_approach} \
Condition (b) in Proposition \ref{prop:model_first-approach} may be restricting for certain GP-admissible pairs. For instance in the category $\Qcoh(X)$, if we consider the GP-admissible pair $(\mfF,\mfF \cap \mfC)$ from Example \ref{ex:GP_objects}, then it is not true in general that $\mfF \cap \mfC$ is special precovering (for instance, if $X = {\rm Spec}(R)$ with $R$ a commutative von Neumann regular ring which is not quasi-Frobenius). Theorem \ref{thm:Gorenstein_models} below will provide a more suitable setting to obtain relative Gorenstein model structures.
\end{remark}

The following result is a consequence of the dual of Remark \ref{rmk:GPGI-admissible} (3), Proposition \ref{coro:properties_relative_Gorenstein_projective} (2), Corollary \ref{glGPD=id} (2), and the property $\pd_{\omega^\wedge}(\omega) = \pd_{\omega}(\omega)$.

\begin{lemma}\label{lem:induced_GI}
Let  $(\mathcal{X,Y})$ be a GP-admissible pair in $\mcC$ with $\omega:=\mcX\cap\mcY$ closed under direct summands and $\glGPD_{(\mathcal{X,Y})}(\mcC)<\infty$. Then, $(\omega,\omega^\wedge)$ is a GI-admissible pair in $\mcC$.
\end{lemma}

\begin{lemma}\label{lem:induced_containment}
If $(\mathcal{X,Y,Z})$ is a G-admissible triple such that $\glGPD_{(\mathcal{X,Y})}(\mcC) < \infty$, then $\mcWGI_\omega = \mcGI_{(\omega,\omega^\wedge)} = \omega^\perp$ and $\coresdim_{\omega^\perp}(C) = \id_{\omega}(C)$ for every $C \in \mcC$.
\end{lemma}

\begin{proof}
By Lemma \ref{lem:induced_GI}, $(\omega,\omega^\wedge)$ is a GI-admissible pair. Now the equality $\mcWGI_\omega = \mcGI_{(\omega,\omega^\wedge)} = \omega^\perp$ follows from the duals of \cite[Thm. 3.32]{BMS} and Proposition \ref{coro:properties_relative_Gorenstein_projective}. The last assertion, on the other hand, is a consequence of the dual of Proposition \ref{prop:weak_resdim_vs_pdim}.
\end{proof}

\begin{theorem}\label{thm:Gorenstein_models}
Let $\mcC$ be a strongly $(\mathcal{X,Y,Z})$-Gorenstein category. The following assertions hold true:
\begin{enumerate}
\item If $\nu \subseteq \mcI_{\omega}^{< \infty}$, then there exists a unique hereditary abelian model structure on $\mcC$, where ${}^\perp\mcY$, $\omega^\perp$ and $\mcI^{<\infty}$ are the classes of cofibrant, fibrant and trivial objects, respectively. Its homotopy category is equivalent to the stable category ${\rm Gor}_{\omega} / \sim$. If in addition, $\mcX$ is closed under epikernels, then $\mcI^{<\infty} = \mcX^\wedge \cap \mcI^{<\infty}_{\mcX}$; and $\mathcal{I}^{< \infty} \subseteq \mathcal{P}^{< \infty}_{\mathcal{Y}}$ provided that $\mcX$ is closed under direct summands.

\item If $\omega \subseteq \mcP_{\nu}^{< \infty}$, then there exists a unique hereditary abelian model structure on $\mcC$, where ${}^{\perp}\nu$, $\mcY^\perp$ and $\mcP^{< \infty}$ are the classes of cofibrant, fibrant and trivial objects, respectively. Its homotopy category is equivalent to the stable category ${\rm Gor}_{\nu} / \sim$. If in addition, $\mcZ$ is closed under monocokernels, then $\mcP^{< \infty} = \mcZ^\vee \cap \mcP^{<\infty}_{\mcZ}$; and $\mcP^{< \infty} \subseteq \mathcal{I}^{< \infty}_{\mathcal{Y}}$ provided that $\mcZ$ is closed under direct summands.
\end{enumerate}
\end{theorem}

\begin{proof}
We only prove part (1), as part (2) is dual. By Corollary \ref{MpropSG}, we have a hereditary complete cotorsion pair $(\mcGP_{(\mcX,\mcY)},\omega^\wedge)$ in $\mcC$ where $\mcGP_{(\mcX,\mcY)} = {}^\perp\mcY$. By the dual of Theorem \ref{theo:Xu_model} applied to this pair, there exists a unique exact model structure on $\mcGI^\vee_{(\mcGP_{(\mcX,\mcY)} \cap \omega^\wedge, \omega^\wedge)}$ such that
\[
{}^\perp\mcY \cap \mcGI^\vee_{(\mcGP_{(\mcX,\mcY)} \cap \omega^\wedge, \omega^\wedge)}, \text{ \ } \mcGI_{(\mcGP_{(\mcX,\mcY)} \cap \omega^\wedge, \omega^\wedge)} \text{ \ and \ } (\omega^\wedge)^\vee = \omega^\smile
\]
are the classes of cofibrant, fibrant and trivial objects, respectively, where $\omega^\smile = \mcI^{< \infty}$ by Corollary \ref{coro:essGIC} (1). Now recall that $\mcGP_{(\mcX,\mcY)} \cap \omega^\wedge = \omega$ by Proposition \ref{coro:properties_relative_Gorenstein_projective} (2), and so $\mcGI_{(\mcGP_{(\mcX,\mcY)} \cap \omega^\wedge, \omega^\wedge)} = \mcGI_{(\omega, \omega^\wedge)} = \omega^\perp$ by Lemma \ref{lem:induced_containment}. Let us show that $\mcGI^\vee_{(\omega, \omega^\wedge)} = \mcC$. For every $C \in \mcC = \mcGI^\vee_{(\mcY,\mcZ)}$, we have by the dual of Proposition \ref{coro:properties_relative_Gorenstein_projective} (4) that there exists a short exact sequence $C \rightarrowtail C' \twoheadrightarrow H$ with $H \in \nu^\vee$ and $C' \in \mcGI_{(\mcY,\mcZ)}$. Now since $\nu \subseteq \mcI_{\omega}^{< \infty}$ and $\mcI_{\omega}^{< \infty}$ is thick, we get that $H \in \mcI_{\omega}^{< \infty}$. On the other hand, $C' \in \mcI_{\omega}^{< \infty}$ by Remark \ref{rmk:GtripleIDPD}. Thus, $C \in \mcI_{\omega}^{< \infty}$ and by Lemma \ref{lem:induced_containment} again, we obtain
\[
\Gid_{(\omega,\omega^\wedge)}(C) = \coresdim_{\omega^\perp}(C) = \id_{\omega}(C) < \infty,
\]
that is, $\mcGI^\vee_{(\omega, \omega^\wedge)} = \mcC$. The mentioned exact model structure on $\mcGI^\vee_{(\omega, \omega^\wedge)}$ then becomes an abelian model structure on $\mcC$ given by the Hovey triple $({}^\perp\mcY,\mcI^{< \infty},\omega^\perp)$. By Proposition \ref{prop:homotopy_category}, its homotopy category is equivalent to the stable category $\omega^\perp \cap {}^{\perp}\mcY / \sim$, where two morphisms $f$ and $g$ between the same bifibrant objects are related if, and only if, $f - g$ factors through some object in $\omega$. Then, it suffices to show the equality
\[
\omega^\perp \cap {}^{\perp}\mcY  = {\rm Gor}_{\omega}.
\]
\begin{itemize}
\item ($\supseteq$): Let $N \in {\rm Gor}_{\omega}$. Then, $N \simeq Z_0(W_\bullet)$ for an exact, $\Hom(-,\omega)$-acyclic and $\Hom(\omega,-)$-acyclic complex in $\Ch(\omega)$. It is then clear that $N \in \omega^\perp$ and $N \in \mcGP_{(\mcX,\omega)}$. On the other hand, since $(\mcX,\omega)$ is GP-admissible and $\omega$ is closed under direct summands, we have that $\pd_{\omega^\wedge}(N) = \pd_{\omega}(N) = 0$ by Proposition \ref{coro:properties_relative_Gorenstein_projective} (3). Thus, $N \in {}^{\perp}(\omega^\wedge) = \mcGP_{(\mcX,\mcY)} = {}^\perp\mcY$.

\item ($\subseteq$): Let $N \in \omega^\perp \cap {}^{\perp}\mcY$. Since $\omega$ is self-orthogonal, in order to show that $N \in {\rm Gor}_{\omega}$ it suffices to construct a left $\omega$-totally acyclic complex or a right $\omega$-totally acyclic complex so that $N$ is a cycle of this complex. Indeed, since $N \in \omega^\perp = \mcGI_{(\omega, \omega^\wedge)}$, we have that $N \simeq Z_0(L_\bullet)$ for an exact and $\Hom(\omega,-)$-acyclic complex in $\Ch(\omega^\wedge)$. Consider the exact sequence $Z_{-1}(L_\bullet) \rightarrowtail L_{-1} \twoheadrightarrow N$. Since $L_{-1} \in \omega^\wedge$, there is a short exact sequence $L'_{-1} \rightarrowtail W_{-1} \twoheadrightarrow L_{-1}$ where $W_{-1} \in \omega$ and $L'_{-1} \in \omega^\wedge$. Taking the pullback of $Z_{-1}(L_\bullet) \rightarrowtail L_{-1} \twoheadleftarrow W_{-1}$ yields the following commutative exact diagram:
\[
\begin{tikzpicture}[description/.style={fill=white,inner sep=2pt}]
\matrix (m) [matrix of math nodes, row sep=3em, column sep=3em, text height=1.25ex, text depth=0.25ex]
{
L'_{-1} & L'_{-1} & {} \\
N_{-1} & W_{-1} & N \\
Z_{-1}(L_\bullet) & L_{-1} & N \\
};
\path[->]
(m-2-1)-- node[pos=0.5] {\footnotesize$\mbox{\bf pb}$} (m-3-2)
;
\path[>->]
(m-1-1) edge (m-2-1) (m-1-2) edge (m-2-2)
(m-2-1) edge (m-2-2) (m-3-1) edge (m-3-2)
;
\path[->>]
(m-2-2) edge (m-2-3) (m-3-2) edge (m-3-3)
(m-2-2) edge (m-3-2) (m-2-1) edge (m-3-1)
;
\path[-,font=\scriptsize]
(m-1-1) edge [double, thick, double distance=2pt] (m-1-2)
(m-2-3) edge [double, thick, double distance=2pt] (m-3-3)
;
\end{tikzpicture}
\]
Note that $N, W_{-1} \in {}^\perp\mcY$ implies that $N_{-1} \in {}^\perp\mcY$. On the other hand, since $L'_{-1}, Z_{-1}(L_\bullet) \in \mcGI_{(\omega,\omega^\wedge)}$, we have that $N_{-1} \in \mcGI_{(\omega,\omega^\wedge)}$. Thus, the central row in the previous diagram is $\Hom(-,\omega)$-acyclic since $N \in {}^{\perp}\mcY$, and $\Hom(\omega,-)$-acyclic since $N_{-1} \in \mcGI_{(\omega,\omega^\wedge)}$. Hence, using that $N_{-1} \in \mcGI_{(\omega,\omega^\wedge)} \cap {}^{\perp}\mcY$, we can apply to $N_{-1}$ the previous argument, and repeat it infinitely many times, in order to construct a $\Hom(-,\omega)$-acyclic and $\Hom(\omega,-)$-acyclic $\omega$-resolution of $N$.

We now construct for $N$ a $\Hom(-,\omega)$-acyclic and $\Hom(\omega,-)$-acyclic $\omega$-coresolution. Since $N \in {}^\perp\mcY = \mcGP_{(\mcX,\mcY)}$, by Proposition \ref{coro:properties_relative_Gorenstein_projective} (4) we have a short exact sequence $N \rightarrowtail W_0 \twoheadrightarrow N'$ with $W_0 \in \omega$ and $N' \in \mathcal{GP}_{(\mcX,\mcY)} = {}^\perp\mcY$. Note also that $N' \in \omega^\perp$, since $N, W_0 \in \omega^\perp$. So the previous sequence is $\Hom(\omega,-)$-acyclic and $\Hom(-,\omega)$-acyclic. It then follows that it is possible to construct inductively a $\Hom(-,\omega)$-acyclic and $\Hom(\omega,-)$-acyclic $\omega$-coresolution of $N$.
\end{itemize}

For the last part, suppose $\mcX$ is closed under epikernels. Since $(\mcX,\mcY)$ is GP-admissible, we have that $\mcX$ is closed under extensions and that $\omega$ is a relative cogenerator in $\mcX$ with $\id_{\mcX}(\omega) = 0$. Moreover, $\omega$ is closed under direct summands. It then follows by \cite[Prop. 4.2]{AB}\footnote{This result also holds in the setting of abelian categories. See Santiago's \cite[Prop. 2.27]{ValenteThesis}} that $\omega^\smile = \mcX^\wedge \cap \mcI^{<\infty}_{\mcX}$. Finally, in the case where $\mcX$ is closed under direct summands, we have by Proposition \ref{prop:resdim_vs_pdim} that $\mcX^\wedge \subseteq \mcP^{<\infty}_{\mcY}$.
\end{proof}

\begin{example} \
\begin{enumerate}
\item Let $(\mathcal{X,Y,Z})$ be a G-admissible triple in $\mathcal{C}$, where $\mathcal{C}$ has enough projective and injective objects, such that $\omega = \mathcal{P}$ and $\nu = \mathcal{I}$.
\begin{itemize}
\item In the case where $(\mathcal{X,Y,Z})$ is a hereditary complete cotorsion triple with $\mathcal{C} = \mathcal{X}^\wedge = \mathcal{Z}^\vee$, we know from Example \ref{ex:condition_triples} that $\mathcal{C}$ is strongly $(\mathcal{X,Y,Z})$-Gorenstein. By Theorem \ref{thm:Gorenstein_models}, we have that there exists a unique projective and a unique injective abelian model structure on $\mathcal{C}$ given by the Hovey triples $(\mathcal{X},\mathcal{I}^{< \infty},\mathcal{C})$ and $(\mathcal{C},\mathcal{I}^{<\infty},\mathcal{Z})$, respectively. Moreover, their homotopy categories are naturally equivalent to the stable categories $\mathcal{GP} / \sim$ and $\mathcal{GI} / \sim$, since ${\rm Gor}_\omega = \mathcal{GP}$ and ${\rm Gor}_{\nu} = \mathcal{GI}$.

\item If we consider the Dold triple $({\rm dg}\widetilde{\mathcal{P}(R)},\mathcal{E},{\rm dg}\widetilde{\mathcal{I}(R)})$ from Example \ref{ex:global_dim_rel_Gor} and assume that ${\rm gl.dim}(R) < \infty$, we have that $\Ch(R)$ is a strongly $({\rm dg}\widetilde{\mathcal{P}(R)},\mathcal{E},{\rm dg}\widetilde{\mathcal{I}(R)})$-Gorenstein category. In this case, one can note that $\mathcal{I}^{< \infty} = \mathcal{P}^{<\infty}$ coincides with the class $\mathcal{E}$ of exact chain complexes. So the resulting model structures from the triples $({\rm dg}\widetilde{\mathcal{P}(R)},\mathcal{E},\Ch(R))$ and $(\Ch(R),\mathcal{E},{\rm dg}\widetilde{\mathcal{I}(R)})$ are the ones described in \cite[\S 2.3]{HoveyBook}, whose homotopy categories are naturally equivalent to the derived category of the ground ring $R$.

\item If $\mathcal{F}$ is an abelian subcategory of $\mathcal{C}$, we know from Example \ref{ex:FrobeniusGorenstein} that $\mcF$ is strongly $(\mathcal{F,PI,F})$-Gorenstein. In this case, one can note that both Hovey triples obtained from Theorem \ref{thm:Gorenstein_models} coincide with $(\mathcal{F},\mathcal{PI},\mathcal{F})$. The corresponding homotopy category is equivalent to the stable category $\mathcal{F} / \sim$.

\item Let $Q \colon 1 \xrightarrow{\alpha} 2$ be the $A_2$ quiver and $R$ is an Iwanaga-Gorenstein ring with ${\rm gl.Gdim}(R) = 1$ and ${\rm gl.dim}(R) = \infty$. From Example \ref{ex:induced_on_quivers}, we know that ${\rm Rep}(Q,\Mod(R))$ is a strongly category relative to the (hereditary complete) triple $(\Phi(\mcGP(R)),{\rm Rep}(Q,\mcP(R)^{< \infty}),\Psi(\mcGI(R)))$, where $\mathcal{P}^{<\infty} = {\rm Rep}(Q,\mcP(R)^{< \infty}) = {\rm Rep}(Q,\mcI(R)^{< \infty}) = \mathcal{I}^{< \infty}$. In this case, we have two Hovey triples $(\Phi(\mcGP(R)),\mathcal{I}^{< \infty},{\rm Rep}(Q,\Mod(R)))$ and $({\rm Rep}(Q,\Mod(R)),\mathcal{I}^{<\infty},\Psi(\mcGI(R)))$.
\end{itemize}

\item Let $\mcC$ be an abelian category with enough projective and injective objects, and $\mcT \subseteq \mcC$ be a tilting-cotilting class in $\mcC$. We known from Example \ref{ex:Gor_st_from_tilting} that $\mcC$ is strongly $({}^\perp\mcT,\mcT,\mcT^\perp)$-Gorenstein. Moreover, $\omega = \nu = \mcT$. Since $\mcT$ is self-orthogonal, we have that $\mcT \subseteq \mathcal{I}^{< \infty}_{\mcT}$ and $\mcT \subseteq \mathcal{P}^{< \infty}_{\mcT}$. So by Theorem \ref{thm:Gorenstein_models} we have the Hovey triples $({}^\perp\mcT,\mathcal{I}^{< \infty},\mcT^\perp)$ and $({}^\perp\mcT,\mathcal{P}^{< \infty},\mcT^\perp)$, having both homotopy categories naturally equivalent to ${\rm Gor}_{\mcT} / \sim$.
\end{enumerate}
\end{example}


\section{Tilting theory in relative Gorenstein categories}\label{sec:tilting}

The purpose of this section is to link relative Gorenstein categories with two different notions of tilting. Namely, the first one deals with (co)tilting pairs as introduced in \cite[Defs. 5.3 \& 5.5]{BMS}; and the second one deals with the notion of (co)tilting related to subcategories of an abelian category $\mcC$ as introduced in  \cite[Def. 3.1]{AM21}.

\begin{definition}
We say that a G-admissible triple $(\mcX,\mcY,\mcZ)$ in $\mcC$, with $\omega = \mathcal{X} \cap \mathcal{Y}$ and $\nu = \mathcal{Y} \cap \mathcal{Z}$, is \textbf{tilting} is the following conditions hold true.
\begin{enumerate}
\item[\tiltingone] For any $C\in {}^\perp\mcY$ there is a monic $\mcY$-preenvelope $C\to Y$ with $Y\in \omega$.

\item[\tiltingtwo] For any $C\in \mcY^\perp$ there is an epic $\mcY$-precovering $Y\to C$ with $Y\in \nu$.
\end{enumerate}
\end{definition}

The terminology introduced in the above definition can be explained by using the notions of tilting and cotilting pairs of classes introduced in \cite[Defs. 5.3 \& 5.5]{BMS}. We recall that these notions are generalizations of the tilting and cotilting modules introduced by Angeleri H\"ugel and Coehlo in \cite{AHC01}.

\begin{remark}\label{RkT}
Let $(\mcX,\mcY,\mcZ)$ be a $G$-admissible triple in $\mcC$. By \cite[Thm. 3.32 \& Coroll. 5.7]{BMS} and their duals, one has that
\begin{align*}
\mathsf{(t1)} & \Longleftrightarrow \text{$(\omega,\mcY)$ is $\omega$-cotilting} \Longleftrightarrow \mcGP_{(\mcX,\mcY)} = {}^\perp\mcY, \text{and} \\
\mathsf{(t2)} & \Longleftrightarrow \text{$(\mcY,\nu)$ is $\nu$-tilting} \Longleftrightarrow \mcGI_{(\mcY,\mcZ)} = \mcY^\perp.
\end{align*}
Then, if $\mcC$ is a strongly  $(\mcX,\mcY,\mcZ)$-Gorenstein category, the previous equivalences along with Corollary \ref{MpropSG} and its dual, imply that the triple $(\mcX,\mcY,\mcZ)$ is tilting.
\end{remark}

\begin{proposition}\label{MpTGT}
For every tilting $G$-admissible triple $(\mcX,\mcY,\mcZ)$ in $\mcC,$ the following statements hold true:
\begin{enumerate}
\item $\mcGP_{(\mcX,\mcY)}^\wedge = \mcP^{<\infty}_{\mcY}$ and $\mcGI_{(\mcY,\mcZ)}^\vee = \mcI^{<\infty}_{\mcY}$.

\item  $\pd_\omega(\mcGP_{(\mcX,\mcY)}^\wedge) = \FPD_\mcY(\mcC) = \FGPD_{(\mcX,\mcY)}(\mcC)$.

\item $\id_\nu(\mcGI_{(\mcY,\mcZ)}^\vee) = \FID_\mcY(\mcC) = \FGID_{(\mcY,\mcZ)}(\mcC)$.

\item ${}^\perp\mcY = {}^\perp\omega\cap\mcGP_{(\mcX,\mcY)}^\wedge$.

\item $\mcY^\perp = \nu^\perp\cap\mcGI_{(\mcY,\mcZ)}^\vee$.
\end{enumerate}
\end{proposition}

\begin{proof}
Part (1) follows from Remark \ref{RkT}, Proposition \ref{prop:weak_resdim_vs_pdim} and its dual. Part (2) can be shown from (1) and Corollary \ref{glGPD=id} (1), while (4) follows from Remark \ref{RkT} and \cite[Coroll. 4.15 (b3)]{BMS}. Finally, (3) and (5) are dual to (2) and (4), respectively.
\end{proof}

From Remark \ref{RkT}, we obtain the following extension of Theorem \ref{MThmGc}.

\begin{theorem}
For a tilting $G$-admissible triple $(\mcX,\mcY,\mcZ)$ in $\mcC$, consider the following:
\begin{enumerate}
\item[(a)] $\mcP^{<\infty}_{\mcY} = \mcI^{<\infty}_{\mcY}$, $\FPD_\mcY(\mcC) < \infty$ and $\FID_\mcY(\mcC) < \infty$.

\item[(b)] $\mcC$ is kernel restricted $(\mcX,\mcY,\mcZ)$-Gorenstein, $\id_\nu(\omega)<\infty$, $\omega \subseteq \mcGI_{(\mcY,\mcZ)}^\vee$ and $\nu\subseteq \mcGP_{(\mcX,\mcY)}^\wedge$.

\item[(c)] $\glGPD_{(\mcX,\mcY)}(\mcC) < \infty$ and $\glGID_{(\mcY,\mcZ)}(\mcC) < \infty$ (that is, $\mathcal{C}$ is strongly Gorenstein).
\end{enumerate}
Then, the implications (b) $\Leftarrow$ (c) $\Rightarrow$ (a) hold, and (a) implies that $\mcGP_{(\mcX,\mcY)}^\wedge = \mcGI_{(\mcY,\mcZ)}^\vee$.

Suppose in addition that $\mathcal{C}$ is AB4 and AB4${}^\ast$, and $(\mcX,\mcY,\mcZ)$ satisfies ($\mathsf{ga5}$) and ($\mathsf{ga6}$). Then:
\begin{enumerate}
\item (b) $\Leftarrow$ (a) $\Rightarrow$ (c), and the equalities
\begin{align*}
\FGID_{(\mcY,\mcZ)}(\mcC) & = \Gid_{(\mathcal{Y,Z})}(\mathcal{GP}_{(\mathcal{X,Y})}^\wedge) = \Gid_{(\mathcal{Y,Z})}(\mathcal{GP}_{(\mathcal{X,Y})}) \\
& = \Gid_{(\mathcal{Y,Z})}(\omega) = \id_{\nu}(\omega) = \id_{\nu}(\mathcal{GP}_{(\mathcal{X,Y})}) \\
& = \id_{\nu}(\mathcal{GP}_{(\mathcal{X,Y})}^\wedge) = \id_{\mathcal{Y}}(\mathcal{GP}_{(\mathcal{X,Y})}^\wedge) = \id_{\mathcal{Y}}(\mathcal{GP}_{(\mathcal{X,Y})}) \\
& = \pd_{\mathcal{Y}}(\mathcal{GI}_{(\mathcal{Y,Z})}) = \pd_{\mathcal{Y}}(\mathcal{GI}_{(\mathcal{Y,Z})}^\vee) = \pd_{\omega}(\mathcal{GI}_{(\mathcal{Y,Z})}^\vee) \\
& = \pd_{\omega}(\mathcal{GI}_{(\mathcal{Y,Z})}) = \pd_{\omega}(\nu) = \Gpd_{(\mathcal{X,Y})}(\nu) \\
& = \Gpd_{(\mathcal{X,Y})}(\mathcal{GI}_{(\mathcal{Y,Z})}) = \Gpd_{(\mathcal{X,Y})}(\mathcal{GI}_{(\mathcal{Y,Z})}^\vee) \\
& = \FGPD_{(\mathcal{X,Y})}(\mathcal{C}) < \infty
\end{align*}
hold. In particular, (a) and (c) are equivalent.

\item (b) $\Rightarrow$ (a) provided that $\mcGP_{(\mathcal{X,Y})} \subseteq \mathcal{I}^{< \infty}_{\mathcal{Z}}$ and $\mcGI_{(\mathcal{Y,Z})} \subseteq \mathcal{P}^{< \infty}_{\mathcal{X}}$. In particular, (a), (b) and (c) are equivalent.
\end{enumerate}
\end{theorem}

\begin{proof}
If we assume (c), then Theorem \ref{MThmGc} and Proposition \ref{MpTGT} yield (a) and (b). On the other hand, if we assume (a), then the equality $\mcGP_{(\mcX,\mcY)}^\wedge = \mcGI_{(\mcY,\mcZ)}^\vee$ follows by Proposition \ref{MpTGT} (1).
\begin{enumerate}
\item Under these conditions, we have $\mcGP_{(\mcX,\mcY)}^\wedge = \mcGI_{(\mcY,\mcZ)}^\vee$. Then, the chain of equalities is a consequence of \cite[Thm. 3.1]{HMP22}. Using this chain, we can note that $\omega \cup \nu \subseteq \mcGP_{(\mcX,\mcY)}^\wedge = \mcGI_{(\mcY,\mcZ)}^\vee$, $\id_\nu(\omega) < \infty$, $\id_\nu(\nu^\vee) \leq \id_\nu(\mcGP_{(\mcX,\mcY)}^\wedge) < \infty$ and $\pd_\omega(\omega^\wedge) \leq \pd_\omega(\mcGI_{(\mcY,\mcZ)}^\vee) < \infty$, that is, we have (b). It then follows that (c) is obtained from Theorem \ref{MThmGc}.

\item By Lemma \ref{lemcGP=cGI} and its dual, we get $\mcGP_{(\mcX,\mcY)}^\wedge = \mcGI_{(\mcY,\mcZ)}^\vee$. Thus, Proposition \ref{MpTGT} (1) implies $\mcP^{<\infty}_{\mcY} = \mcI^{<\infty}_{\mcY}$. Finally, using \cite[Thm. 3.1]{HMP22} and Proposition \ref{MpTGT} (2,3), we conclude that $\FPD_\mcY(\mcC) = \FID_\mcY(\mcC) = \id_\nu(\omega) < \infty$.
\end{enumerate}
\end{proof}

For the rest of this section, we shall need the notion of $n$-$\mcX$-tilting classes in an abelian category $\mcC$ (see Example \ref{ex:GP_objects} (4)). An $n$-$\mcX$-tilting class $\mcT$ is \emph{big} (resp. \emph{small}) if $\mcT$ is closed under arbitrary (resp. finite) coproducts. An object $T \in \mcC$ is \emph{big} (resp. \emph{small}) \emph{$n$-$\mcX$-tilting} if ${\rm Add}(T)$ (resp. ${\rm add}(T)$) is $n$-$\mcX$-tilting. In the case $\mcC = \mcX$, we just say \emph{$n$-tilting} instead of $n$-$\mcC$-tilting.

\begin{proposition}\label{tilting-in-triples}
Let $(\mcX,\mcY,\mcZ)$ be a tilting $G$-admissible triple in $\mcC$, and assume that either $\mcC$ has enough injective objects or that $\mcZ \subseteq \mcGI_{(\mcY,\mcZ)}^\perp$. Then, the following assertions hold for $\nu := \mcY \cap \mcZ$:
\begin{enumerate}
\item $\nu$ is a small $0$-$\mcGI_{(\mcY,\mcZ)}$-tilting category.

\item $\nu$ is a small $n$-$\mcGI_{(\mcY,\mcZ)}^\vee$-tilting category if $n := \FID_\mcY(\mcC) < \infty$.

\item $\nu = {}^\perp(\nu^\perp) \cap \nu^\perp \cap \mcGI_{(\mcY,\mcZ)} = {}^\perp(\nu^\perp) \cap \mcY^\perp$.
\end{enumerate}
\end{proposition}

\begin{proof} \
\begin{enumerate}
\item First, $\nu$ is closed under direct summands by \gatwo, and the containment $\nu \subseteq \mcGI_{(\mcY,\mcZ)} \subseteq \mcY^\perp \subseteq \nu^\perp$ is clear. Thus, we get conditions \tiltzero, \tiltone \ and \tilttwo. Condition \tiltthree \ follows by the dual of \cite[Coroll. 3.25 (a)]{BMS} and the trivial containment $\nu \subseteq \nu^\vee_{\mcGI_{(\mcY,\mcZ)}}$. Condition \tiltfive \ is trivial since $\nu^\perp \cap \mcGI_{(\mcY,\mcZ)} = \mcGI_{(\mcY,\mcZ)}$. Finally, for \tiltfour \ note first that $\nu^\perp \cap \mcGI_{(\mcY,\mcZ)}^\perp = \mcGI_{(\mcY,\mcZ)}^\perp$. Under the assumption that $\mcC$ has enough injective objects, it is clear that the condition is fulfilled by setting $\beta := \mcI$, by Remark \ref{RkT}. Now in the case where $\mcZ \subseteq \mcGI_{(\mcY,\mcZ)}^\perp$, we have that $\mcZ \subseteq \nu^\perp \cap \mcGI_{(\mcY,\mcZ)}^\perp$ and $\mcZ$ is a relative cogenerator in $\mcGI_{(\mcY,\mcZ)}$.

\item Condition \tiltzero \ was proved in the previous part, and \tiltone \ follows by Proposition \ref{MpTGT} (3).  Since $\nu \cap \mcGI_{(\mcY,\mcZ)}^\vee = \nu \subseteq \nu^\perp$, \tilttwo \ is clear. Now by the dual of Proposition \ref{coro:properties_relative_Gorenstein_projective} (1), we know that $\mcGI_{(\mcY,\mcZ)}^\vee$ is thick. Then, $\nu^\vee_{\mcGI_{(\mcY,\mcZ)}^\vee} = \nu^\vee$ since $\nu \subseteq \mcGI_{(\mcY,\mcZ)}$. On the other hand, the dual of Proposition \ref{coro:properties_relative_Gorenstein_projective} (4) implies that $\nu^\vee$ is a relative generator in $\mcGI_{(\mcY,\mcZ)}^\vee$, that is, \tiltthree. Condition \tiltfive \ follows, since $\nu^\perp \cap \mcGI_{(\mcY,\mcZ)}^\vee = \mcY^\perp$ by Proposition \ref{MpTGT} (5). Finally, in order to verify \tiltfour, first note that $(\mcGI_{(\mcY,\mcZ)}^\vee)^\perp = \mcGI_{(\mcY,\mcZ)}^\perp$ by the dual of \cite[Rmk. 3.13]{BMS}. So in the case where $\mcZ \subseteq \mcGI_{(\mcY,\mcZ)}^\perp$, it is enough to show that $\mcZ$ is a relative cogenerator in $\mcGI_{(\mcY,\mcZ)}^\vee$. Given $M \in \mcGI_{(\mcY,\mcZ)}^\vee$, by the dual of Proposition \ref{coro:properties_relative_Gorenstein_projective} (4) again, there is an exact sequence $M \rightarrowtail I_0 \twoheadrightarrow V$ with $I_0 \in \mcGI_{(\mcY,\mcZ)}$ and $V \in \nu^\vee$. Moreover, using that $\mcZ$ is a relative cogenerator in $\mcGI_{(\mcY,\mcZ)}$, there is an exact sequence $I_0 \rightarrowtail Z_0 \twoheadrightarrow I_1$ with $Z_0 \in \mcZ$ and $I_1 \in \mcGI_{(\mcY,\mcZ)}$. Thus, the composition $M \rightarrowtail I_0 \rightarrowtail Z_0$ has cokernel in $\mcGI_{(\mcY,\mcZ)}^\vee$ by Snake's lemma. The case where $\mcC$ has enough injective objects follows as in part (1).

\item By (1), we know that $\nu$ is $0$-$\mcGI_{(\mcY,\mcZ)}$-tilting. Hence, the result is a direct consequence of \cite[Prop. 3.19 (b)]{AM21} along with Remark \ref{RkT} $\mathsf{(t2)}$.
\end{enumerate}
\end{proof}

In what follows, given a positive integer $k \geq 1$ and two classes of objects $\mcT, \mcX \subseteq \mcC$, let ${\rm Fac}^{\mcX}_k(\mcT)$ denote the class of objects $C \in \mcC$ admitting an exact sequence
\[
K \rightarrowtail T_k \xrightarrow{f_k} \cdots \to T_1 \stackrel{f_1}\twoheadrightarrow C
\]
with ${\rm Ker}(f_i) \in \mcX$ and $T_i \in \mcT \cap \mcX$ for every $1 \leq i \leq k$ (see \cite[Def. 5.1]{AM21P1}). In the case where $\mcX = \mcC$, we simply write ${\rm Fac}_k(\mcT)$.

A pair $(\mcA,\mcB)$ of classes of objects in $\mcC$ is \emph{$\mcX$-complete} if for any $X \in \mcX$ there are exact sequences $B \rightarrowtail A \twoheadrightarrow X$ and $X \rightarrowtail B' \twoheadrightarrow A'$ where $A, A' \in \mcA \cap \mcX$ and $B, B' \in \mcB \cap \mcX$.

\begin{theorem}\label{teo:tilting-in-triples}
Let $(\mcX,\mcY,\mcZ)$ be a tilting $G$-admissible triple in $\mcC$ for which $n := \FID_\mcY(\mcC)$ is finite. If either $\mcC$ has enough injective objects or $\mcZ \subseteq \mcGI_{(\mcY,\mcZ)}^\perp$, then the following assertions hold true:
\begin{enumerate}
\item ${}^\perp(\nu^\perp) \cap \mcGI_{(\mcY,\mcZ)}^\vee = \nu^\vee$.

\item $\mcGI_{(\mcY,\mcZ)} = \nu^\perp \cap \mcGI_{(\mcY,\mcZ)}^\vee = {\rm Fac}_k(\nu)\cap\mcGI_{(\mcY,\mcZ)}^\vee$ for every $k \geq \max\{1, n\}$.

\item The pair $({}^\perp(\nu^\perp),\nu^\perp)$ is $\mcGI_{(\mcY,\mcZ)}^\vee$-complete.

\item $\Gid_{(\mcY,\mcZ)}(\nu^\vee) = \FGID_{(\mcY,\mcZ)}(\mcC) = \pd_{\nu^\vee}(\nu^\vee) = \coresdim_\nu(\nu^\vee) = n$.
\end{enumerate}
\end{theorem}

\begin{proof}
By Proposition \ref{tilting-in-triples} (2), we know that $\nu$ is $n$-$\mcGI_{(\mcY,\mcZ)}^\vee$-tilting. Then by \cite[Thm. 3.12 (a)]{AM21} and the fact that $\mcGI_{(\mcY,\mcZ)}^\vee$ is thick, we obtain
\[
{}^\perp(\nu^\perp) \cap \mcGI_{(\mcY,\mcZ)}^\vee = \nu^\vee_{\mcGI_{(\mcY,\mcZ)}^\vee} \cap \mcGI_{(\mcY,\mcZ)}^\vee = \nu^\vee,
\]
and thus (1) follows. Part (2), on the other hand, follows by Remark \ref{RkT}, Proposition \ref{MpTGT} (5) and \cite[Thm. 3.12 (b)]{AM21}. Part (3) follows from \cite[Thm. 3.12 (c)]{AM21}. Finally, (4) follows from (1), (2), Proposition \ref{MpTGT} (3) and \cite[Prop. 3.18 (a)]{AM21}.
\end{proof}

\begin{corollary}\label{coro:tilting-in-triples}
Let $(\mcX,\mcY,\mcZ)$ be a tilting $G$-admissible triple in $\mcC$ such that $\mcC$ has enough injective objects and $n := \glGID_{(\mcY,\mcZ)}(\mcC)$ is finite. Then, the following statements hold true:
\begin{enumerate}
\item $\nu$ is a small $n$-tilting category.

\item $\mathcal{Y}^\perp = \mcGI_{(\mcY,\mcZ)} = \nu^\perp = {\rm Fac}_k(\nu)$ for every $k \geq \max\{1, n\}$.

\item ${}^\perp(\nu^\perp) = \nu^\vee$ and $\nu = {}^\perp(\nu^\perp) \cap \nu^\perp$.

\item $\FID_{\mathcal{Y}}(\mathcal{C}) = \pd(\nu) = \Gid_{(\mcY,\mcZ)}(\nu^\vee) = \pd_{\nu^\vee}(\nu^\vee) = \coresdim_\nu(\nu^\vee) = n$.

\item Suppose in addition that $\mcC$ has enough projective objects, and that $\mcP = {\rm add}(P)$ for some $P \in \mcP$. If ${\rm add}(V)$ is precovering in $V^\perp$ for any $V \in \nu$, then there is a small $n$-tilting object $T \in \nu$ such that $\nu = {\rm add}(T)$.

\item Suppose in addition that $\mcC$ is AB4 and has enough projective objects, and that $\mcP = {\rm Add}(P)$ for some $P \in \mcP$. If $\nu = {\rm Add}(\nu),$ then there is a big $n$-tilting object $T \in \nu$ such that $\nu = {\rm Add}(T)$.

\item $\mcWGP_\nu = \mcGP_{(\nu^\vee,\nu)} = {}^\perp\nu$ and $\id(\nu) = \glGPD_{(\nu^\vee,\nu)}(\mathcal{C})$.
\end{enumerate}
\end{corollary}

\begin{proof}
Part (1) follows from Propositions \ref{MpTGT} (3) and \ref{tilting-in-triples} (2). Part (2) can be obtained from Theorem \ref{teo:tilting-in-triples} (2). For (3), we use Theorem \ref{teo:tilting-in-triples} (1) to obtain the equality ${}^\perp(\nu^\perp) = \nu^\vee$. On the other hand, an induction argument shows that $\nu = \nu^\vee_k \cap \nu^\perp$ for every $k \geq 0$, and thus ${}^\perp(\nu^\perp) \cap \nu^\perp = \nu^\vee \cap \nu^\perp = \nu$ follows. Part (4) is a consequence of Proposition \ref{MpTGT} (3) and Theorem \ref{teo:tilting-in-triples} (4).

Let us now show (5). By the dual of Corollary \ref{glGPD=id} (2) we know that $(\nu^\vee,\mcGI_{(\mcY,\mcZ)})$ is a hereditary complete cotorsion pair. This along with part (3) yields
\[
\pd({}^\perp\mcGI_{(\mcY,\mcZ)}) = \pd(\nu^\vee) = \pd(\nu) = n
\]
and
\[
\mcGI_{(\mcY,\mcZ)} \cap {}^\perp\mcGI_{(\mcY,\mcZ)} = \mcGI_{(\mcY,\mcZ)} \cap \nu^\vee = \nu.
\]
It then follows by \cite[Coroll. 3.50]{AM21} that there is some small $n$-tilting object $T \in \mcC$ such that $\mcGI_{(\mcY,\mcZ)} = T^\perp$. Hence, by parts (2) and (3), along with \cite[Prop. 3.19 (b)]{AM21}, we get that
\[
\nu = {}^\perp(\nu^\perp) \cap \nu^\perp = {}^\perp(T^\perp) \cap T^\perp = {\rm add}(T).
\]
Notice that (6) is similar to (5) by using \cite[Coroll. 3.49]{AM21} instead. Finally, (7) follows from the dual of Lemma \ref{lem:induced_containment}.
\end{proof}

\begin{corollary}\label{coro2:tilting-in-triples}
Let $\Lambda$ be an Artin $R$-algebra such that $\modu(\Lambda)$ is a strongly $(\mcX,\mcY,\mcZ)$-Gorenstein category, $n := \glGID_{(\mcY,\mcZ)}(\modu(\Lambda))$ and $m := \glGPD_{(\mathcal{X,Y})}(\mathcal{C})$. Then, there exist a basic $n$-tilting $\Lambda$-module $T$ and a basic $m$-cotilting $\Lambda$-module $U$ satisfying the following properties:
\begin{enumerate}
\item $\nu = \add(T)$ and $\omega = \add(U)$.

\item $\mcGI_{(\mcY,\mcZ)} = \mcWGI_\nu = T^\perp$ and $\mcGP_{(\mcX,\mcY)} = \mcWGP_\omega = {}^\perp U$.

\item $\pd(T) = \FID_{\mathcal{Y}}(\mathcal{C}) = \coresdim_{\nu}(\nu^\vee) = n$; and \\ $\id(U) = \FPD_{\mathcal{Y}}(\mathcal{C}) = \resdim_{\omega}(\omega^\wedge) = m$.

\item $T$ is cotilting if $\id(T) < \infty$.

\item $U$ is tilting if $\pd(U) < \infty$.
\end{enumerate}
\end{corollary}

\begin{proof}
For the Artin $R$-algebra, we have the usual duality (contravariant) functor $D := \Hom_R(-,\kappa) \colon \modu(\Lambda) \to \modu(\Lambda^{\rm op})$, where $\kappa$ is the injective envelope of $R/rad(R)$. In particular, we have that $\modu(\Lambda)$ is an abelian category with enough projective and injective objects, and $\mcI = \add(D(\Lambda^{\rm op}))$. Moreover, it is also true that $\add(M)$ is functorially finite for any $M \in \modu(\Lambda)$, and that the triple $(\mcX,\mcY,\mcZ)$ is tilting by Remark \ref{RkT}. Thus, by Corollary \ref{coro:tilting-in-triples} (2,4, 5) and its dual, we get (1), (3) and the equalities  $\mcGI_{(\mcY,\mcZ)} = T^\perp$ and  $\mcGP_{(\mcX,\mcY)} = {}^\perp U$. Notice also that, from Auslander and Reiten's \cite[Thm. 5.4 (b)]{AR91} and its dual, we obtain the equalities $\mcWGP_\omega = {}^\perp U$ and $\mcWGI_\nu = T^\perp$, proving (2).

Let us now show (4). By Corollary \ref{coro:tilting-in-triples} (7), we have that $\mcWGP_\nu = {}^\perp\nu$ and $\resdim_{{}^\perp\nu}(\modu(\Lambda)) = \id(\nu) = \id(T)$. Since $\id(T) < \infty$, we get from \cite[Thm. 5.4 (a)]{AR91} that $\nu = \add(T')$ for some basic cotilting $\Lambda$-module $T'$. It is well known that $T$ and $T'$ have the same number of non isomorphic indecomposable direct summands, and hence $T\simeq T'$, proving that $T$ is cotilting. Finally, part (5) is dual.
\end{proof}

Let us close this section showing a bijective correspondence between tilting-cotilting classes and certain equivalence classes of triples of classes of objects in $\mathcal{C}$ which induce on $\mathcal{C}$ a relative Gorenstein category structure. Recall from Examples \ref{ex:cotorsion_triples_are_G-admissible} (3) and \ref{ex:Gor_st_from_tilting} that if $\mathcal{C}$ is an abelian category with enough projective and injective objects, and $\mcT \subseteq \mcC$ is a tilting-cotilting class in $\mcC$ (that is, $\mcT$ is $n$-tilting and $m$-cotilting for a pair of nonnegative integer $n, m \geq 0$), then $({}^\perp\mcT,\mcT,\mcT^\perp)$ is a G-admissible triple and $\mathcal{S}$ is a strongly $({}^\perp\mcT,\mcT,\mcT^\perp)$-Gorenstein category. Note also that $({}^\perp\mcT,\mcT,\mcT^\perp)$ is tilting. We now aim to characterize all of the strongly relative Gorenstein structures of this sort. To that aim, consider the following classes
\begin{align*}
\text{tcSG}(\mathcal{C}) & := \{ \text{tilting G-admissible triples} \ (\mathcal{X,Y,Z}) \ \text{in} \ \mathcal{C} \ \text{such that} \ \mathcal{C} \ \text{is strongly} \\
& \hspace{0.7cm} (\mathcal{X,Y,Z})\text{-Gorenstein with} \ \mathcal{Y} \subseteq \mathcal{X} \cap \mathcal{Z} \}, \\
\text{tc}(\mathcal{C}) & := \{ \mathcal{Y} \subseteq \mathcal{C} \ \text{:} \ \mathcal{Y} \ \text{is $n$-tilting and $m$-cotilting for some $m,n \geq 0$} \},
\end{align*}
along with the following relation: for $(\mathcal{X,Y,Z}), (\mathcal{X}',\mathcal{Y}',\mathcal{Z}') \in \text{tcSG}(\mathcal{C})$,
\[
(\mathcal{X,Y,Z}) \sim (\mathcal{X}',\mathcal{Y}',\mathcal{Z}')
\]
if:
\begin{enumerate}
\item $\mathcal{Y} = \mathcal{Y}'$,
\item $\mathcal{GP}_{(\mathcal{X,Y})} = \mathcal{GP}_{(\mathcal{X}',\mathcal{Y}')}$, and
\item $\mathcal{GI}_{(\mathcal{Y,Z})} = \mathcal{GI}_{(\mathcal{Y}',\mathcal{Z}')}$.
\end{enumerate}
It is clear that $\sim$ defines an equivalence relation on $\text{tcSG}(\mathcal{C})$. The equivalence class of a triple $(\mathcal{X,Y,Z}) \in \text{tcSG}(\mathcal{C})$ will be denoted by $[\mathcal{X,Y,Z}]$. We have the following result.

\begin{theorem}
Let $\mathcal{C}$ be an abelian category with enough projective and injective objects. Consider the map $\varphi \colon {\rm tc}(\mathcal{C}) \to {\rm tcSG}(\mathcal{C}) / \sim$ given by
\[
\mathcal{Y} \mapsto [{}^\perp\mathcal{Y},\mathcal{Y},\mathcal{Y}^\perp]
\]
for every $\mathcal{Y} \in {\rm tc}(\mathcal{C})$. Then, $\varphi$ is bijective, with inverse $\phi \colon {\rm tcSG}(\mathcal{C}) / \sim \to {\rm tc}(\mathcal{C})$ given by
\[
[\mathcal{X,Y,Z}] \mapsto \mathcal{Y},
\]
for every $[\mathcal{X,Y,Z}] \in {\rm tcSG}(\mathcal{C}) / \sim$.
\end{theorem}

\begin{proof}
By previous comments, we know that $[{}^\perp\mathcal{Y},\mathcal{Y},\mathcal{Y}^\perp] \in {\rm tcSG}(\mathcal{C}) / \sim$ for every $\mathcal{Y} \in {\rm tc}(\mathcal{C})$. On the other hand, let $[\mathcal{X,Y,Z}] \in {\rm tcSG}(\mathcal{C}) / \sim$. Then,
\[
\omega := \mathcal{X} \cap \mathcal{Y} = \mathcal{Y} = \mathcal{Y} \cap \mathcal{Z} =: \nu,
\]
and so by Corollary \ref{coro:tilting-in-triples} (1) and its dual, we have that $\mathcal{Y}$ is tilting-cotilting, and the map $\phi$ is well defined. One can easily check that $\phi \circ \varphi = {\rm id}_{{\rm tc}(\mathcal{C})}$. Finally, let us show that $\varphi \circ \phi = {\rm id}_{{\rm tcSG}(\mathcal{C}) / \sim}$. Let $[\mathcal{X,Y,Z}] \in {\rm tcSG}(\mathcal{C}) / \sim$. We already know that $\mathcal{Y}$ is tilting-cotilting. Moreover, from Example \ref{ex:cotorsion_triples_are_G-admissible} (3) and Remark \ref{RkT}, we have that $\mathcal{GP}_{({}^\perp\mathcal{Y},\mathcal{Y})} = {}^\perp\mathcal{Y} = \mathcal{GP}_{(\mathcal{X,Y})}$ and $\mathcal{GI}_{(\mathcal{Y},\mathcal{Y}^\perp)} = \mathcal{Y}^\perp = \mathcal{GI}_{(\mathcal{Y,Z})}$. Hence, $(\mathcal{X,Y,Z}) \sim ({}^\perp\mathcal{Y},\mathcal{Y},\mathcal{Y}^\perp)$.
\end{proof}


\subsection*{Acknowledgements}

Parts of this work were carried out during academic visits at the Instituto de Matem\'aticas (Universidad Nacional Aut\'onoma de M\'exico) and Instituto de Matem\'atica y Estad\'istica ``Prof. Ing. Rafael Laguardia'' (Universidad de la Rep\'ublica). We want to thank the IMATE-UNAM and the IMERL-UdelaR staffs for their hospitality. We also thank professor Jun Peng Wang for pointing out the existence of \cite{WangDi19}, in which the equality between the global Ding projective and the global Ding injective dimensions is shown.


\subsection*{Funding}

The first named author was supported by Grant PID2020-113206GB-I00 funded
by MICIU/AEI/10.13039/501100011033 and by Grant 22004/PI/22 funded by Fundaci\'on S\'eneca-Agencia de Ciencia y Tecnolog\'ia de la Regi\'on de Murcia.   The second named author was supported by the Programa de Apoyo a Proyectos de Investigaci\'on e Innovaci\'on Tecnol\'ogica (PAPIIT) IN100124. The third named author was supported by ANII - Agencia Nacional de Investigaci\'on e Innovaci\'on and PEDECIBA - Programa de Desarrollo de las Ciencias B\'asicas.


\bibliographystyle{plain}
\bibliography{biblioRelGor}

\end{document}